\documentclass[12pt,a4paper]{article}
\usepackage[utf8]{inputenc}

\usepackage{amsfonts,amsmath, amssymb}
\usepackage{amsthm}
\usepackage{graphicx}
\usepackage{longtable, rotating}
\usepackage{float, mathtools}
\usepackage{array}
\usepackage{enumerate}
\usepackage{threeparttable}
\usepackage{mathrsfs}
\usepackage[text={15cm,24cm},top=2cm,margin=2cm]{geometry} 
\usepackage{listings}
\usepackage{url}
\usepackage{subfigure}
\usepackage{multirow}
\usepackage{textcomp,booktabs}
\usepackage[usenames,dvipsnames]{color}
\usepackage{colortbl}
\usepackage{lscape}
\usepackage{bm}
\usepackage[numbers,square]{natbib}
\usepackage[colorlinks=true,urlcolor=black,citecolor=blue,linkcolor=blue,bookmarks=true]{hyperref} 
\usepackage{verbatim}
\usepackage{thmtools}
\usepackage{thm-restate}
\usepackage{cleveref}
\usepackage{ulem}
\usepackage{tikz}
\usetikzlibrary{positioning,arrows}
\tikzset{
	block/.style={
		draw, 
		rectangle, 
		minimum height=1cm, 
		minimum width=2cm, align=center
	}, 
	line/.style={->,>=latex'}
}

\definecolor{mygray}{gray}{.9}
\usepackage[labelfont={bf,it},textfont={bf,it}]{caption}
\DeclareGraphicsExtensions{.jpg,.pdf,.gif,.png}

\newtheorem{lemma}{Lemma}
\newtheorem{remark}{Remark}

\newcommand{\R}{\mathbb{R}}

\DeclareMathOperator*{\argmin}{argmin}

\numberwithin{equation}{section}
\numberwithin{definition}{section}
\numberwithin{remark}{section}
\numberwithin{theorem}{section}
\numberwithin{proposition}{section}
\numberwithin{lemma}{section}
\numberwithin{remark}{section}
\numberwithin{example}{section}
\numberwithin{figure}{section}
\numberwithin{conjecture}{section}
\numberwithin{table}{section}

\begin{document}
	\pagestyle{plain}
	\title{\bf Quality of approximating a mass-emitting object by a point source in a diffusion model}
	\author{Qiyao Peng$^{1*}$, Sander C. Hille$^{1}$\\
	{\footnotesize 1. Mathematical Institute, Faculty of Science, Leiden University.} \\
	{\footnotesize	Neils Bohrweg 1, 2333 CA, Leiden, The Netherlands}\\
   {\footnotesize * Correspondence: q.peng@math.leidenuniv.nl}}
	\maketitle
	
	\begin{abstract}
		For the sake of computational efficiency and for theoretical purposes, in mathematical modelling, the Dirac Delta distributions are often utilized as a replacement for cells or vesicles, since the size of cells or vesicles is much smaller than the size of the surrounding tissues. Here, we consider the scenario that the cell or the vesicle releases the diffusive compounds to the immediate environment, which is modelled by the diffusion equation. Typically, one separates the intracellular and extracellular environment and uses homogeneous Neumann boundary condition for the cell boundary (so-called spatial exclusion approach), while the point source approach neglects the intracellular environment. We show that extra conditions are needed such that the solutions to the two approaches are consistent. We prove a necessary and sufficient condition for the consistency. Suggested by the numerical results, we conclude that an initial condition in the form of Gaussian kernel in the point source approach compensates for a time-delay discrepancy between the solutions to the two approaches in the numerical solutions. Various approaches determining optimal amplitude and variance of the Gaussian kernel have been discussed. 
		
	\end{abstract}

	\section{Introduction}\label{Sec_intro}
	\noindent
	In mathematical modelling, it is a common practice to replace a spatial object with a negligible size by a point-particle, which has no volume. In classical mechanics, objects may be replaced by point masses \citep{Westphal_1968}. In electrostatics, point charge \citep{jackson1999classical} is a theoretical concept, used e.g. to describe the electric field that results from a spatial distribution of charge. In biological modelling, for instance, chemotaxis \citep{Stock_2009, Wadhams_2004} or in wound healing, when one deals with a large-scale wound that is in the order of millimeter or even centimeter, individual cells are regarded as particles and point sources that secrete signalling molecules, which then diffuse in the surrounding environment \citep{Peng2020,Peng2021}. 
	
	These are all idealisations that are convenient not only for the modeller and theorist. For computational efficiency of model simulation it may be even necessary to employ point-particles. In particular when the objects are moving, interacting or have internal dynamics of their own, operating at several different scales. For example, simulation of several moving objects that secrete a diffusing compound in a finite element approach can easily become quite cumbersome \citep{Jiao_2019,Secomb_2007}. An implementation in terms of several moving and `mass-emitting' point-particles may then be better tractable \citep{Painter_2003}, in particular when the number of objects becomes large -- but not so large that a continuum description with densities or concentrations is a proper representation.
	
	What representation is proper, is determined by the underlying research question for the modelling effort. It yields a tolerance for the deviation from observations, which in turn yields a tolerance for deviations between solutions of different modelling approaches. Moreover, in view of this question and the implied tolerance it can be more or less relevant that a point source cannot represent the spatial heterogeneity in shape or in flux density of compound over the boundary of a truly spatial object. 
	
	Motivated by biological applications, in this paper we are concerned with the mathematical question of the assessment of the quality of approximation when a model for diffusion in the environment of a compound that is secreted by several stationary spatial objects is replaced by that with point sources at central locations of these objects. We do so by numerical simulations of two-dimensional spatial configurations, in which the objects are equal, of circular shape, with homogeneous constant flux density over their boundary (see Figure \ref{Fig_intro}\subref{fig_domain} for the schematic set-up, for a single object). As such, it can be viewed as a follow-up to an analytic approach in \citet{HMEvers2015} (for a single object) and a numerical approach in the setting of wound healing dynamics in \citet{Peng2022JCAM,Peng2022MATCOM}. Heterogeneity in shape and flux density, and movement of objects will be considered in follow-up research. If mass-emitting point sources cannot approximate circular stationary objects with homogeneous flux density over their boundary within tolerance, then this cannot be expected either in the heterogeneous or in the non-stationary setting. 
	
	\begin{figure}[h!]
		\centering
		\begin{minipage}{0.48\linewidth}
			\subfigure[A schematic presentation of a single object ($\Omega_C$) within the domain $\Omega$. $\Omega\setminus\Omega_C$ is the environment into which the object emits a compound that freely diffuses there. It cannot escape through the boundary $\partial\Omega$ of the domain. A typical direction of the normal vector $\mathbf{n}$ on $\partial\Omega_c$ is indicated]{
				\includegraphics[width=0.9\textwidth]{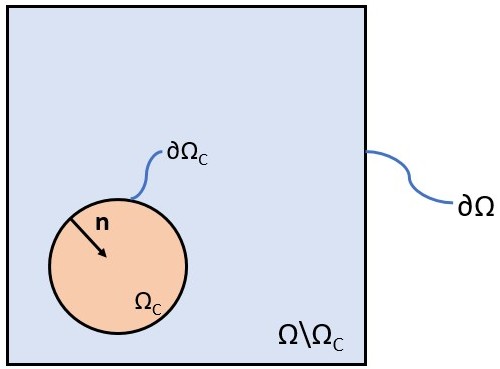}
				\label{fig_domain}}
		\end{minipage}
		\begin{minipage}{0.48\linewidth}
			\subfigure[The concentration of compound in the environment of the spatial exclusion approach at the last time step.]{
				\includegraphics[width=\textwidth]{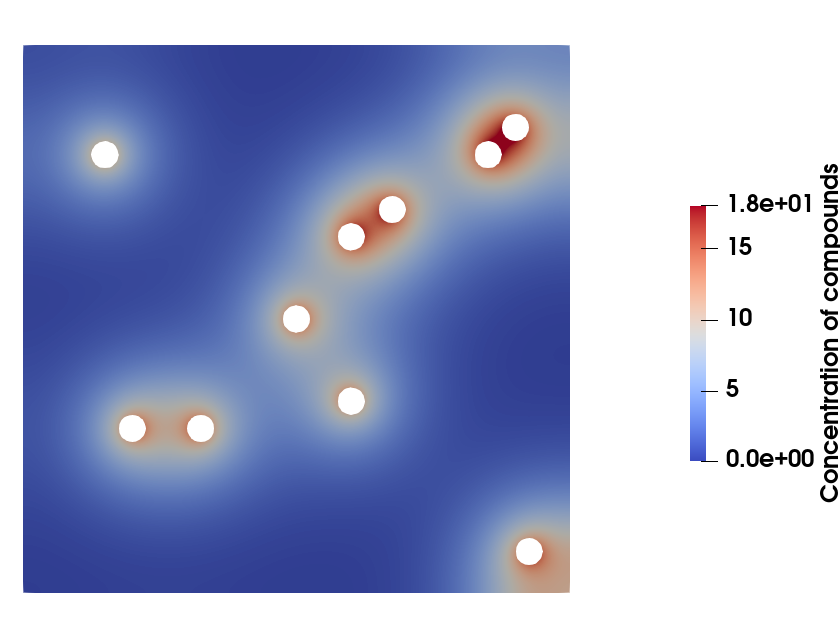}
				\label{fig_hole_0_1000}}
			\\
			\subfigure[The concentration of compound of the point source approach at the last time step.]{
				\includegraphics[width=\textwidth]{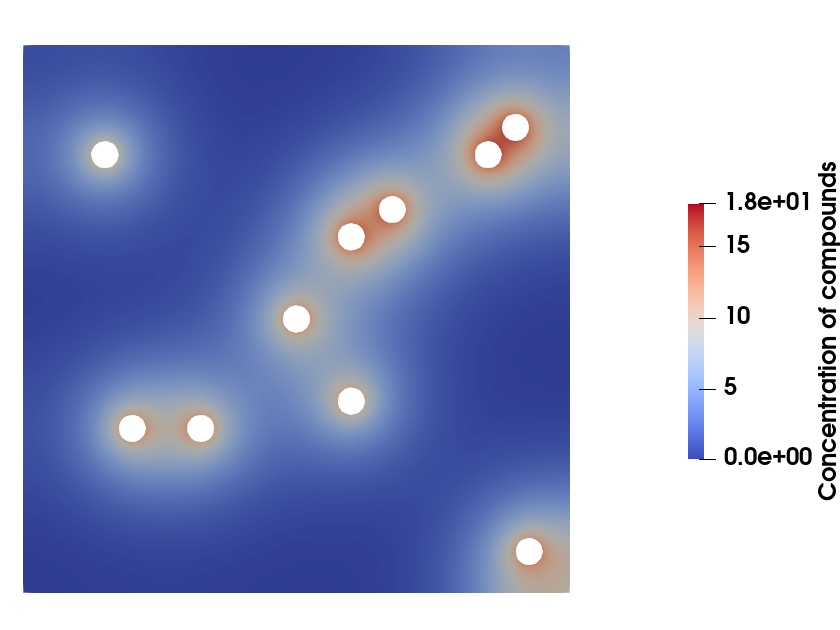}
				\label{fig_ps_ad_0_1000}}
		\end{minipage}
		\caption{Schematic presentation (not in scale) of the spatial set-up in the models, with circularly shaped stationary objects within a bounded domain (Panel (a)).   When there are multiple cells in the computational domain, differences between the simulation results of the model with spatial objects (Panel (b)) and that with mass-emitting point sources (Panel (c)) can be observed. By proper choice of parameters and initial condition in the point-source approach these can be kept within tolerance (see Section \ref{Sec_number_of_holes} for simulation details). Here, the zero homogeneous initial condition is utilized, which is pointing into the cell $\Omega_C$.} 
		\label{Fig_intro}
	\end{figure}
	
	The paper aims at exhibiting in which cases and how -- practically -- a proper representation can be made by means of mass-emitting point sources of the spatial object setting. We propose a quantification of deviations between the two approaches, find intolerable differences that occur and identify their sources, and we propose how these can be resolved, e.g. concerning the extension of initial conditions to a larger spatial domain (see Section \ref{Sec_Results_Inhomo} and see Figure \ref{Fig_intro}\subref{fig_hole_0_1000} and \subref{fig_ps_ad_0_1000} for a visual comparison). We present a (mainly numerical) analysis of consequences of particular choices for doing so on the quality of approximation. We do not claim that the proposed choice of extension by means of a Gaussian-shaped function is the best. It is a reasonable and intuitively motivated choice for which we give an approach for finding parameter settings for the extension that may keep the approximation within accepted tolerance.
	
	\subsection{Mathematical formulation of the research questions}
	
	We consider a bounded domain $\Omega\subset \R^2$ with piece-wise $C^1$-boundary $\partial\Omega$ in which there are embedded a finite number $N_c$ of non-overlapping spatial objects, called {\it cells}, which are considered as disjoint subdomains $\Omega_{C_i}$, ($i=1,\dots, N_c$), also with piece-wise $C^1$-boundaries $\partial\Omega_{C_i}$, such that these boundaries do not touch, nor touch the boundary $\partial\Omega$ of the initial domain. Write $\Omega_C:=\bigcup_{i=1}^{N_c}\Omega_{C_i}$ for the totality of cells. The complement $\Omega\setminus\overline{\Omega}_C$ will be called the {\it extracellular environment} of the cells. For each cell we select a point $\boldsymbol{x}_c^i\in\Omega_{C_i}$, which will function as center for representing cell $C_i$ by a point-particle located at that point; see Figure \ref{Fig_intro}\subref{fig_domain} for a schematic presentation of this set-up.
	
	If $\boldsymbol{v}=(v_1,v_2)\in\R^2$, we write $|\boldsymbol{v}|:=(|v_1|^2+|v_2|^2)^{1/2}$ for its Euclidean norm.
	\vskip 0.2cm
	
	\subsubsection{Spatial exclusion model}
	\label{subsec:spatial exclusion model}
	The cells secrete a compound into the environment over their boundary with prescribed flux density $\phi(\boldsymbol{x},t)\geq 0$ at $\boldsymbol{x}\in\partial\Omega_C$ and time $t\geq 0$. This compound diffuses in this environment according to Fickian diffusion with homogeneous diffusion constant $D$, without further interaction. It cannot escape the domain $\Omega$. Initially, there is a distribution $u_0(\boldsymbol{x})$ of this compound in the environment. Thus, the density $u_S(\boldsymbol{x},t)$ at time $t$ of the compound in the environment is described by the initial boundary value problem
	\begin{equation}
		\label{Eq_BVP_hole}
		(BVP_S)\quad\left\{
		\begin{aligned}
			\frac{\partial u_S(\boldsymbol{x},t)}{\partial t} - D\Delta u_S(\boldsymbol{x},t) &= 0, &\mbox{in $\Omega\backslash\bar{\Omega}_C, t>0$,}\\
			D\nabla u_S(\boldsymbol{x},t)\cdot\boldsymbol{n} &= \phi(\boldsymbol{x},t), &\mbox{on $\partial\Omega_C, t>0$,}\\
			D\nabla u_S(\boldsymbol{x},t)\cdot\boldsymbol{n} &= 0, &\mbox{on $\partial\Omega, t>0$,}\\
			u_S(\boldsymbol{x}, 0) &= u_0(\boldsymbol{x}), &\mbox{in $\Omega\backslash\bar{\Omega}_C$,}
		\end{aligned}
		\right.
	\end{equation}
	where $\boldsymbol{n}$ is the outward pointing unit normal vector to the domain boundary of $\Omega\setminus\bar{\Omega}_C$. Note that the flux density $\phi(\boldsymbol{x},t)$ is positive at $\boldsymbol{x}\in\partial\Omega_C$ where there is flux of compound into the environment $\Omega\setminus\bar{\Omega}_C$, as it is shown in Figure \ref{fig_domain}.
	
	As appropriate in the setting of the Finite Element Method (FEM) (cf. \cite{van2005numerical}) we consider the weak solution concept for $(BVP_S)$ in the spatial dimension, which is given by 
	\begin{equation*}
		(WF_S)\left\{
		\begin{aligned}
			&\text{Find $u_S(\boldsymbol{x},t)\in H^1(\Omega\backslash\bar{\Omega}_C)$, such that}\\
			&\int_{\Omega\backslash\bar{\Omega}_C}\frac{\partial u_S(\boldsymbol{x},t)}{\partial t}v_1(\boldsymbol{x},t)d\Omega + \int_{\Omega\backslash\bar{\Omega}_C} D\nabla u_S(\boldsymbol{x},t)\cdot\nabla v_1(\boldsymbol{x},t) d\Omega\\
			&- \int_{\partial\Omega_C}\phi(\boldsymbol{x},t) v_1(\boldsymbol{x},t)d\Gamma=0,\\
			&\text{for any $v_1(\boldsymbol{x},t)\in H^1(\Omega\backslash\bar{\Omega}_C)$.}
		\end{aligned}
		\right.
	\end{equation*}
	Here $d\Omega$ is the restriction of Lebesgue measure on $\R^2$ to $\Omega$ and $d\Gamma$ denotes the surface measure on $\partial\Omega_C$, that is so normalized that the Divergence Theorem holds without additional constant. 
	
	\subsubsection{Point source model}
	We want to compare the (weak) solution $u_S$ to $(BVP_S)$ with the solution to a suitable boundary value problem with point sources at the locations $\boldsymbol{x}_c^i$ instead of spatial cells. These sources will be expressed using Dirac measures $\delta_{\boldsymbol{x}_c^i}$ at $\boldsymbol{x}_c^i$, or -- equivalently -- in the form of the Schwartzian delta distribution $\delta$ (see e.g. \cite{Schwartz1951}), which is defined in any dimension by
	\[
	\langle \delta,f\rangle = f(\boldsymbol{0}),\qquad\mbox{if}\ f\in C^\infty_c(\R^n).
	\]
	Informally written, as often done, the Dirac measure at $\boldsymbol{x}_0$ can then be viewed as translation of the delta distribution:
	\[
	\int_\Omega f(\boldsymbol{x}) d\delta_{\boldsymbol{x}_0}(\boldsymbol{x})\ =\ 
	\int_\Omega \delta(\boldsymbol{x}- \boldsymbol{x}_0) f(\boldsymbol{x}) d\Omega = f(\boldsymbol{x}_0).
	\]
	The initial-boundary value problem defined by point sources is then given by
	\begin{equation}
		\label{Eq_BVP_dirac}
		(BVP_P)\quad\left\{
		\begin{aligned}
			\frac{\partial u_P(\boldsymbol{x},t)}{\partial t} - D\Delta u_P(\boldsymbol{x},t) &= \sum_{i = 1}^{N_c}\Phi_i(t)\delta(\boldsymbol{x}-\boldsymbol{x}^i_c), &\mbox{in $\Omega, t>0$,}\\
			D\nabla u_P\cdot\boldsymbol{n} &= 0, &\mbox{on $\partial\Omega, t>0$,}\\
			u_P(\boldsymbol{x}, 0) &= \bar{u}_0(\boldsymbol{x}), &\mbox{in $\Omega, t=0$.}
		\end{aligned}
		\right.
	\end{equation}
	Here, $\Phi_i(t)$ is a function that describes the flux of mass per unit time from the source at $\boldsymbol{x}_c^i$. We shall take
	\begin{equation}\label{eq:choice Phi}
		\Phi_i(t) = \int_{\partial\Omega_{C_i}}\phi(\boldsymbol{x}, t)d\Gamma.
	\end{equation}
	See Section \ref{subsec:defining influx function} for further discussion of the selection of a suitable flux function $\Phi_i$.
	
	Again, we consider weak solutions in the setting of FEM. This amounts to the following weak form -- formulated for a single cell for convenience (dropping indices $i$ that distinguish cells):
	\begin{equation*}
		(WF_P)\left\{
		\begin{aligned}
			&\text{Find $u_P(\boldsymbol{x},t)\in H^1(\Omega)$, such that}\\
			&\int_{\Omega\backslash\bar{\Omega}_C}\frac{\partial u_P(\boldsymbol{x},t)}{\partial t}v_2(\boldsymbol{x},t)d\Omega + \int_{\Omega_C}\frac{\partial u_P(\boldsymbol{x},t)}{\partial t}v_2(\boldsymbol{x},t)d\Omega +  \int_{\Omega\backslash\bar{\Omega}_C} D\nabla u_P(\boldsymbol{x},t)\nabla v_2(\boldsymbol{x},t) d\Omega \\
			&+ \int_{\Omega_C} D\nabla u_P(\boldsymbol{x},t)\nabla v_2(\boldsymbol{x},t) d\Omega = \int_{\Omega}\Phi(t)\delta(\boldsymbol{x} - \boldsymbol{x}_c)v_2(\boldsymbol{x},t)d\Omega,\\
			&\text{for any $v_2(\boldsymbol{x},t)\in H^1(\Omega)$.}
		\end{aligned}
		\right.
	\end{equation*}
	
	The singular nature of the delta distribution causes functional analytic issues. Although weak solutions to $(WF_F)$ exist and are unique (see \citet{HMEvers2015}), there does not exists a stationary solution in $H^1(\Omega)$. The solution to the elliptic boundary value problem from the balance of momentum with delta distribution is singular in the sense that for spatial dimension higher than one, no formal solutions in the finite-element space $H^1$ exist. Dealing with this singularity caused by delta distributions was discussed in \citet{Peng2022JCAM,Peng2022MATCOM}. 
	
	\subsubsection{Measures for quantitative comparison}
	\label{subsec:measures for comparison}
	
	Quantifying the differences between the solutions to the spatial exclusion model and the point source model in a meaningful way is not fully straightforward. Various choices can be made. Hence, a first question is what measure for comparison is meaningful in the context of applications. Here we shall motivate those that we selected for use in this paper and one resulting from \citet{HMEvers2015}.
	
	First note that the two solutions `live' on different spatial domains: $u_P$ is defined on $\Omega$, while $u_S$ is defined on the environment, the subset $\Omega\setminus\bar{\Omega}_C$, only. Since there is no canonical way of extending $u_S$ to the larger set $\Omega$, an objective comparison of the two solutions is possible only on the smaller set $\Omega\setminus\bar{\Omega}_C$. In accordance with this reasoning, the part of solution $u_P$ on $\bar{\Omega}_C$ has not been shown in Figure \ref{Fig_intro}\subref{fig_ps_ad_0_1000}.
	
	The following proposition substantiates the intuition that any difference in the two solutions is caused by a difference of flux over the boundary $\partial\Omega_C$ of the two solutions, provided that their initial conditions are the same on the environment of the cells.
	
	\begin{restatable}{proposition}{prop}\label{Prop_condition}
		Denote by $u_S(\boldsymbol{x},t)$ and $u_P(\boldsymbol{x},t)$ the weak solutions to the spatial exclusion model $(BVP_S)$ and the point source model $(BVP_P)$, respectively, and let $\partial\Omega_C$ be the boundary of the cells, from which the compounds are released, with normal vector $\boldsymbol{n}$ pointing into $\Omega_C$. Then
		\begin{align}
			\frac{1}{2}\frac{d}{dt} \bigl\| u_S-u_P \bigr\|^2_{L^2(\Omega\setminus\Omega_C)}\ &=\ -D \int_{\Omega\setminus\Omega_C} \bigl| \nabla(u_S -u_P) \bigr|^2 d\Omega \label{eq:L2-norm differnce}\\
			&\qquad\qquad +\ \int_{\partial\Omega_C} (u_s-u_P)(\phi-D\nabla u_P\cdot \boldsymbol{n})\,d\Gamma.  \nonumber
		\end{align}
		Assume moreover, that $u_S(\cdot,0)=u_P(\cdot,0)$ a.e. on $\Omega\setminus\Omega_C$. Then, $u_S(\boldsymbol{x},t) = u_P(\boldsymbol{x},t)$ a.e. in $\Omega\backslash\bar{\Omega}_C\times [0,\infty)$ if and only if 
		$$\phi(\boldsymbol{x},t) - D\nabla u_P(\boldsymbol{x},t)\cdot\boldsymbol{n}= 0, \qquad\mbox{ a.e. on $\partial\Omega_C\times[0,\infty)$.}$$
	\end{restatable}
	\begin{proof}
		See Appendix \ref{Sec_proposition_pf}.
	\end{proof}
	\begin{remark}
		The proposition can be analogously extended to any spatial dimension.
	\end{remark}
	
	Therefore, the difference in total flux over the boundary, summed up to time $t$, i.e.
	\[
	\int_0^t \int_{\Omega_C} \phi(\boldsymbol{x},s) - D\nabla u_P(\boldsymbol{x},s)\cdot\boldsymbol{n}\; d\Gamma ds,
	\]
	yields a physically interpretable quantity of the deviation between the two solutions. It is the difference in total amount of compound in the environment between the spatial exclusion and point source approach, caused by the difference in flux over the cell boundary $\partial\Omega_C$. Following the definition of \citet{HMEvers2015}, we take the related quantity 
	\begin{equation}
		\label{Eq_c_star}
		c^*(t) := \int_0^t\|\phi(\boldsymbol{x},s) - D\nabla u_P(\boldsymbol{x},s)\cdot \boldsymbol{n}\|^2_{L^2(\partial\Omega_C)}ds.
	\end{equation}
	as a measure of comparison. For technical reasons, we chose to work with $L^2$-norm rather than the physically readily interpretable $L^1$-norm on $\partial\Omega_C$. The former is easier accessible through the Finite Element Method (FEM). Moreover, 
	\citet{HMEvers2015} gives various theoretical estimates for $c^*(t)$, for a single cell and point source. Of course, $\|f\|_{L^1(X,\mu)}\leq \mu(X)^{1/2} \|f\|_{L^2(X,\mu)}$.
	
	The $L^1$-norm difference on the environment $\Omega\setminus\bar{\Omega_C}$ compares the total amount of compound between the two solutions. The $L^1$-norm difference of the gradients yields information on differences in local fluxes that occur. For both we again prefer to use the (related) $L^2$-norms, because of FEM that is used in the numerical analysis.
	
	We shall be looking for extension of the initial condition in the spatial exclusion model to the cells $\Omega_C$, such that one arrives at an initial condition on $\Omega$ for the point source approach that yields a good approximation of the solution of the spatial exclusion model. To assess this quality it is necessary to use a relative measure of comparison. That is, to quantise a deviation in comparison to the total amount of compound in the environment of the cells. A comparison of the effect of different initial conditions with an absolute measure cannot be easily made, since different conditions tend to give different amounts of compound in the environment.

	\subsubsection{Determining influx for the point source approach} 
	\label{subsec:defining influx function}
	
	When replacing a cell $C_i$ with mass flux density $\phi_i(\boldsymbol{x},t)$ over its boundary $\partial\Omega_{C_i}$ in the direction of the environment by a point source at $\boldsymbol{x}_c^i$ with mass flux $\Phi_i(t)$, one has to decide how to relate the latter to the former. In Equation \eqref{eq:choice Phi} we made the choice that the total mass emitted by a cell at time $t$ per unit time is equal to that emitted by the point source that replaces this cell. Thus, the total mass emitted by the point source up to time $t$ is kept equal to the total mass emitted by the cell into the environment. Since mass needs on average a time of the order $\mathrm{diam}(C_i)^2/4D$ to travel from the center $\boldsymbol{x}_c^i$ to the boundary $\partial\Omega_{C_i}$, there will be a time lag between the solution $u_S$ and $u_P$ on $\partial\bar{\Omega}_C$. 
	
	Other choices for $\Phi_i(t)$ could be made. However, the definition in Equation \eqref{eq:choice Phi} seems most natural, particularly if one considers the homogeneous flux. Moreover, the time lag between the solutions $u_S$ and $u_P$ may be partially overcome by appropriately choosing the initial condition $\bar{u}_0$ on $\Omega$ of solution $u_P$ in relation to the initial condition $u_0$ on $\Omega\setminus\bar{\Omega}_C$ for $u_S$. In this paper we shall focus on the latter means for minimizing the difference between the two solutions.
	
	\subsubsection{Extension of initial condition}
	
	The major question addressed in this research is, in what way one can best replace the initial condition $u_0$ for the spatial exclusion model by an initial condition $\bar{u}_0$, defined on the whole domain $\Omega$, such that the solutions $u_S$ and $u_P$ are `optimally close'. Here we fix the flux relation as in Equation \eqref{eq:choice Phi}, as discussed in the previous section.
	
	We consider $\bar{u}_0$ as an extension of $u_0$ to the larger domain. The question is then, what function profile to take on $\Omega_{C}$, to make the extension and how smooth does $u_0$ and this profile connects at the boundary $\partial \Omega_C$?
	
	In view of Proposition \ref{Prop_condition} and the further discussion in Section \ref{subsec:measures for comparison} it seems reasonable to require that the flux over the boundary $\partial\Omega_{C_i}$ in the point source model, created by the combined effect of the selected initial condition profile in $\Omega_{C_i}$ and the mass flux $\Phi_i(t)$ from the point source at $\boldsymbol{x}_c^i$ are as close to $\phi_i(\boldsymbol{x},t)$ as possible over all time for which the solution is computed. Again, various quantifiers for this difference may be selected.
	
	\subsection{Major assumptions}
	
	In this paper we limit our attention to a specific setting of the questions raised above. First of all, we consider the general question of the quality of approximating by a point source only in two spatial dimensions. This is mainly to reduce computational intensity of the simulations. Moreover, the preceding theoretical work \cite{HMEvers2015} also considered a two-dimensional spatial domain.
	
	The spatial objects are all taken circular in shape, with equal radius $r$. A point source will emit compounds in a symmetric manner into its environment in case of isotropic diffusion, which we consider here. One therefore knows {\it a priori}, that replacing a non-circular object by a point source will not only create error because of the reduction of the spatial object to a point, but also because a non-symmetric object (for rotation) cannot be expected to produce a symmetric emission profile, typically. The appropriate approximation of a generally shaped object by one that is circular is another question, that would be best considered separately.
	
	The boundary $\partial\Omega_{C_i}$ of the circular cell $C_i$ of radius $r$, centred at $\boldsymbol{x}_c^i$ is parameterized by the angle $\theta\in[0,2\pi)$ relative to a reference direction. The parameterisation $\gamma_i(\theta)$ is such that the measure $d\Gamma$ on $\partial\Omega_{C_i}$, which is -- recall -- normalized such that the Divergence Theorem holds without additional constants (see Section \ref{subsec:spatial exclusion model}), is given by
	\begin{equation}
		\int_{\partial\Omega_{C_i}} f\,d\Gamma = r\int_0^{2\pi} f\bigl(\gamma_i(\theta)\bigr)\,d\theta, \qquad \mbox{for}\ f\in C(\partial\Omega_{C_i}).
	\end{equation}
	
	The simplest case of emission is one that is constant in time and spatially homogeneous over the circular object. We consider that case here. That is,
	\begin{equation}
		\phi(\boldsymbol{x},t) = \phi_0 > 0\qquad \mbox{for all}\ \boldsymbol{x}\in\partial\Omega_C,\ t\geq 0.
	\end{equation}
	Spatially non-homogeneous flux over a circular boundary will allow to approximate flux emitting from a non-circular object.  Thus, this is certainly an important aspect of the approximation question to consider. This setting is currently being investigated. Results on this more general setting will appear elsewhere.
	
	We expect that the diffusion constant will influence the quality of approximation. Therefore, we shall vary $D$ around a central value $D_0$ by one order of magnitude.
	
	\subsection{Structure of the paper}
	
	The manuscript is structured as follows: The preparations for numerical analysis and the approach taken are presented in Section \ref{Sec_sol_concept}. Furthermore, the results of  numerical simulations in Sections \ref{Sec_Results_Somo} and \ref{Sec_Results_Inhomo}. The first considers solutions with zero initial condition on the environment $\Omega\setminus\Omega_C$, with non-zero extension to $\Omega$. The second considers non-zero -- but constant -- initial condition on $\Omega\setminus\Omega_C$ and suitable extension of this initial condition to $\Omega$. Finally, conclusions and the outlook of this work are discussed in Section \ref{Sec_Conclu}.

	\section{Preparation for Numerical Analysis}\label{Sec_sol_concept}
	
	For the numerical analysis we shall use non-dimensional versions of the models $({BVP}_S)$ and $(BVP_P)$, in particular, their weak formulations. 
	
	\subsection{Non-dimensional models}
	\label{subsec:non-dimensional models}
	\noindent 
	Denote the circular cell region centered at $\boldsymbol{x}_c$ and radius $r$ by $\Omega_C:=\mathbb{B}(\boldsymbol{x}_c, r)$. The entire domain is $\Omega:= [-L,L]\times[-L,L]$. In the dimensionless model and simulations thereof we scale space such that cell diameter becomes $1$. Thus, we get spatial variables $\xi$ and computational domains given by
	\begin{equation*}
		\boldsymbol{\xi} := \frac{\boldsymbol{x}}{2r},\ \hat{\Omega}_C: = \mathbb{B}\bigl(\boldsymbol{\xi}_c, \displaystyle\frac{1}{2}\bigr),\ \hat{\Omega} := \left[-\frac{L}{2r},\ \frac{L}{2r}\right]\times\left[-\frac{L}{2r},\ \frac{L}{2r}\right].
	\end{equation*}
	Time is scaled by $\tau_0$ such that dimensionless time and the diffusion constant become
	\begin{equation*}
		\tau:= \frac{t}{\tau_0},\quad \hat{D}: = \frac{D\tau_0}{4r^2}.
	\end{equation*}
	Here, $\tau_0$ is chosen such that $\hat{D}=1$ corresponds to $D=D_0$, the central value, and the relevant range for varying $D$ becomes $\hat{D}\in[0.1,10]$. At last, we will consider $\phi(\boldsymbol{x},t)=\phi_0$ constant in time and space. We scale compound density by $u^*$ such that the flux density in the new coordinates becomes $1$. That is,
	\begin{equation*}
		\frac{\phi_0\tau_0}{2ru^*} = 1,\quad \gamma := \frac{u}{u^*}.  
	\end{equation*}
	Then, $(BVP_S)$ yields in a dimensionless system given by 
	\begin{equation*}
		(BVP'_S)\left\{
		\begin{aligned}
			\frac{\partial \gamma_S(\boldsymbol{\xi},\tau)}{\partial \tau} - \hat{D}\Delta_{\boldsymbol{\xi}} \gamma_S(\boldsymbol{\xi},\tau) &= 0, &\mbox{in $\hat{\Omega}\backslash\bar{\hat{\Omega}}_C, \tau>0$,}\\
			\hat{D}\nabla_{\boldsymbol{\xi}} \gamma_S(\boldsymbol{\xi},\tau)\cdot\boldsymbol{n}_{\boldsymbol{\xi}} &= 1, &\mbox{on $\partial\hat{\Omega}_C, \tau>0$,}\\
			\hat{D}\nabla_{\boldsymbol{\xi}} \gamma_S(\boldsymbol{\xi},\tau)\cdot\boldsymbol{n}_{\boldsymbol{\xi}} &= 0, &\mbox{on $\partial\hat{\Omega}, \tau>0$,}\\
			\gamma_S(\boldsymbol{\xi}, 0) &= \frac{u_0(\boldsymbol{\xi})}{u^*}, &\mbox{in $\hat{\Omega}\backslash\bar{\hat{\Omega}}_C$.}
		\end{aligned}
		\right.
	\end{equation*}
	A similar transformation can be done analogously in the point source approach:
	\begin{equation*}
		(BVP'_P)\left\{
		\begin{aligned}
			\frac{\partial \gamma_P(\boldsymbol{\xi},\tau)}{\partial \tau} - \hat{D}\Delta_{\boldsymbol{\xi}} \gamma_P(\boldsymbol{\xi},\tau) &= \hat{\Phi}\delta(\boldsymbol{\xi}-\boldsymbol{\xi}_c), &\mbox{in $\hat{\Omega}, \tau>0$,}\\
			\hat{D}\nabla_{\boldsymbol{\xi}} \gamma_P\cdot\boldsymbol{n}_{\boldsymbol{\xi}} &= 0, &\mbox{on $\partial\hat{\Omega}, \tau>0$,}\\
			\gamma_P(\boldsymbol{\xi}, 0) &= \frac{\bar{u}_0(\boldsymbol{\xi})}{u^*}, &\mbox{in $\hat{\Omega}$.}
		\end{aligned}
		\right.
	\end{equation*}
	Here,
	\begin{equation}\label{eq:phi rescaled}
		\hat{\Phi}\ =\ \frac{\Phi\tau_0}{u^*\cdot 4r^2}\ = \ \frac{2\pi r\phi_0 \tau_0}{u^*\cdot 4r^2}\ =\ \pi,
	\end{equation}
	under the given scalings.
	The factor $4r^2$ in the denominator of the first expression for $\hat{\Phi}$  in Equation \eqref{eq:phi rescaled} results from the different behaviour under the scaling transformation $x\mapsto \xi=x/(2r)$ of Dirac measure compared to Lebesgue measure in two dimensions, which is the reference measure for the density functions on the left-hand side in the differential equation.
	\vskip 0.2cm
	
	Having made explicit the non-dimensional models ($BVP'_S$) and ($BVP'_P$) and having observed their similarity to the dimensional models ($BVP_S$) and ($BVP_P$), we continue in our exposition with using the original notation of $u_S$, $u_P$, etc. for the dimension-free solutions and by omitting the `hats' on parameters. Moreover, we stress that in analysis and simulations we use the spatially homogeneous and constant (rescaled) flux density, i.e. $\phi(\boldsymbol{x},t) = 1$, in $(BVP_S)$ and corresponding flux at the point source  $\Phi(t) = \Phi = \pi$, according to Equation \eqref{eq:phi rescaled}.
	
	\subsection{Computational Approach}\label{Sec_Maths_Models}
	\noindent
	In \citet{HMEvers2015}, the authors worked analytically on the upper bound of the global difference between the solutions to these two approaches. Inspired by it, a subsequent research question is `Under what circumstances and to what extent, these two approaches are consistent?'  Therefore, the numerical simulations are employed in this manuscript to quantify and visualize the difference between the two solutions.
	A schematic representation of the mesh structure of the two approaches are displayed in Figure \ref{Fig_mesh}. In other words, for the spatial exclusion approach, we exclude the cell region from the computational domain.
	\begin{figure}[h!]
		\centering
		\subfigure[The spatial exclusion approach]{
			\includegraphics[width = 0.48\textwidth]{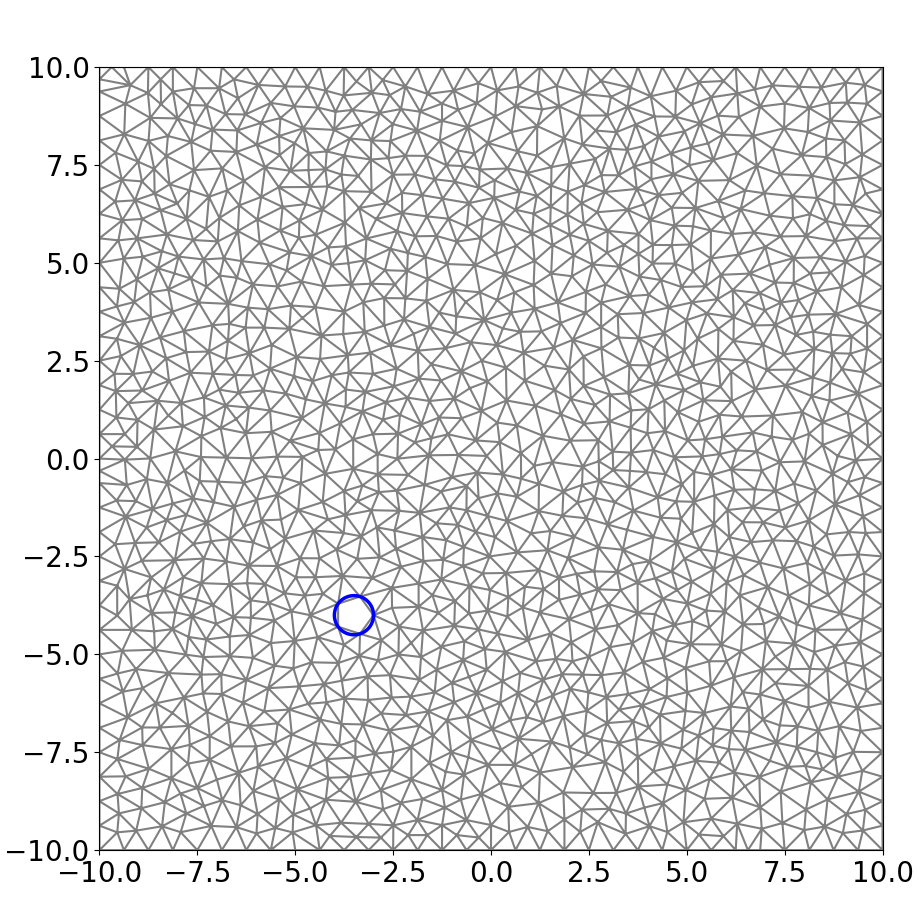}
			\label{fig_mesh_hole}}
		\subfigure[The point source approach]{
			\includegraphics[width = 0.48\textwidth]{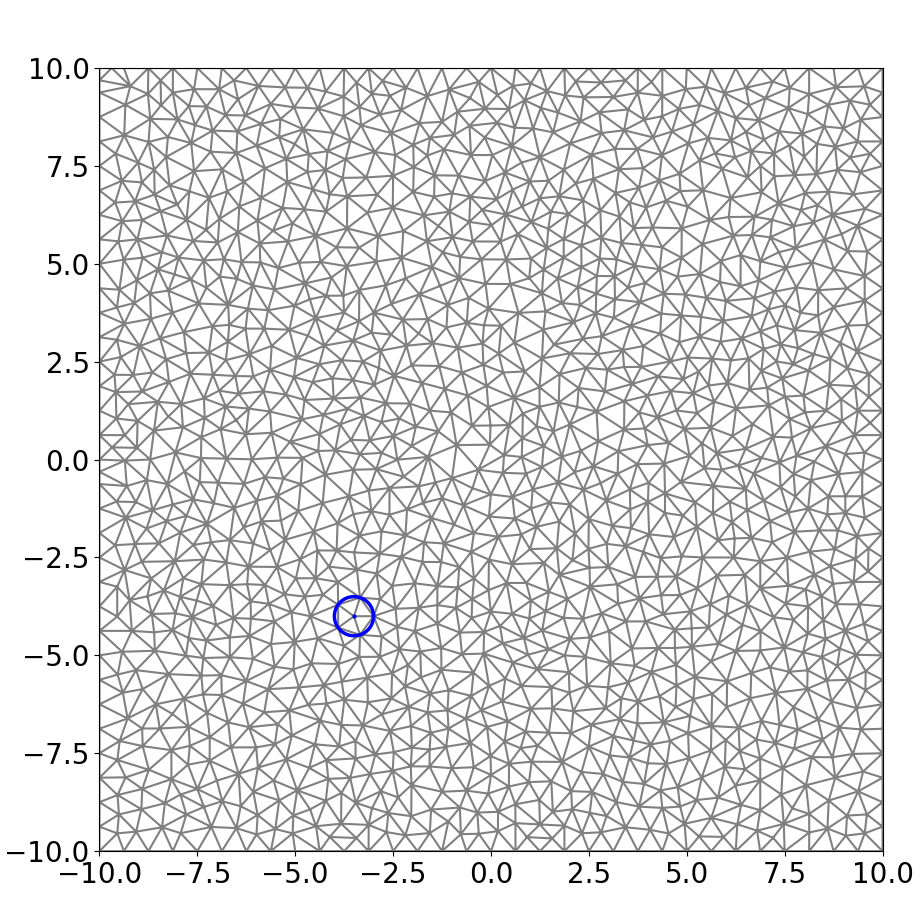}
			\label{fig_mesh_dirac}}
		\caption{A schematic representation of the two approaches and the mesh structure. The computational domain is $(-10, 10)\times(-10, 10)$, and the cell is located at $(-3.5, -4)$ with radius $\displaystyle\frac{1}{2}$. The blue circles are the predefined cell boundary.} 
		\label{Fig_mesh}
	\end{figure}
	
	The parameter values that are used in this section are shown in Table \ref{tab:para_all}, if there is no further specification. Note that the parameters are dimensionless. In this manuscript, finite-element methods and backward Euler are used for the numerical simulations, for the spatial discretisation and time integration respectively. Particularly, we use Python 3.10 and \texttt{FEniCS} package \citep{AlnaesEtal2015} version 2019.2.0.dev0. We bare in mind that in the implementation, instead of a smooth circle, the cell region is constructed by a series of mesh points as a polygon. In the point source approach, numerically, the initial condition at the mesh points on the cell boundary $\partial\Omega_C$ is taken the same as in the extracellular environment $\Omega\backslash\Omega_C$.
	\begin{table}[h!]\footnotesize
		\centering
		\caption{Parameter values used in Section \ref{Sec_Results_Somo} and \ref{Sec_Results_Inhomo}, corresponding to the dimensionless systems derived in Section \ref{subsec:non-dimensional models}}
		\begin{tabular}{m{2cm}m{2cm}m{8cm}}
			\toprule
			{\bf Parameter} & {\bf Value} & {\bf Description}  \\
			\midrule
			$\hat{D}$ & $0.1$ & Diffusion coefficient \\
			$L/(2r)$ & 10 & Size of the computational domain\\
			$\Delta \tau$ & $0.04$ & Time step \\
			$T$ & $40$ & Total time\\
			$h$ & $0.127$ & Average mesh size\\
			\bottomrule
		\end{tabular}
		\label{tab:para_all}
	\end{table}
	
	
	\subsection{Preliminary Results: Effects of Varying the Diffusion Constant}
	\label{subsec:preliminary results}
	\noindent
	A diffusing particle that is released at the centre of a circular cell of radius $R=1/2$ reaches the cell boundary on average on a time scale $R^2/D$. When the time step $\Delta\tau$ is larger than $R^2/D$, the global difference between the norms of the solutions in both approaches is not really significant; see Figure \ref{Fig_norm_D_10_01}\subref{fig_norm_D_10_all}-\subref{fig_norm_D_10_diff}. This is mainly due to the fact that within one time step, the compounds have already reached the boundary of the cell in the point source approach. Subsequently, the diffusion basically starts from the boundary of the cell at $t = \Delta\tau$. However, when we decrease the diffusion coefficient significantly, longer time is needed for the compounds to reach the boundary of the cell in the point source approach. In other words, the influx from the point source firstly needs to `fill' the intracellular space $\Omega_C$, before it can mimic diffusion from the boundary of the cell.

	\begin{figure}[h!]
		\centering
		\subfigure[$D = 10$]{
			\includegraphics[width = 0.48\textwidth]{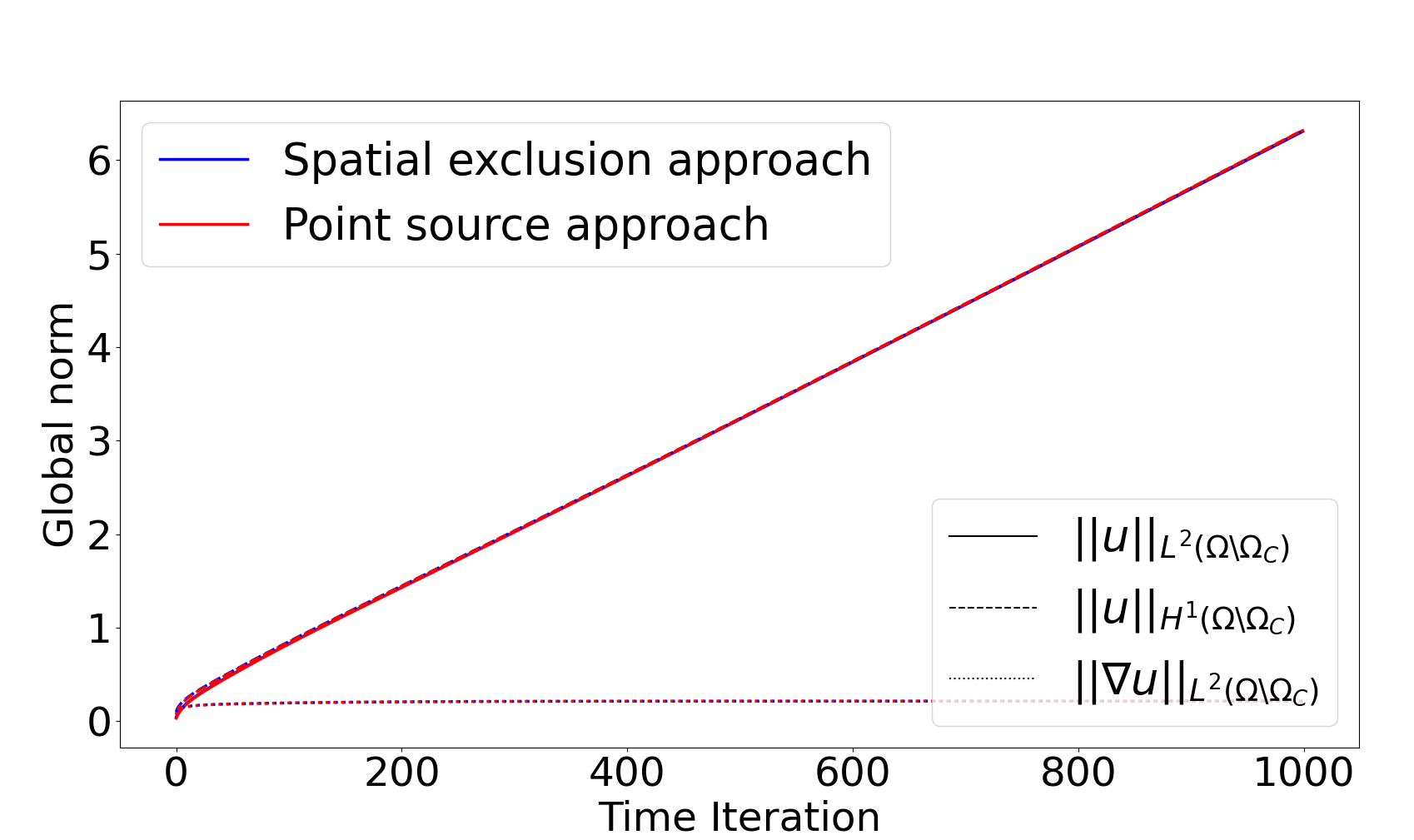}
			\label{fig_norm_D_10_all}}
		\subfigure[$D = 10$]{
			\includegraphics[width = 0.48\textwidth]{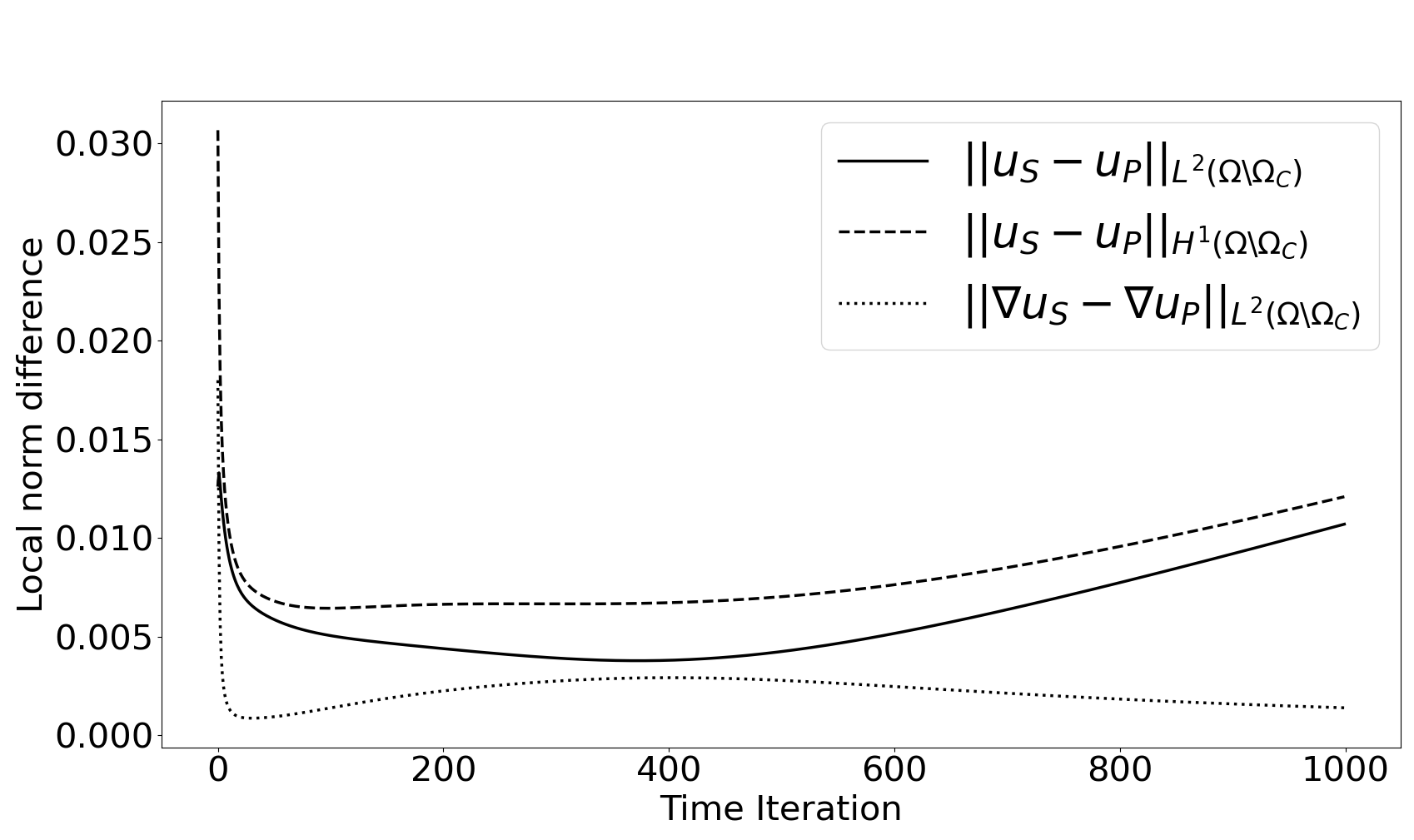}
			\label{fig_norm_D_10_diff}}
		\subfigure[$D = 1$]{
			\includegraphics[width = 0.48\textwidth]{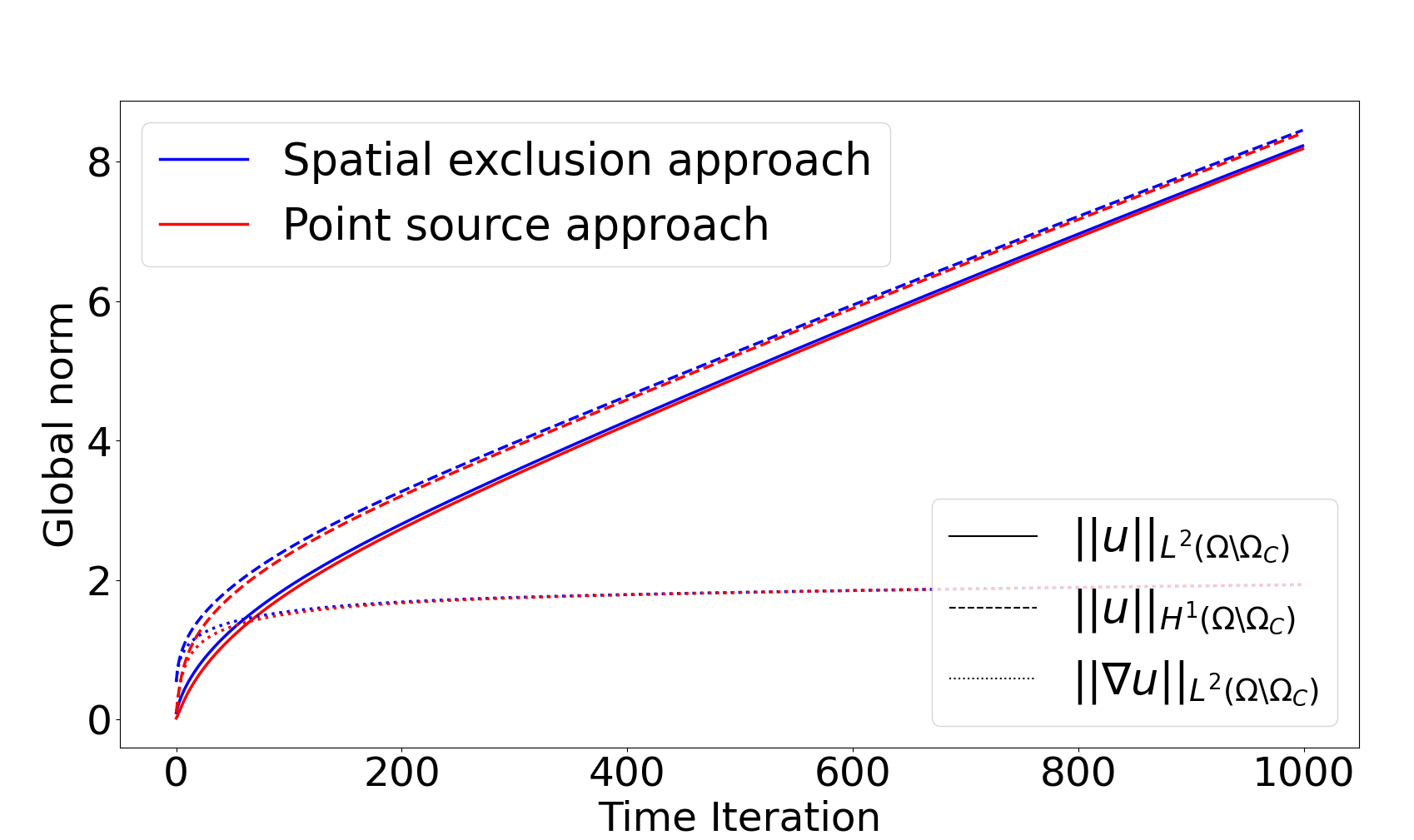}
			\label{fig_norm_D_1_all}}
		\subfigure[$D = 1$]{
			\includegraphics[width = 0.48\textwidth]{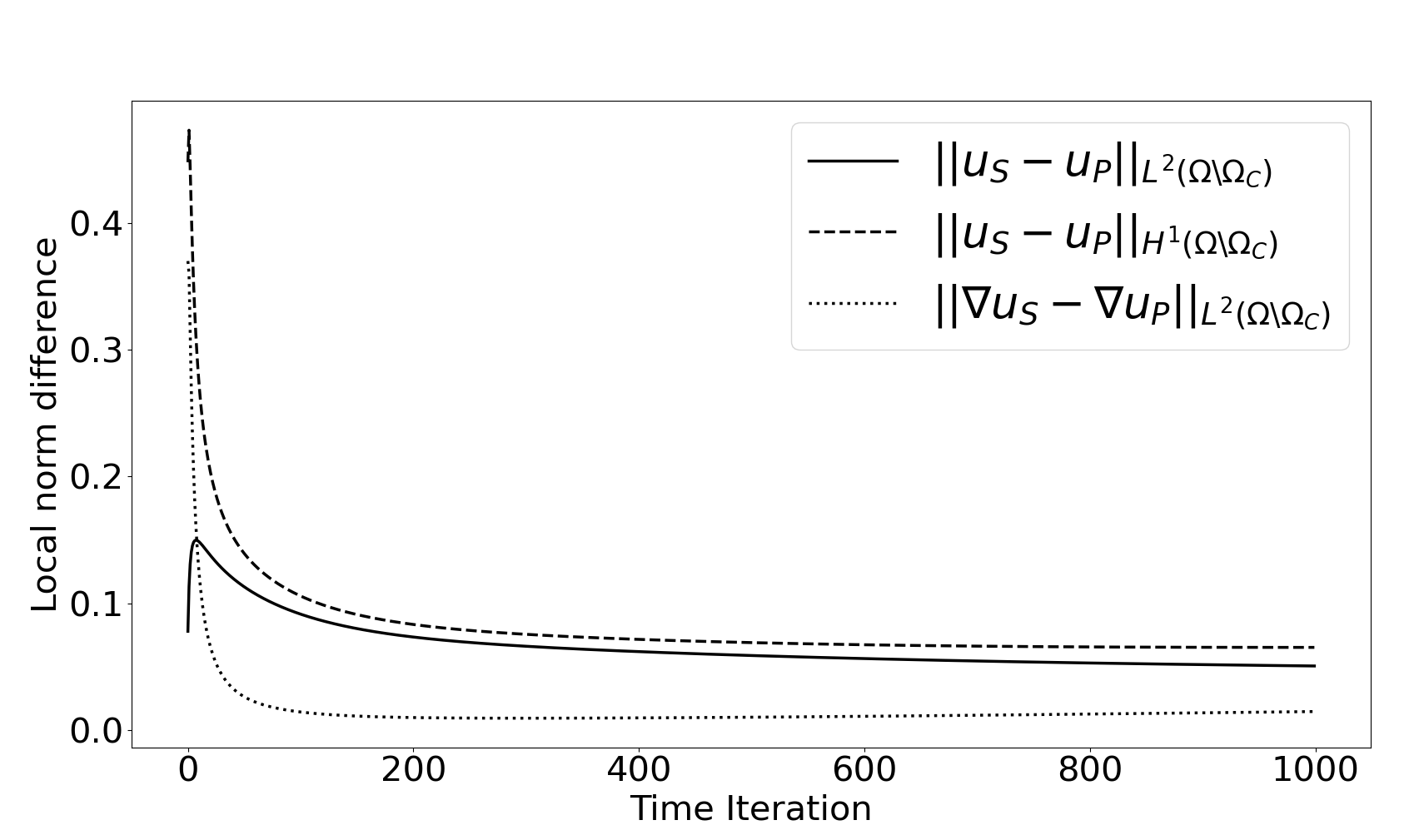}
			\label{fig_norm_D_1_diff}}
		\subfigure[$D = 0.1$]{
			\includegraphics[width = 0.48\textwidth]{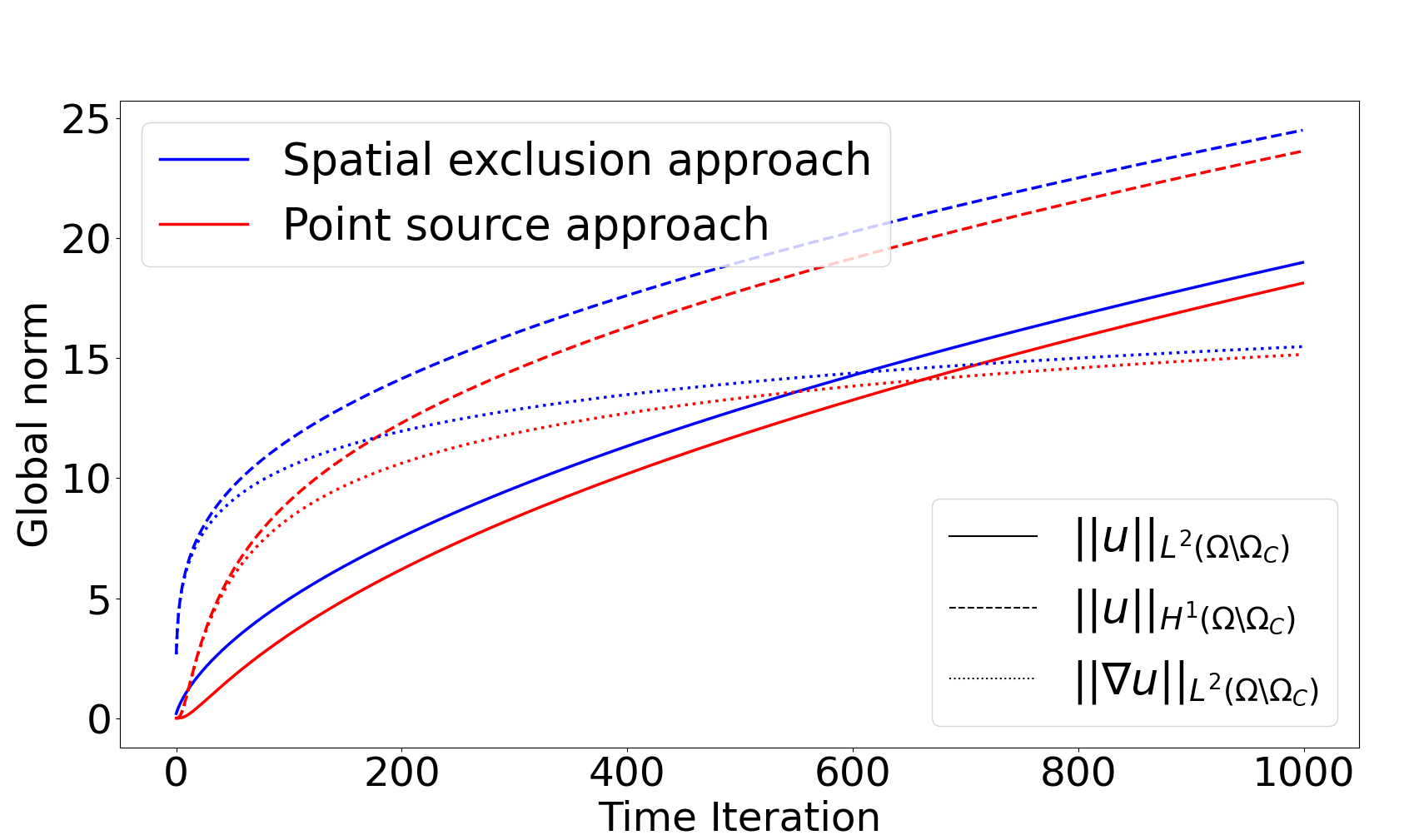}
			\label{fig_norm_D_01_all}}
		\subfigure[$D = 0.1$]{
			\includegraphics[width = 0.48\textwidth]{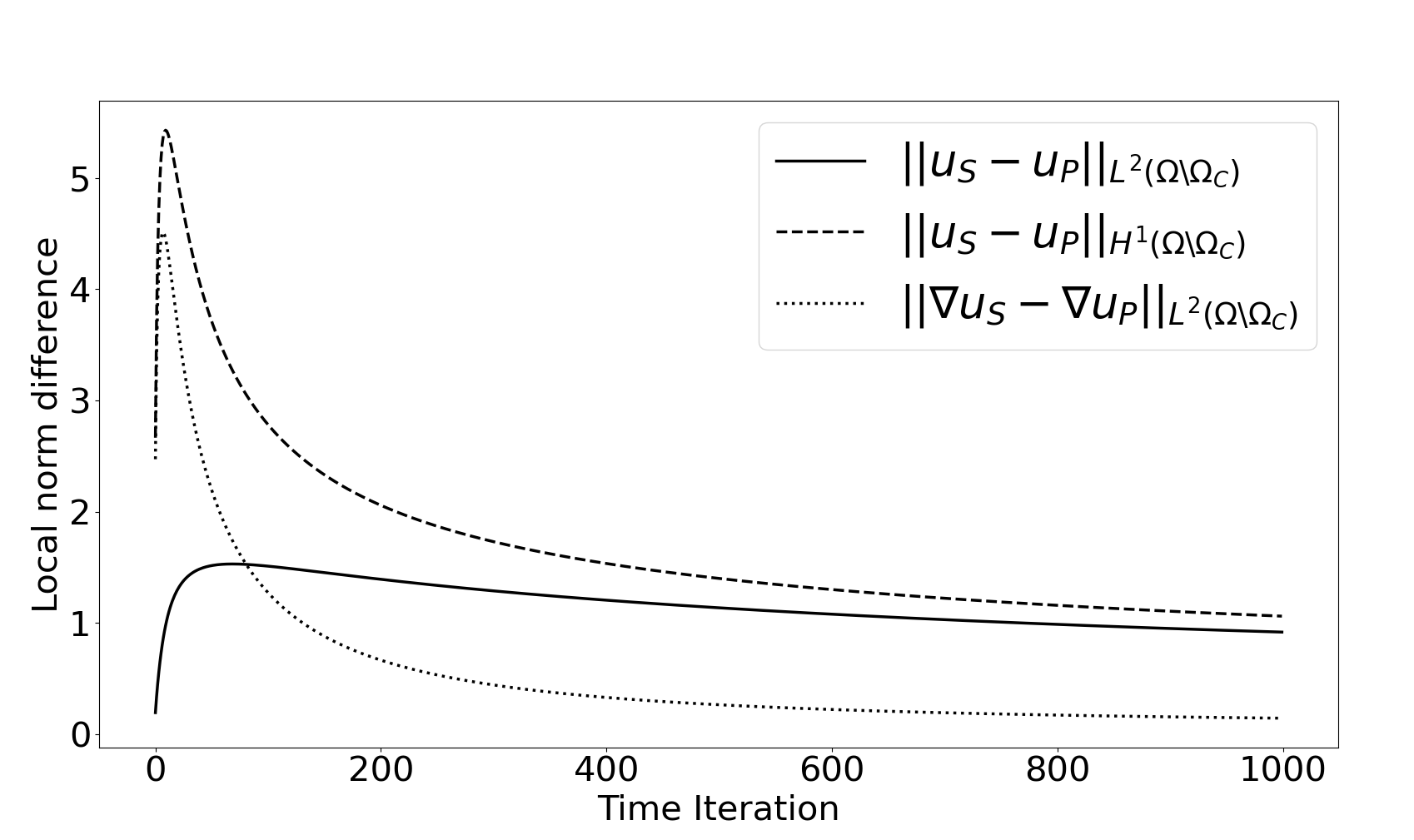}
			\label{fig_norm_D_01_diff}}
		\caption{Several simulations were conducted with varying diffusion coefficients, namely, $D \in \{10, 1, 0.1\}$. Zero initial conditions were taken in all cases. In the subfigures in the left column, various global norms (i.e. $\|u\|_{L^2(\Omega\backslash\Omega_C)},\|u\|_{H^1(\Omega\backslash\Omega_C)}, \|\nabla u\|_{L^2(\Omega\backslash\Omega_C)}$) are shown for both approaches. In the subfigures in the right column, we show the local norm differences of all the categories of aforementioned norms. The increase in $L^2$-norm of the solution on $\Omega\setminus \Omega_C$ -- hence also in $H^1$-norm -- in (a), (c) and (e) is due to the constant influx of compound in both spatial exclusion and point source approach. Here, $H^1(\Omega\setminus\Omega_C)$ is equiped with the Hilbert space norm $\|f\|_{H^1}^2 = \|f\|_{L^2}^2 + \|\, |\nabla f|\,\|_{L^2}^2$.}
		\label{Fig_norm_D_10_01}
	\end{figure}
	
	Notice that the local difference between the two solutions increases in all the norms with decreasing diffusivity $D$; see Figure \ref{Fig_norm_D_10_01} \subref{fig_norm_D_10_diff}, \subref{fig_norm_D_1_diff} and \subref{fig_norm_D_01_diff}. This is caused by the effect that even though compounds may have reached the boundary of the cell in the point source approach after some time, these can never fully `catch up' with the amount of mass already released by this boundary from the start in the spatial exclusion approach. Thus, one gets a systematic delay, visible in Figure \ref{Fig_norm_D_10_01}.
	
	Furthermore, as time proceeds, all the norms of differences slowly decrease for $D = 1$ and $D = 0.1$. The $L^2$-norm is controlled by expression in Equation \eqref{eq:L2-norm differnce}. After a transient, the difference in flux density over the boundary $\partial\Omega_C$ becomes small; see e.g. Figure \ref{fig_c_star} that indicates that concentration differences are small, hence also the differences in flux. Thus, the boundary-integral term in Equation \eqref{eq:L2-norm differnce} converges to 0. The first term, with the difference in gradients of solutions, will also converge to 0. It depends on the speed of convergence of both terms, whether the sign of the derivative of the $L^2$-norm of the difference will be positive or negative. This phenomenon can be observed in Figure \ref{Fig_norm_D_10_01} \subref{fig_norm_D_10_diff}, for $D=10$, where the $L^2$-norm of the difference is increasing towards the end of the simulation time interval. 
	
	This phenomenon seems to persist with decreasing $D$ until it disappears for $D$ between $1$ and $5$. To check whether this is a phenomenon in the numerical simulation rather than the model, we decreased the time step substantially. This did not affect the computed $L^2$-norm. We conclude, that it is a consequence of a change in dominance between the two integral terms in Equation \eqref{eq:L2-norm differnce}. It is not yet clear what controls this change precisely. 
	
	For the remainder of this paper, we shall focus on a second observation that can be made from the preliminary simulations: varying $D$, these show that for smaller diffusion constant, there is larger `delay' between the solutions, see Figure \ref{fig_norm_D_01_all}. This causes a systematic deviation between the two solutions. Hence, it is important to see how to reduce this effect. Therefore, in all following numerical simulations, we select $D = 0.1$, since it gives the largest deviation in norms between the spatial exclusion model and the point source model. 
	
	\section{Extending Zero environmental Initial Value}
	\label{Sec_Results_Somo}
	\noindent
	
	We shall first examine the issue of systematic delay for the point source solution when there is a zero initial condition in the environment for the spatial exclusion model. The situation where the initial condition there is constant, but non-zero, turned out to require a modified approach. It is discussed in the next section.
	
	Since the solutions in the two approaches are defined on different spatial domains, the initial condition for the spatial exclusion model cannot be simply carried over to the larger domain of the point source model: choices for extension to the complement $\bar{\Omega}_C$ must be made. In the preliminary results of Section \ref{subsec:preliminary results} we have seen that simply extending by zero leads to structural delay in the solution, especially for small diffusion. 
	
	\subsection{A Gaussian-shaped extension}
	\noindent 
	We consider now a single cell, centered at $\boldsymbol{x}_c$. According to Proposition \ref{Prop_condition}, it is expected that the initial condition in $(BVP_P)$, i.e. $\bar{u}_0(\boldsymbol{x})$, cannot be simply set to zero on $\Omega_C$ as extension of the zero initial value $u_0(\boldsymbol{x})$ of $(BVP_S)$ on $\Omega\setminus\Omega_C$. As the spatial exclusion approach removes the cell region in the computational domain and starts diffusion from the cell boundary directly, while in the point source approach, firstly the compound needs to reach the cell boundary from the cell center, which takes extra time. To compensate this time difference and inspired by the fact that the difference between two approaches (see Figure \ref{Fig_norm_D_10_01}\subref{fig_norm_D_1_all}-\subref{fig_norm_D_01_diff}) stays more or less constant, we set $\bar{u}_0(\boldsymbol{x})$ on $\Omega_C$ in the form of the fundamental solution to the diffusion equation on $\R^d$, which is given by \citep{evans2010partial} 
	\begin{equation}
		\label{Eq_fndmtal_sol}
		P^D(\boldsymbol{x},t) = \left\{
		\begin{aligned}
			&\frac{1}{(4\pi Dt)^{d/2}}\exp\left\{-\frac{|\boldsymbol{x}|^2}{4Dt}\right\}, &\mbox{ $t>0, \boldsymbol{x}\in\R^d$,}\\
			&0, &\mbox{ $t<0, \boldsymbol{x}\in\R^d$,}
		\end{aligned}
		\right.
	\end{equation}
	where $d$ is the dimension. In this study, we consider $d=2$ only. 
	If we imagine that diffusion inside the cell has started a time $t_0>0$ before the start of the point source model from a unit Dirac mass at the centre $\boldsymbol{x}_c$, then at time $t$ it will have reached a distribution on $\R^2$ of Gaussian shape 
	\begin{equation}
		\label{Eq_fndmtal_sol_t0}
		P_{t+t_0}^D(\boldsymbol{x},\boldsymbol{x}_c) = \frac{1}{4\pi D(t+t_0)}\exp\left\{-\frac{|\boldsymbol{x} - \boldsymbol{x}_c|^2}{4D(t+t_0)}\right\},\qquad \mbox{ $t,t_0>0,\ \boldsymbol{x}\in\R^2$.}
	\end{equation} 
	We can modify $t_0$ and the intensity $p_0>0$ of the initial condition to arrive at a proposed extension of the initial condition $\bar{u}_0(\boldsymbol{x})$ in $(BVP_P)$ as a (discontinuous) truncated and scaled fundamental solution:
	\begin{equation}
		\label{Eq_u0_dirac}
		\bar{u}_0(\boldsymbol{x}) = \begin{cases} \ p_0P_{t_0}^D(\boldsymbol{x},\boldsymbol{x}_c), \quad &\boldsymbol{x}\in\bar{\Omega}_C,\\
			\ 0, &\boldsymbol{x}\in\Omega\backslash\bar{\Omega}_C.
		\end{cases}
	\end{equation}
	The idea is, to choose $(p_0,t_0)$ in such a way that the flux condition of Proposition \ref{Prop_condition} is met in the best possible way. 
	
	Denoting $r = |\boldsymbol{x} - \boldsymbol{x}_c|$, the distance to the singular point, the flux density at $\boldsymbol{x}$ in the direction pointing away from $\boldsymbol{x}_c$ that originates from the initial condition only is given by
	\[
	\phi_1(r,t) = - D\nabla(p_0P_{t+t_0}^D(r))\cdot\boldsymbol{n}
	= -Dp_0\frac{\partial P_{t+t_0}^D(r)}{\partial r}
	= \frac{p_0r}{2(t+t_0)}P_{t+t_0}^D(r).
	\]
	Hence, the flux density over $\partial\Omega_C$ from the initial condition in Equation (\ref{Eq_u0_dirac}) reads as
	\begin{equation}
		\label{Eq_flux_initial}
		\phi_1(R,t) \ =\ \frac{p_0R}{2(t+t_0)}P_{t+t_0}^D(R)\ 
		=\ \frac{p_0R}{8\pi D(t+t_0)^2} \exp\left\{ -\frac{R^2}{4D(t+t_0)}\right\}.
	\end{equation}
	With the production of compounds at the center of the cell $\boldsymbol{x}_c$ and production rate $\Phi(\boldsymbol{x}_c)$, the fundamental solution of $(BVP_P)$ is given by
	\begin{align*}
		u_P(\boldsymbol{x},t) &= \int_{0}^{t}\int_{\R^2}P^D_{t-s}(\boldsymbol{x}, \boldsymbol{y})\delta(\boldsymbol{y}-\boldsymbol{x}_c)\Phi(\boldsymbol{x}_c)d\boldsymbol{y}ds\\
		& = \int_{0}^t\Phi(\boldsymbol{x}_c)P^D_{t-s}(\boldsymbol{x}, \boldsymbol{x}_c)ds
		\ =\ \int_{0}^t \Phi(\boldsymbol{x}_c)P^D_{t-s}(r)ds.
	\end{align*} 
	Subsequently, the flux density caused by the point source at $\boldsymbol{x}_c$ only is computed as 
	\begin{align*}
		\phi_2(r, t) &= -D\nabla u(r,t)\cdot\boldsymbol{n}\ 
		=\ -D\frac{\partial u(r,t)}{\partial r}\\
		& = D\frac{\partial}{\partial r}\int_{0}^t  \Phi(\boldsymbol{x}_c)P^D_{t-s}(r)ds
		\ =\  \frac{\Phi(\boldsymbol{x}_c)}{2\pi r}\exp\left\{-\frac{r^2}{4Dt}\right\}.
	\end{align*}
	Then, we obtain the flux over $\partial\Omega_C$ as 
	\begin{equation}
		\label{Eq_flux_prod}
		\phi_2(R,t) = \frac{\Phi(\boldsymbol{x}_c)}{2\pi R}\exp\left\{-\frac{R^2}{4Dt}\right\}.
	\end{equation}
	According to Proposition \ref{Prop_condition}, and given $-D\nabla u_P\cdot\boldsymbol{n} = \phi_1(R,t) +\phi_2(R,t)$ over $\partial\Omega_C$, the relation between the scale $p_0$ and prediffused time $t_0$ can be determined from the approximate equation
	\begin{align}
		\phi(\boldsymbol{x},t) \ &=\ -D\nabla u_P\cdot\boldsymbol{n}\ \approx\ \phi_{sum} :=\ \phi_1(R,t) +\phi_2(R,t) \nonumber \\
		&\ =\ \frac{p_0R}{8\pi D(t+t_0)^2} \exp\left\{ -\frac{R^2}{4D(t+t_0)}\right\}\ + \ \frac{\Phi(\boldsymbol{x}_c)}{2\pi R}\exp\left\{-\frac{R^2}{4Dt}\right\}.\label{eq:expresion phi_sum}
	\end{align}
	Let $t=0$, then 
	\begin{equation}
		\label{Eq_p0_t0}
		p_0(t_0) \approx \frac{2t_0\phi(\boldsymbol{x},t)}{RP^D_{t_0}(R)}.
	\end{equation}

	\subsection{A Comparison: Gaussian-shaped Initial Value inside the Cell}
	\noindent 
	We set $D=0.1$, such that the time delay between the solutions in the two approaches is the largest in the range of $D$ that we consider, see Section \ref{subsec:preliminary results}. We take as initial condition $\bar{u}_0$ on $\Omega$ for the point source model the function defined in Equation \eqref{Eq_u0_dirac}. Strictly speaking, this function is not in $H^1(\Omega)$. However, numerically the function value of $\bar{u}_0$ at the mesh points in the interior of the cell, i.e. in $\Omega_C$ (and 0 for those on the boundary $\partial\Omega_C$, see Section \ref{Sec_Maths_Models}) are used to obtain a numerical approximation for $\bar{u}_0$ through FEM that is in $H^1$. This is the extension $\bar{u}_0$ that is actually considered, but which is difficult to prescribe explicitly.
	
	
	Thus, one has two degrees of freedom in selecting the Gaussian-shaped initial condition. The question is, how to choose the value of $(p_0, t_0)$. 
	As first option we considered determining $(p_0,t_0)$ by minimizing the total deviation between $\phi_{sum}(R, t)$ and $\phi(\boldsymbol{x},t) = 1$ over the time interval $[0,T]$ in $L^1$-sense. That is, $(p_0,t_0)$ is taken as
	\begin{equation}
		\label{Eq_p0_t0_Case_2}
		(\hat{p}_0, \hat{t}_0) \in  \argmin_{(p_0, t_0)} \int_0^T |\\Phi_{sum}(t) - \phi(\boldsymbol{x}, t)| dt,
	\end{equation}
	This choice is referred to as `\textit{Option 1}'.
	
	Figure \ref{Fig_zero_Gauss} shows the global norms of the solutions and norms of differences between the spatial exclusion approach and the point source approach with two types of initial condition: zero initial condition compared to a Gaussian-shaped initial condition on the cell's interior. The shape parameter $(p_0,t_0)$ has been determined according to \textit{Option 1}, i.e. Equation \eqref{Eq_p0_t0_Case_2}. In Figure \ref{Fig_zero_Gauss}(a) one sees that the norms of the individual solutions of spatial exclusion and point source model will converge to each other after a transient, when the Gaussian-shaped extension of the initial condition is used. It indicates that the use of Gaussian-shaped initial condition does compensate the time delay, compared to the use of zero initial condition. The norms of the difference of the two solutions (with Gaussian-shaped extension) converge quickly to a steady -- though non-zero -- level. So, error is controlled well after an initial transient. Moreover, the error is less than that in the case of the zero-extension, as can be seen in Figure \ref{Fig_zero_Gauss}(b). Thus, the Gaussian-shaped extended initial condition yields an improvement of the approximation over the homogeneously-zero extended initial condition.
	
	When the Gaussian-shaped initial condition is used, due to the discontinuity of $\bar{u}_0(\boldsymbol{x})$ given by Equation (\ref{Eq_u0_dirac}) in the point source approach,  the $H^1$-approximation of $\bar{u}_0(\boldsymbol{x})$ that is taken as initial condition in the numerical analysis has a large gradient locally in a small region of the boundary of the cell. This produces a spike of $\|u_S-u_P\|_{H^1(\Omega\backslash\Omega_C)}$ and $\|\nabla u_S-u_P\|_{L^2(\Omega\backslash\Omega_C)}$ that appear in the beginning of the simulation. However, the differences quickly dropped to a lower level, compared to when the homogeneously-zero initial condition is used. 
	\begin{figure}
		\centering
		\subfigure[Global norm differences of solutions to the two approaches]{
			\includegraphics[width = 0.75\textwidth]{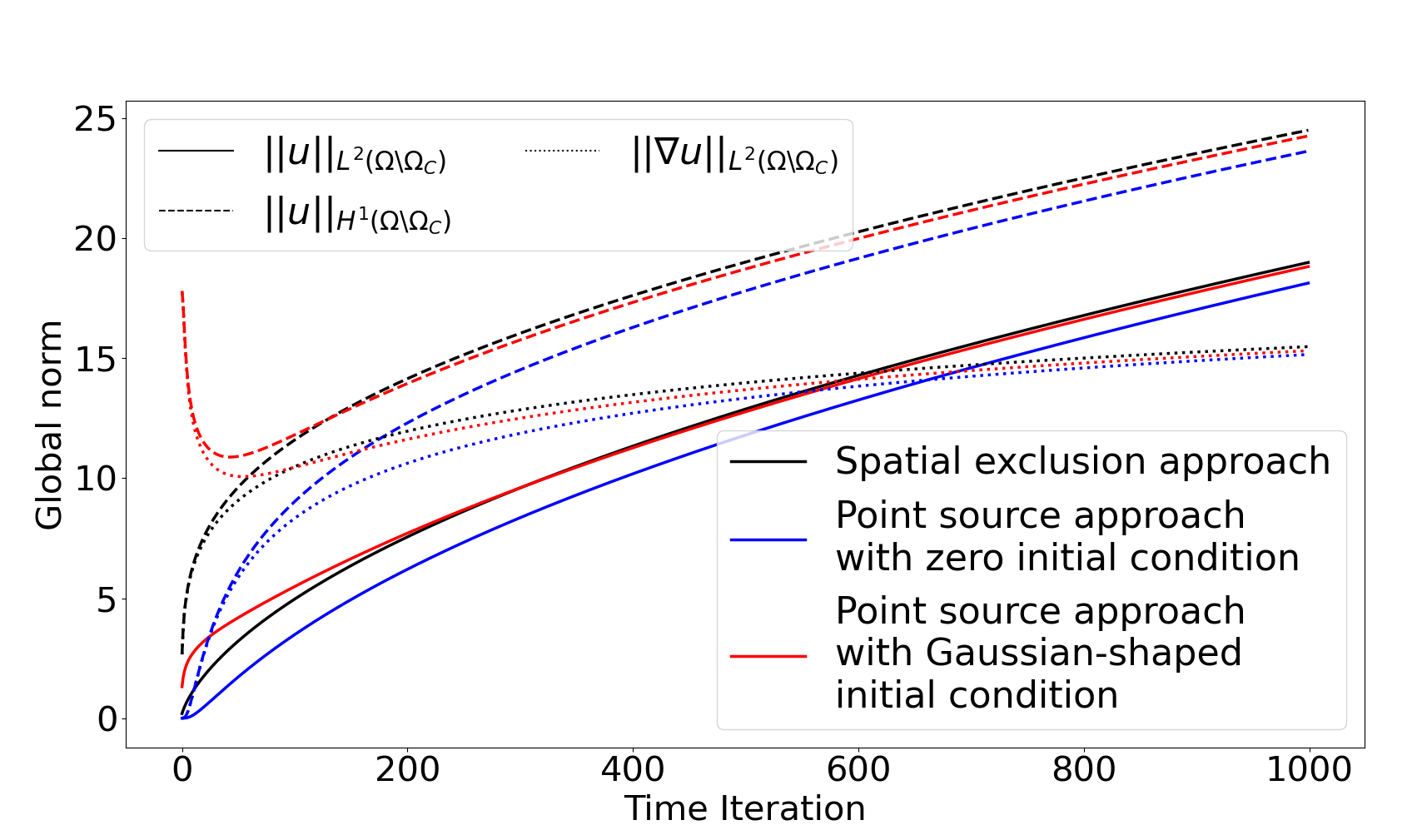}
		}
		\subfigure[Local norm differences of solutions to the two approaches]{
			\includegraphics[width = 0.48 \textwidth]{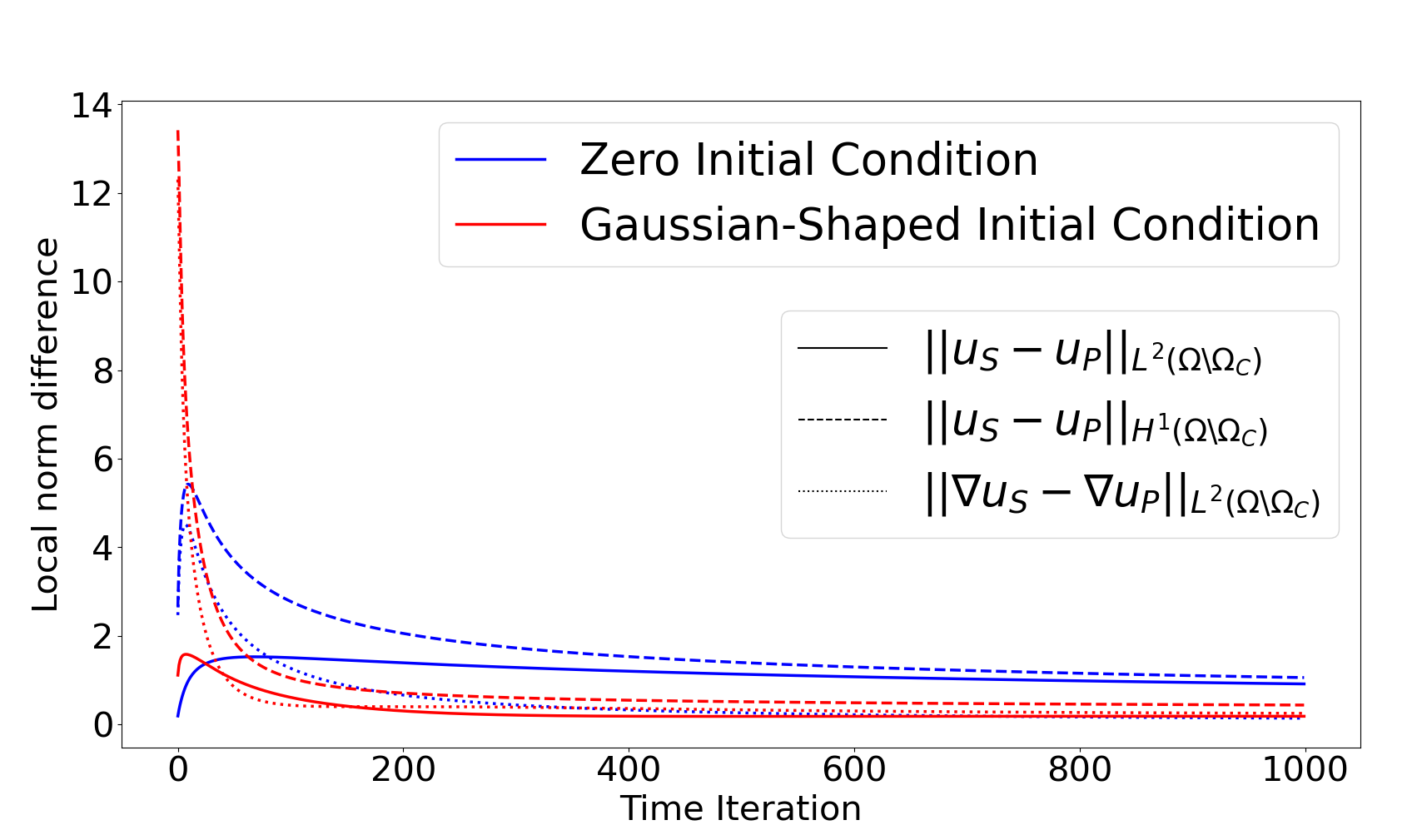}
			\label{fig_norms}}
		\subfigure[$c^*(t)$]{
			\includegraphics[width=0.48\textwidth]{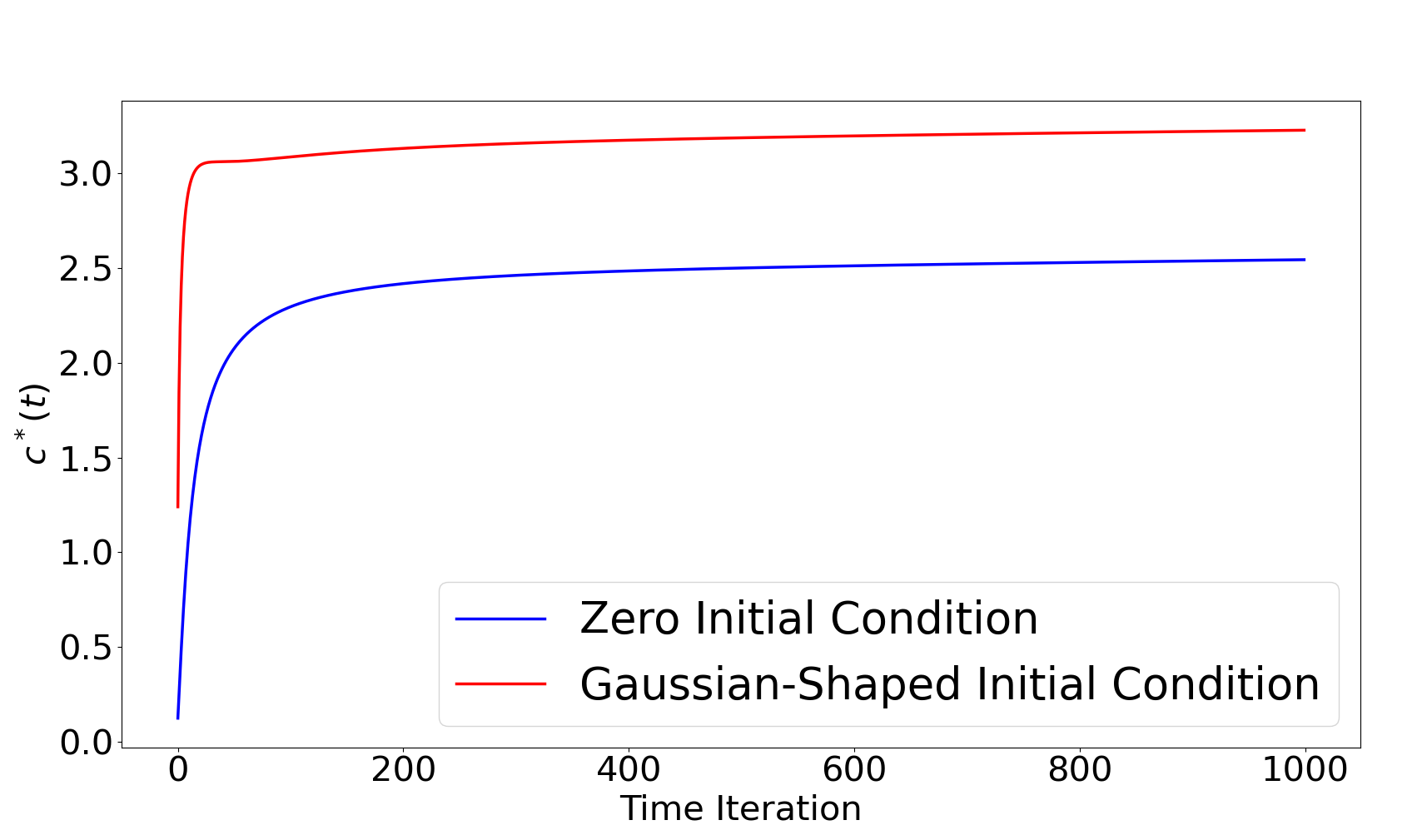}
			\label{fig_c_star}}
		\caption{For $D=0.1$, computed (a) global norms for solutions and (b) local norm differences (namely, the $L^2-$ and $H^1-$norm of the solution $u$, and the $L^2-$norm of the gradient of the solution $u$) between the two approaches, as well as (c) the quantity $c^*(t)$. Here, the amplitude and the variance of the Gaussian-shaped initial condition are determined by minimizing the total flux deviation on $\partial\Omega_C$ over $[0, T]$, i.e. $(p_0,t_0)$ have been taken as in $\textit{Option 1}$ in Table \ref{Tbl_opt_options}. Black curves represent results from the spatial exclusion approach, and blue and red curves represent the zero and Gaussian-shaped initial conditions, respectively. Solid, dashed and dotted lines represent $\|u_S-u_P\|_{L^2(\Omega\backslash\Omega_C)}$, $\|u_S-u_P\|_{H^1(\Omega\backslash\Omega_C)}$ and $\|\nabla u_S-\nabla u_P\|_{L^2(\Omega\backslash\Omega_C)}$, respectively. Note that the introduction of the Gaussian-shaped initial condition in the point source approach removed the time delay, essentially.}
		\label{Fig_zero_Gauss}
	\end{figure}
	
	Figure \ref{Fig_zero_Gauss}\subref{fig_c_star} shows the time-integrated deviation between the prescribed flux $\phi(\boldsymbol{x},t)=1$ in $(BVP_S)$ and the flux generated from $(BVP_P)$ on the cell boundary $\partial\Omega_C$ (i.e. Equation (\ref{Eq_c_star})). Due to the discontinuity in the initial condition for $(BVP_P)$, the gradient of the flux in $(BVP_P)$ over $\partial\Omega_C$ is large, hence, the $c^*(t)$ is larger when the Gaussian-shaped initial condition is used than for the zero initial condition. As shown in Figure \ref{Fig_zero_Gauss}\subref{fig_c_star}, this measure of quality reaches a rather small steady rate of increase quicker for the Gaussian-shaped condition. However, upon quick inspection, \ref{Fig_zero_Gauss}\subref{fig_c_star} seems to indicate that the zero initial condition performs better than the Gaussian-shaped extension, which is not the case in view of the norms of local differences presented in Figure \ref{Fig_zero_Gauss}(b). 
	\vskip 0.2cm
	
	We conclude that the Gaussian-shaped extension of the initial condition improves the quality of approximation compared to the homogeneously-zero initial condition. However, {\it the graph of $c^*(t)$ turns out to be difficult to interpret towards drawing conclusions on the quality of approximation in terms of $L^2$- and $H^1$-norms.}

	\subsection{Multiple Cells in the Computational Domain}\label{Sec_number_of_holes}
	\noindent
	
	The chief benefits of the point source model over the spatial exclusion model lie in analytical tractability and computational efficiency when there are many cells, in particular when these are also moving. Therefore, we shall now examine effects that may occur due to the presence of multiple cells. That there will be some effect of multiple cells can be anticipated. Intuitively speaking, part of the flux of mass released by one cell in the point source model will at some point in time freely transverse the part of the domain that is the interior of another cell in the spatial exclusion model. In the latter, diffusing particles would have reflected on the boundary of this cell. So, different trajectories of the diffusion process are expected, which may result in differences in the solutions. The impact of this phenomenon will depend on the distance between the cells, their size and the diffusivity.
	\vskip 0.2cm
	
	In this study, we shall only exhibit the impact of this phenomenon. Investigation of the question how to compensate this impact is deferred to another study. For simplicity, we consider two and ten cells respectively with different locations. All the cells are assumed to be identical in shape, size and (constant) flux density over the boundary. The locations of the cells over the domain are shown in Figure \ref{Fig_mesh_number_of_cells} in Appendix \ref{app_sec:location cells}. Similarly to the previous section, we plot against time various norm differences of the solutions, and $c^*(t)$ of the same single cell, i.e. the cell that is present at the same location in all configurations. In every subfigure in Figure \ref{Fig_multi_holes}, the left panel shows the result when the homogeneous initial condition is applied, i.e. $u_0(\boldsymbol{x},0) = 0$ for the computational domain, and the right panel is when the Gaussian-shaped extension is used as the initial condition inside the cell $\Omega_C$.

	\begin{figure}[h!]
		\centering
		\subfigure[$\|u_S(\boldsymbol{x},t) - u_P(\boldsymbol{x},t)\|_{L^2(\Omega\backslash\Omega_C)}$]{
			\includegraphics[width = 0.45\textwidth]{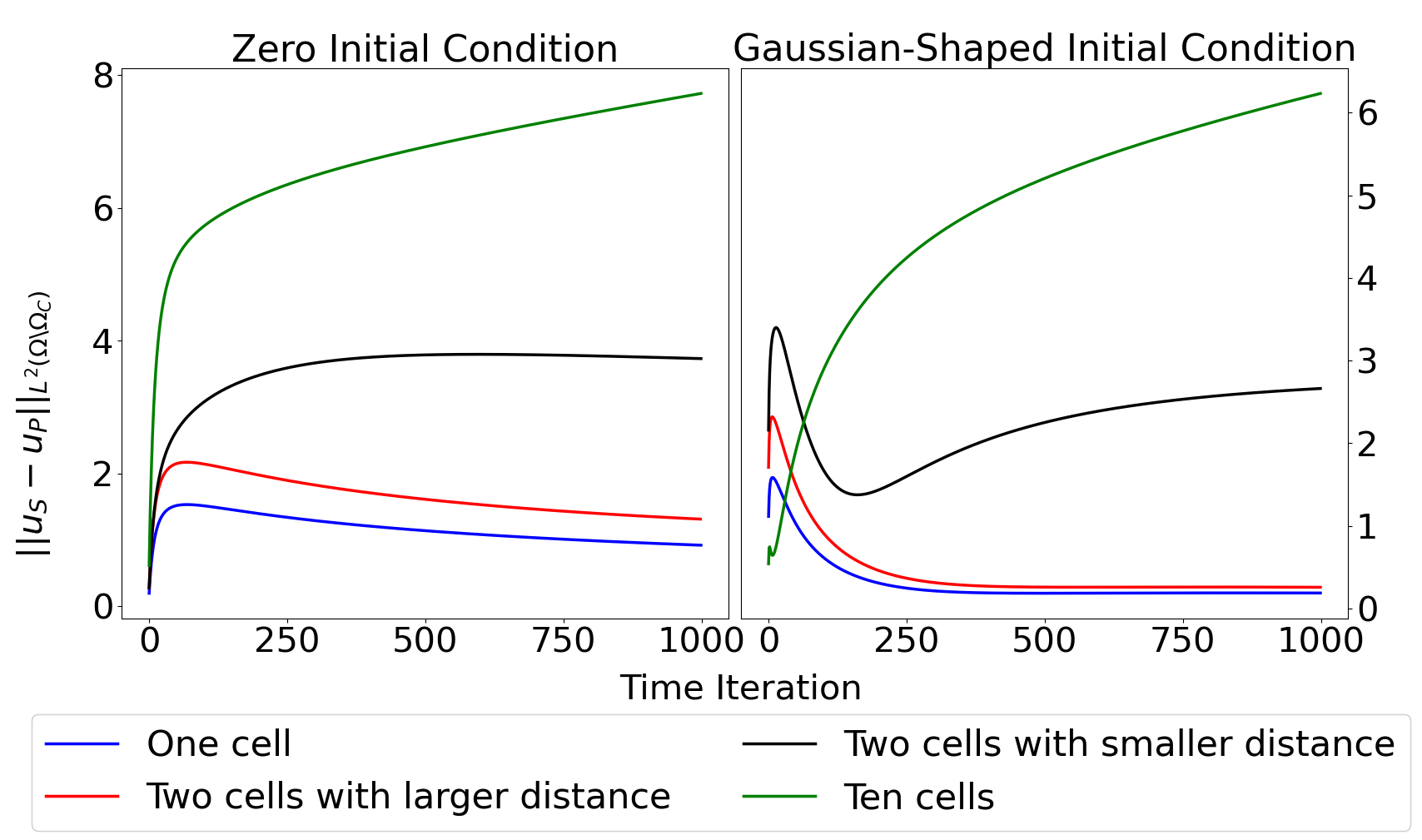}}
		\subfigure[$\|\nabla u_S(\boldsymbol{x},t) - \nabla u_P(\boldsymbol{x},t)\|_{L^2(\Omega\backslash\Omega_C)}$]{
			\includegraphics[width = 0.45\textwidth]{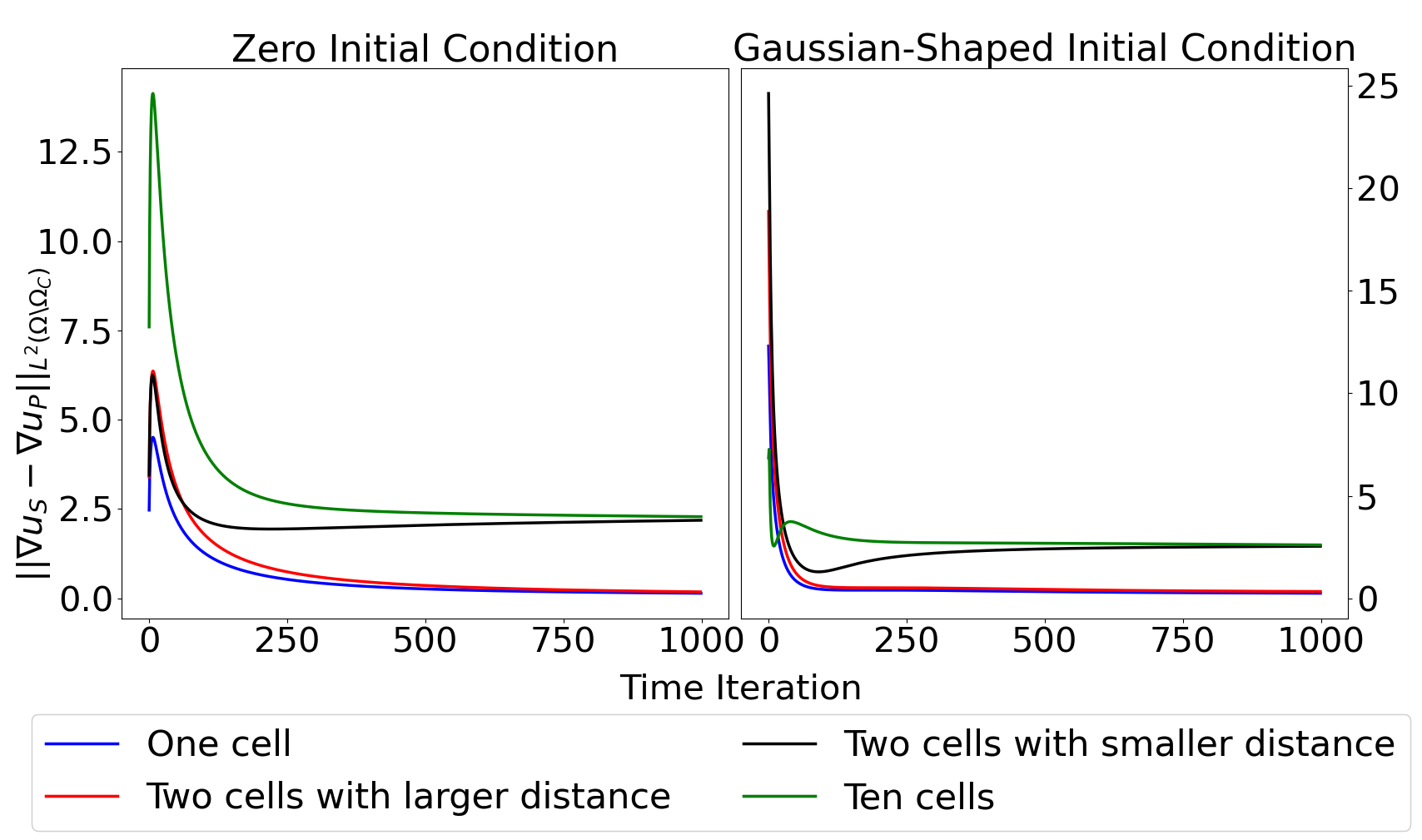}}
		\subfigure[$\|u_S(\boldsymbol{x},t) - u_P(\boldsymbol{x},t)\|_{H^1(\Omega\backslash\Omega_C)}$]{
			\includegraphics[width = 0.45\textwidth]{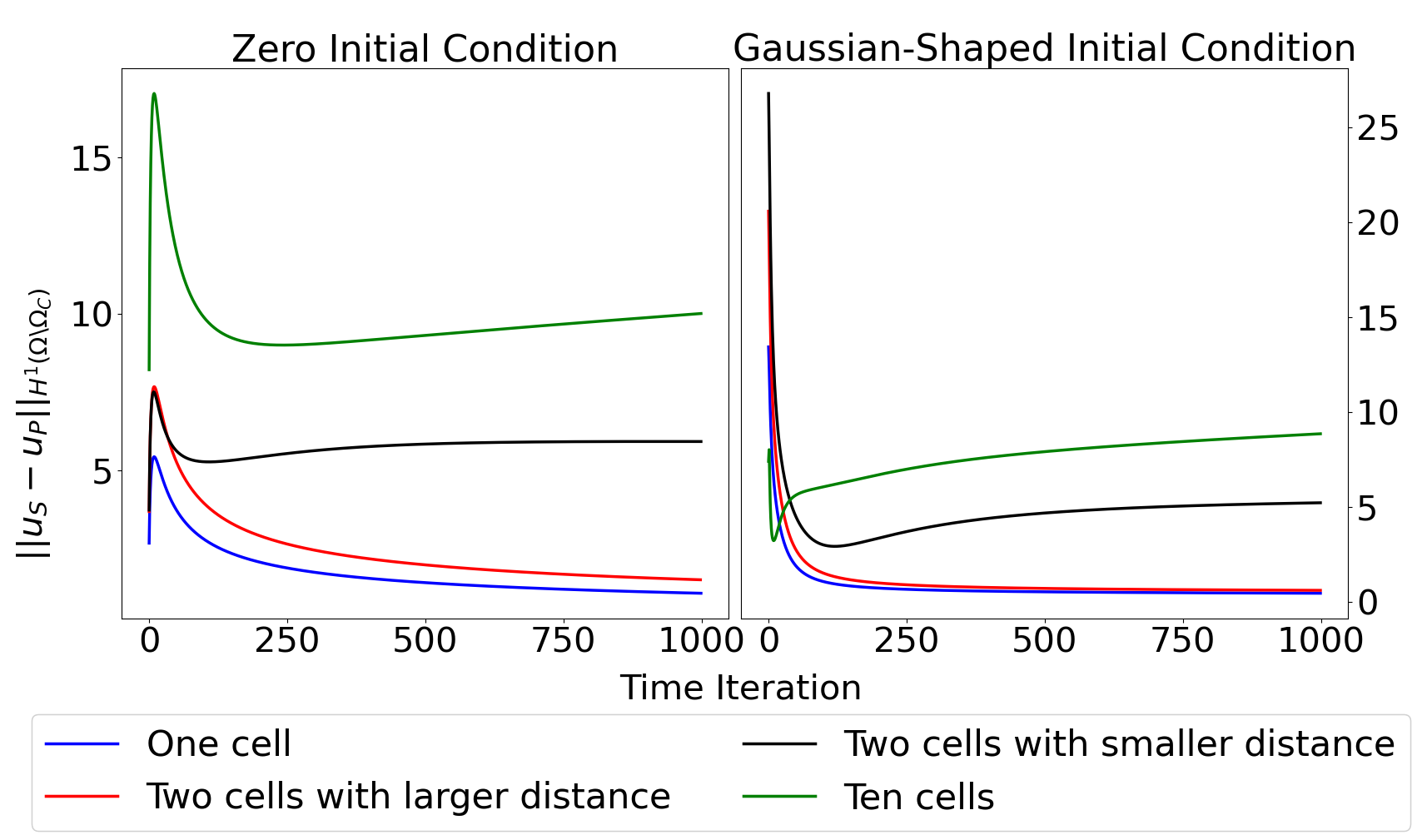}}
		\subfigure[$c^*(t)$ of the same cell]{
			\includegraphics[width = 0.45\textwidth]{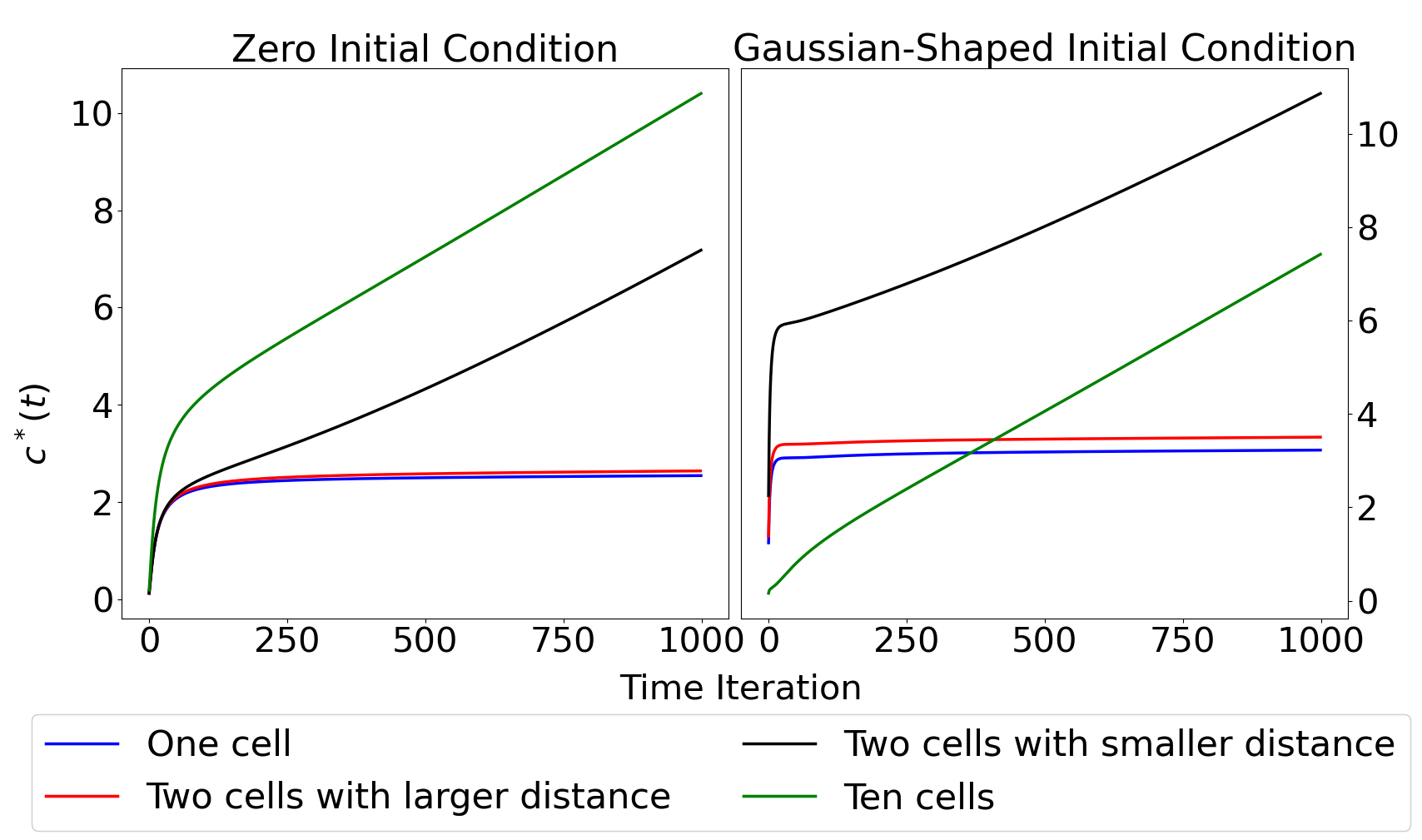}}
		\caption{Numerical results of multiple cells in the computational domain. Cells are identical ($\phi_i=1$, $R_i=\frac{1}{2}$). $D=01.$. Results have been plotted for homogeneously-zero initial condition (solid curve in the left) and Gaussian-shaped initial condition (dashed curve in the left) in every subfigure. Different colors of curves represent different number of cells, as indicated. The used cell configurations for the computations are shown in Figure \ref{Fig_mesh_number_of_cells}.}
		\label{Fig_multi_holes}
	\end{figure}
	
	Generally speaking, we observe the existence of the impact from other cells, since for every error quantifier, the more cells in the computational domain, the larger the quantity. Furthermore, the distance between the cells plays a significant role in this impact: for the two-cell case, when the cells are closer to each other, the influence appears sooner, and $c^*(t)$ is increasing continuously and seems to not be able to reach a steady state, which is opposite to the two-cell case with a larger distance. As the number of cells increases, the quantities are increasing as well except for the $c^*(t)$, which results from the minimal distance between the cells being smaller than the two-cell case with a smaller distance. In other words, it verifies the assumption that the distance between cells is a significant factor to study the consistency of the solutions to the two approaches when there are multiple cells in the computational domain. Also in this setting, a Gaussian-shaped extension of the initial conditions yields a better approximation, although the difference tends to become smaller when the number of cells increases.
	
	\subsection{Optimal Selection of $(p_0, t_0)$ in the Initial Condition}
	\noindent
	So far, we compute the amplitude $p_0$ and variance parameter $t_0$ by minimizing expression (\ref{Eq_p0_t0_Case_2}). However, there are other options to select $(p_0, t_0)$ by minimizing different objective functions. Key objective is to minimize the deviation between $\phi(\boldsymbol{x},t)$ and $\phi_{sum}(t)$ as expressed in Equation \eqref{eq:expresion phi_sum}. The $L^1$-norm deviation over time, as in Equation \eqref{Eq_p0_t0_Case_2} may be replaced by e.g. supremum norm. In Lemma \ref{lem:shape phi-sum} we summarized characteristic properties of $\phi_{sum}(t)$. It may have at most two extreme values, one above the (constant) level $\phi$, one below. In order to make both extreme deviations from the target value $\phi$ as small as possible, one can take as objective to minimize
	\begin{equation}\label{eq:max-min constraint}
		\bigl|\max_{0\leqslant t\leqslant T}(\phi_{sum}(t)-\phi)\bigr|\ +\ \bigl|\min_{0\leqslant t\leqslant T}(\phi_{sum}(t)-\phi)\bigr|,
	\end{equation}
	over $(p_0,t_0)$. One may think too of starting the approximation $\phi_{sum}$ at level $\phi$ at $t=0$. This yields Equation \eqref{Eq_p0_t0} as constraint on the $(p_0, t_0)$ value pairs. There is no {\it a priori} guarantee however, that this yields the best result for the corresponding unconstrained minimisation problem with the same objective function. A combination of objective functions defined in Equations \eqref{Eq_p0_t0_Case_2} and \eqref{eq:max-min constraint} has also been considered.

	Table \ref{Tbl_opt_options} gives an overview of the objective functions that have been examined, each with and without the value-pair constraint (Equation \eqref{Eq_p0_t0}). The $(p_0,t_0)$-values thus obtained for each option, by applying the \texttt{optimize.minimize} function of Python package \texttt{Scipy} (version 1.8.0) are shown in the last two columns. 
	
	\begin{table}\footnotesize
		\centering
		\caption{Considered options to compute $(p_0, t_0)$ as minimisation of the indicated different objective functions and constraint on the $(p_0,t_0)$-value pair. Initial condition is $u_0 = 0$ in $\Omega\backslash\Omega_C$. The optimization was performed by the function \texttt{optimize.minimize} of \texttt{Scipy} package (version $1.8.0$) in Python.}
		\begin{tabular}{p{1.7cm}<{\centering}p{6cm}<{\centering}p{3cm}<{\centering}p{2cm}<{\centering}p{2cm}<{\centering}}
			\toprule
			{\bf Options} & {\bf Objective function} & {\bf Constraints} & {\bf Value of $p_0$} & {\bf Value of $t_0$} \\
			\midrule
			{\bf Option 1} & \multirow{2}{*}{$ \int_0^T |\phi_{sum}(t) - \phi| dt$} & -  & $55.379$ & $3.660$ \\
			{\bf Option 2} &  & $\displaystyle p_0(t_0) = \frac{2t_0\phi}{RP^D_{t_0}(R)}$ & $37.114$ & $2.383$ \\
			\midrule
			{\bf Option 3} & 
			$ |\max_{0\leqslant t\leqslant T}(\phi_{sum}(t)-\phi)|$ & - & $27.946$ & $1.905$\\
			{\bf Option 4} & $+ |\min_{0\leqslant t\leqslant T}(\phi_{sum}(t)-\phi)|$ & $\displaystyle p_0(t_0) = \frac{2t_0\phi}{RP^D_{t_0}(R)}$ & $21.737$ & $1.737$ \\
			\midrule
			{\bf Option 5} & $ \int_0^T |\phi_{sum}(t) - \phi| dt + |\max_{0\leqslant t\leqslant T}(\phi_{sum}(t)-\phi|$ & - & $31.451$ & $2.086$ \\
			{\bf Option 6} & $+ |\min_{0\leqslant t\leqslant T}(\phi_{sum}(t)-\phi)|$ & $\displaystyle p_0(t_0) = \frac{2t_0\phi}{RP^D_{t_0}(R)}$ & $34.439$ & $2.283$\\
			\bottomrule
		\end{tabular}
		\label{Tbl_opt_options}
	\end{table}
	
	Figure \ref{Fig_p0_t0_all_cases} shows the plot of the important quantifiers of the consistency between the approaches. As a reference, we add the curve which is obtained by defining zero initial conditions over $\Omega$ in $(BVP_P)$. Extending the initial condition on $\Omega_C$ with the Gaussian kernel provides sufficient flux to reduce the local difference between the solutions in these two approaches, which can be seen in Figure \ref{Fig_p0_t0_all_cases}(a). Nevertheless, due to discontinuity of the initial condition on $\partial\Omega_C$ as defined in Equation (\ref{Eq_u0_dirac}), similarly to Figure \ref{Fig_zero_Gauss}(a), using the inhomogeneous initial condition causes that the $H^1-$norm (see Figure \ref{Fig_p0_t0_all_cases}(d)) and the $L^2-$norm of the gradient of the solution (see Figure \ref{Fig_p0_t0_all_cases}(c)) difference are significantly larger than using the homogeneous initial condition in the first few time steps. However, when these have stabilized after this transient time period, the deviations for the all Gaussian-shaped extensions of the initial condition are almost one order of magnitude lower than those for the homogeneously-zero initial condition. 
	
	Again, the quantifier $c^*(t)$, defined in Equation \eqref{Eq_c_star}, does not allow for a conclusion as clear as that provided by the $L^2$ and $H^1$-norm differences. It also contains the gradient of $u_P$. Hence, most graphs of $c^*(t)$ in Figure \ref{Fig_p0_t0_all_cases}(d) are above the graph for the homogeneously-zero initial condition (dashed curve), except for Option 3 and Option 4. There, the maximal and minimal difference between the presumed flux density in $(BVP_S)$ and the analytical flux $D\nabla u_P\cdot\boldsymbol{n}$ from $(BVP_P)$ are minimized. In other words, Option 3 and 4 select $(p_0, t_0)$ such that mostly $D\nabla u_P\cdot\boldsymbol{n}$ is close to $\phi(\boldsymbol{x}, t)$. In particular, Option 4 appears to be the best option, since $c^*(t)$ reaches the steady state fastest and with the smallest value. In the latter option one enforces that $\phi(\boldsymbol{x},0) = \phi_{sum}(0)$. However, Figure \ref{Fig_p0_t0_all_cases}(a) indicates that Option 4 is the worst choice among the Gaussian-shaped extensions, from the point of view of the $L^2$-norm distance. There is no clear distinction among the options from the point of view of the $H^1$-norm.
	
	\begin{figure}
		\centering
		\subfigure[$\|u_S(\boldsymbol{x},t) - u_P(\boldsymbol{x},t)\|_{L^2(\Omega\backslash\Omega_C)}$]{
			\includegraphics[width = 0.48\textwidth]{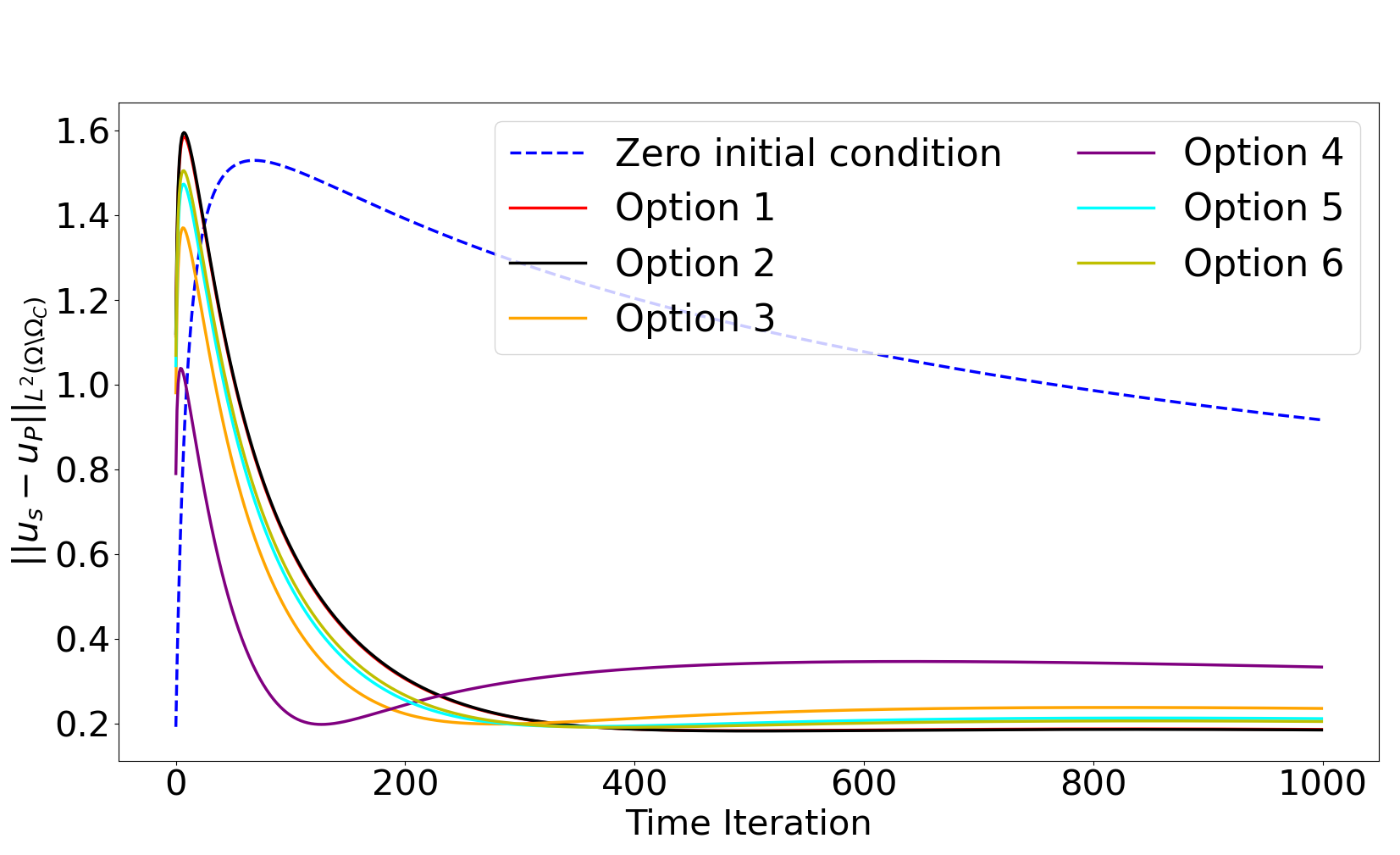}}
		\subfigure[$\|\nabla u_S(\boldsymbol{x},t) - \nabla u_P(\boldsymbol{x},t)\|_{L^2(\Omega\backslash\Omega_C)}$]{
			\includegraphics[width = 0.48\textwidth]{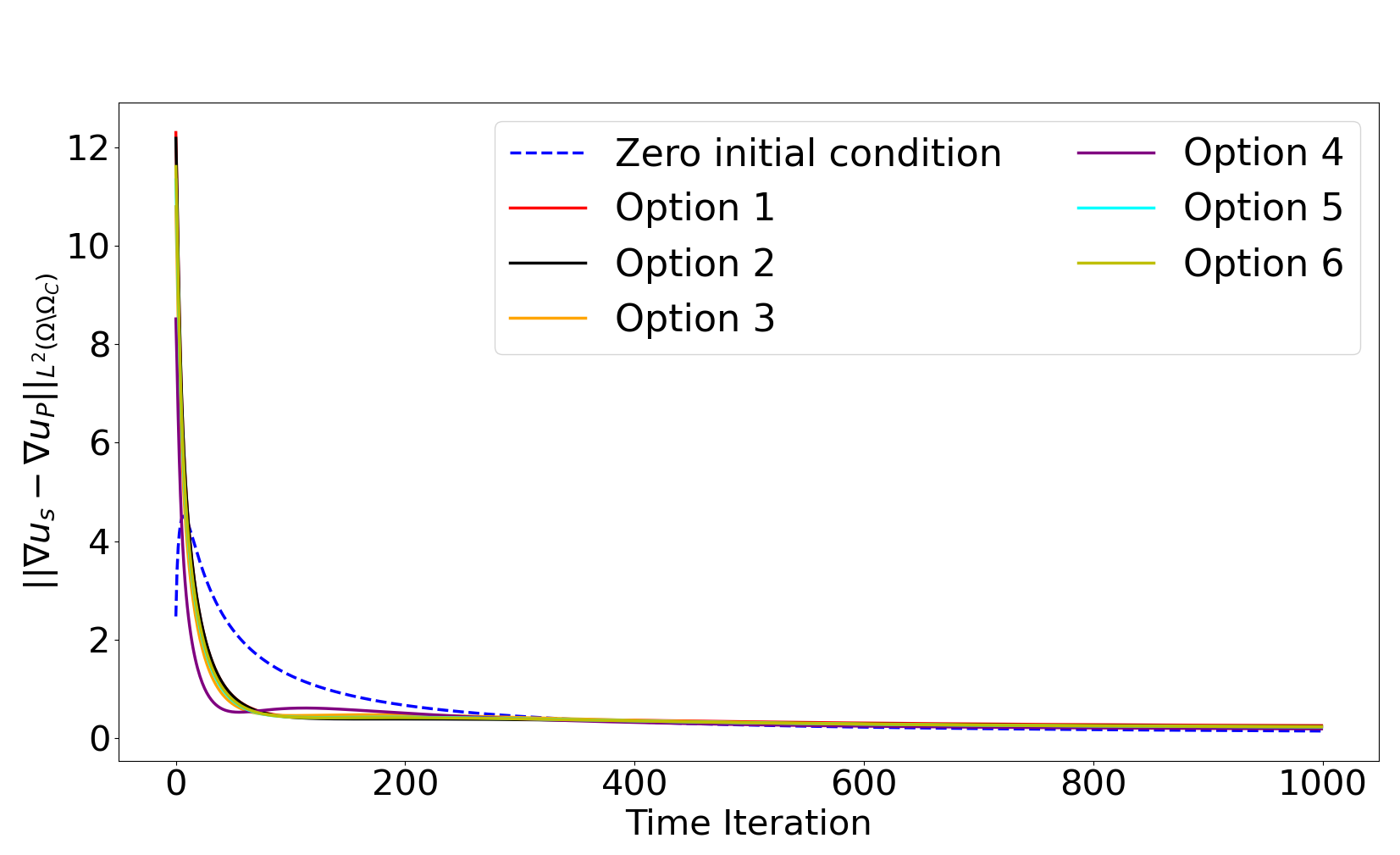}}
		\subfigure[$\|u_S(\boldsymbol{x},t) - u_P(\boldsymbol{x},t)\|_{H^1(\Omega\backslash\Omega_C)}$]{
			\includegraphics[width = 0.48\textwidth]{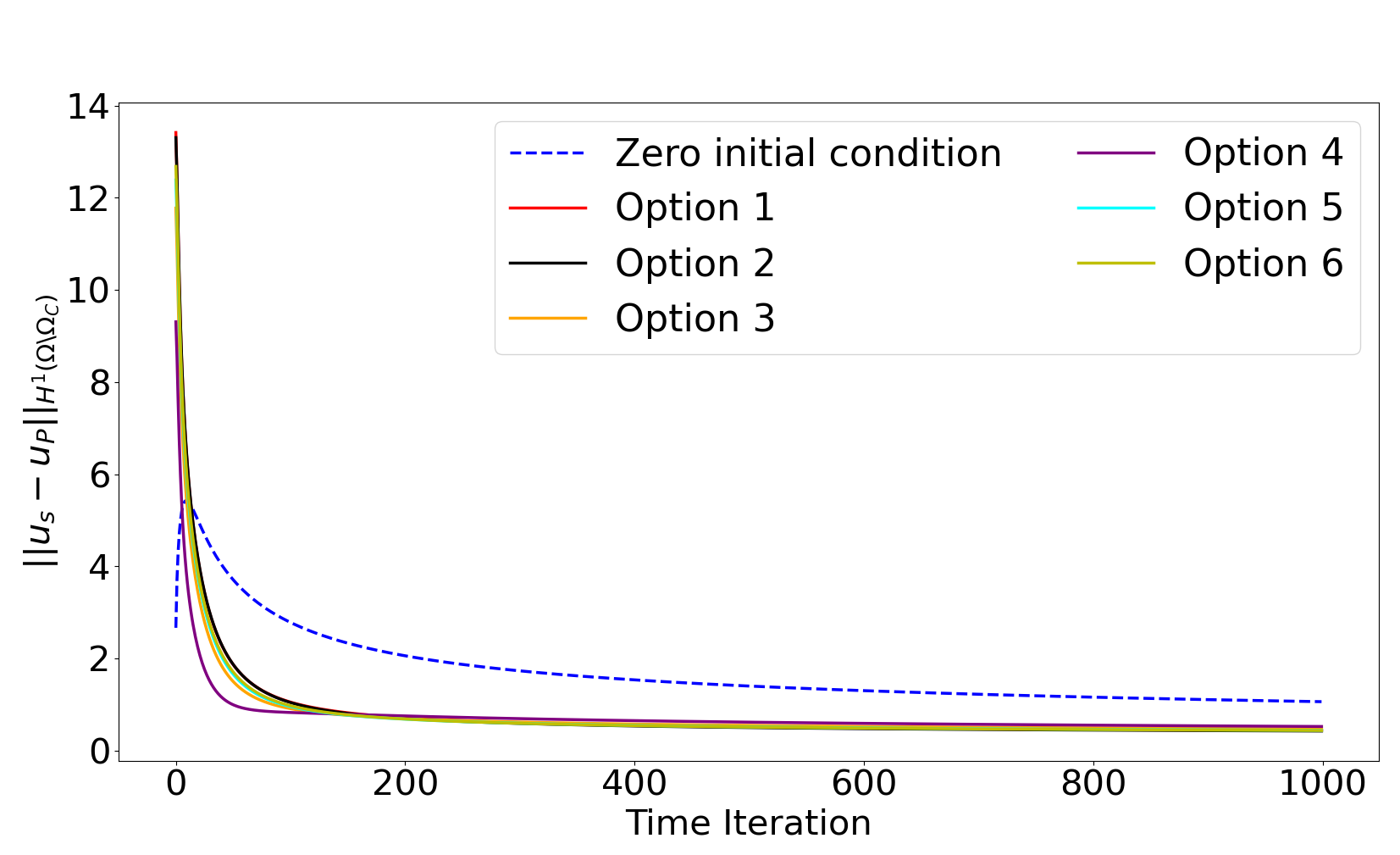}}
		\subfigure[$c^*(t)$]{
			\includegraphics[width = 0.48\textwidth]{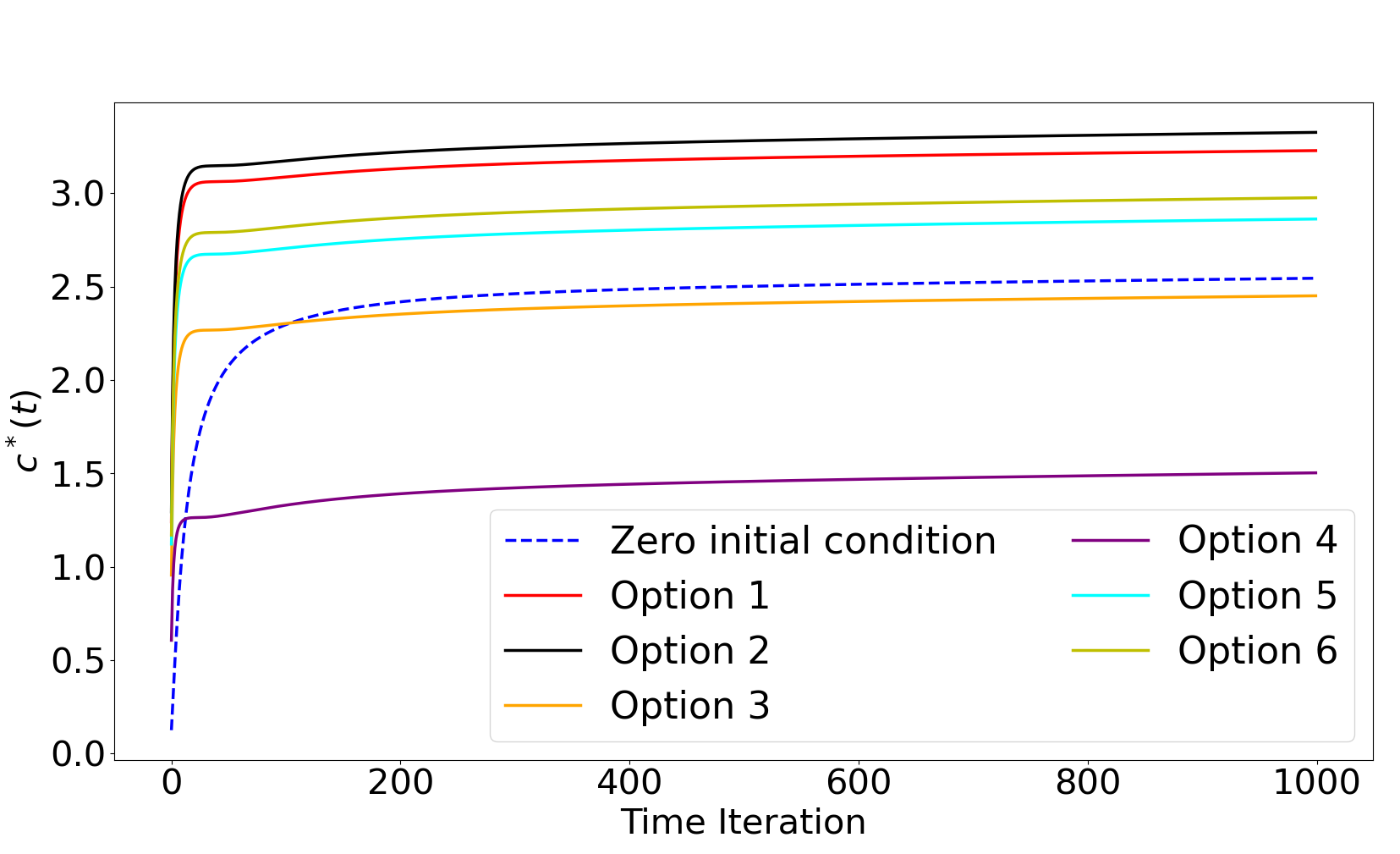}}
		\caption{Measures of quality of approximation between solutions as a function of time. The Gaussian-shape parameters $(p_0,t_0)$ have been computed according to the options listed in Table \ref{Tbl_opt_options} and the resulting values for $(p_0,t_0)$ have been used in the simulation ($D=0.1, \phi=1, R=\frac{1}{2}$). Blue dashed line represents the $(BVP_P)$ when $u_P(\boldsymbol{x},0) = 0$ for $\Omega$, and other coloured solid lines represent the $(BVP_P)$ when the initial condition is given by Equation (\ref{Eq_u0_dirac}). The curve for Option overlaps visually with that of Option 2 in Panel (a), (b) and (c)}.
		\label{Fig_p0_t0_all_cases}
	\end{figure}
	
	In view of Proposition \ref{Prop_condition}, Equation \eqref{eq:L2-norm differnce}, the behaviour of the $L^2$ norm is a delicate interplay between the flux difference over the boundary (measured by $c^*(t)$ and the difference of the gradient of the solutions. These two quantities are not independent. More analytic insight -- if obtainable -- is required to understand properly the apparent discrepancy between the $c^*$ measure and $L^2$-norm of difference of solutions.
	
	In general, the smaller the quantifier is, the better the option. However, again there seems to be no ``\textit{best}" option for all quantifiers considered. Again, it turns out that $c^*(t)$ has to be interpreted with greatest care.

	\section{Extending Nonzero environmental Initial Value}
	\label{Sec_Results_Inhomo}
	\noindent
	As observed in the previous results, discontinuity over $\partial\Omega_C$ caused by the Gaussian-shaped initial condition in the interior of the cell and the zero homogeneous initial condition in the rest of the computational domain, results in a significant difference between the two approaches in the quantifiers that involve the gradient of $u_P(\boldsymbol{x},t)$. This suggests to introduce an additional condition of continuity on the cell boundary. For zero initial condition this can never be achieved by the strictly positive Gaussian kernel. Therefore, we propose to alter the zero initial condition to be a positive constant, denoted by $C$ and to investigate the selection of $(p_0, t_0)$ further for such non-zero initial conditions. 
	
	In this section, we only consider the $L^1$-objective function from Option $1$ and $2$ given by Equation \eqref{Eq_p0_t0_Case_2}.
	We take the continuity constraint for the Gaussian-shaped extension on the boundary $\partial\Omega_C$, which amounts to taking 
	\begin{equation}
		\label{Eq_new_constraint_C}
		\displaystyle\frac{p_0}{4\pi Dt_0}\exp\{-\frac{R^2}{4Dt_0}\}=C
	\end{equation}
	as the new constraint for $(p_0, t_0)$ instead of Equation \eqref{Eq_p0_t0}. By doing this, we ensure that there exists no jump at the boundary of the cell, i.e. $\partial\Omega_C$. Of course, there may be a jump in flux. Note that for the options listed in Table \ref{Tbl_opt_options} that do not include constraint, the value of $(p_0, t_0)$ does not change. In Table \ref{Tbl_opt_options_new}, we present the value of $(p_0, t_0)$, for Option 2 with the new constraint (\ref{Eq_new_constraint_C}) with $C = 0.1, 10, 100$, respectively.
	\begin{table}\footnotesize
		\centering
		\caption{Computed $(p_0, t_0)$ from the optimization problem with indicated objective function and value-pair constraint. Initial condition is $u_0 = C$ in $\Omega\backslash\Omega_C$, with $C\in\{0.1, 10, 100\}$.  Optimization was performed by the function \texttt{optimize.minimize} of \texttt{Scipy} package (version $1.8.0$) in Python.}
		\begin{tabular}{p{0.5cm}<{\centering}p{5.5cm}<{\centering}p{4.5cm}<{\centering}p{2cm}<{\centering}p{2cm}<{\centering}}
			\toprule
			{\bf C} & {\bf Objective function} & {\bf Constraints} & {\bf Value of $p_0$} & {\bf Value of $t_0$} \\
			\midrule
			$0.1$ & \multirow{3}{*}{$ \int_0^T |\phi_{sum}(t) - \phi| dt$} & $\displaystyle 0.1 = \frac{p_0}{4\pi Dt_0}\exp\{-\frac{R^2}{4Dt_0}\}$ & $4.688$ & $0.107$ \\ 
			$10$ & &  $\displaystyle 10 = \frac{p_0}{4\pi Dt_0}\exp\{-\frac{R^2}{4Dt_0}\}$ & $53.422$ & $3.568$\\
			$100$ & &  $\displaystyle 100 = \frac{p_0}{4\pi Dt_0}\exp\{-\frac{R^2}{4Dt_0}\}$ & $9.194\times 10^5$ & $7.310\times10^2$ \\
			\bottomrule
		\end{tabular}
		\label{Tbl_opt_options_new}
	\end{table}

	Figure \ref{Fig_positive_constant_10} shows the results when $C = 10$ is chosen, that is, the initial condition is $10$ over the domain of the spatial exclusion approach. Compared with the results in Figure \ref{Fig_p0_t0_all_cases} when the initial condition is zero, the $H^1-$norm and the $L^2-$norm of the gradient of the difference between the solutions to the two approaches have reduced. The same holds for the quantity $c^*(t)$. All the subplots indicate that using of Gaussian distribution inside the cell as initial condition reduces the difference between the solutions to the two approaches. On the other hand, we note that Option $1$ (red curve in Figure \ref{Fig_positive_constant_10}) performs slightly better than Option 2 with the previous constraint (black curve). This can be attributed to the fact that the pair of $(p_0, t_0)$ computed without constraint in Option 1 results in the value $10.150$ approximately over the boundary of the cell. Hence, the constraint in Equation (\ref{Eq_new_constraint_C}) is almost satisfied (see the value computed in Table \ref{Tbl_opt_options_new}), while the flux difference on the cell boundary over time is minimized at the same time. When $C = 0.1$ is chosen, similar patterns in Figure \ref{Fig_positive_constant_01} appear compared to that the zero initial condition is used. However, Figure \ref{Fig_positive_constant_01}(b)-(d) show a better performance when the continuity of the cell boundary is guaranteed; see the black curve in each subfigure. Next to it, we also conduct the simulation with $C=100$, hence, the values of $(p_0, t_0)$ computed by Option 1 and 2 will still result in a significant difference between the two approaches, with respect to all the quantities. As a consequence, we categorize the simulations into two subgroups: (1) Option 1 and Option 2, and (2) homogeneous initial condition and Option 2 with the new constraint; see Figure \ref{Fig_positive_constant_100}. In all the subfigures, we observe that the results in Subgroup (1) are at least $15$ times larger that the results in Subgroup (2), which supports the necessity to guarantee the continuity of the initial condition, at the cell boundary. Moreover, it can be concluded that with the current parameter values, it is not needed to use a Gaussian-shaped extension. The homogeneous extension of $u_0=C$ by the same constant value in the cell's interior leads to solutions that are comparable in quality as those obtained from Gaussian-shaped extensions. 
	
	\begin{figure}[h!]
		\centering
		\subfigure[$\|u_S(\boldsymbol{x},t) - u_P(\boldsymbol{x},t)\|_{L^2(\Omega\backslash\Omega_C)}$]{\includegraphics[width = 0.48\textwidth]{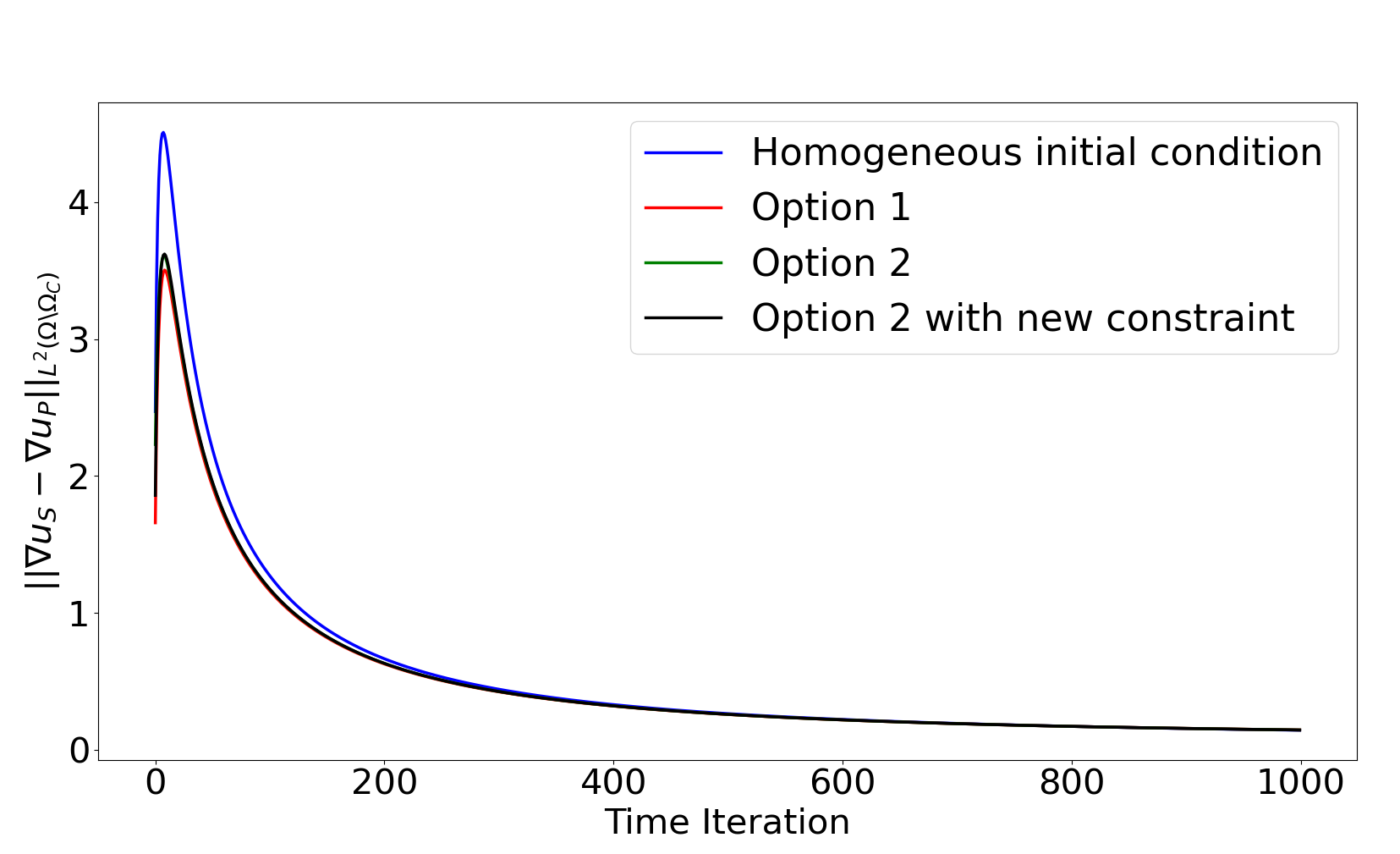}}
		\subfigure[$\|\nabla u_S(\boldsymbol{x},t) - \nabla u_P(\boldsymbol{x},t)\|_{L^2(\Omega\backslash\Omega_C)}$]{\includegraphics[width = 0.48\textwidth]{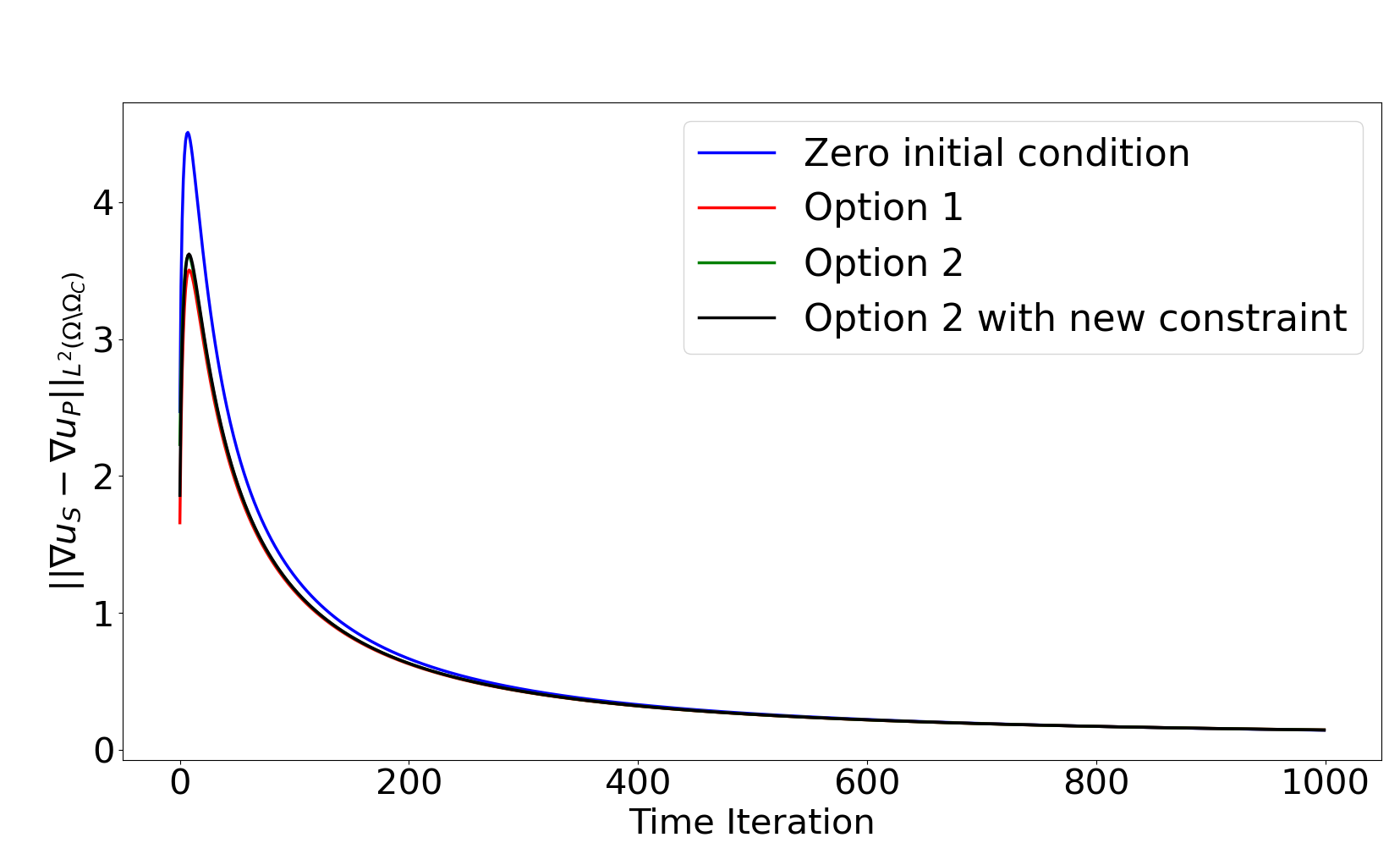}}
		\subfigure[$\|u_S(\boldsymbol{x},t) - u_P(\boldsymbol{x},t)\|_{H^1(\Omega\backslash\Omega_C)}$]{\includegraphics[width = 0.48\textwidth]{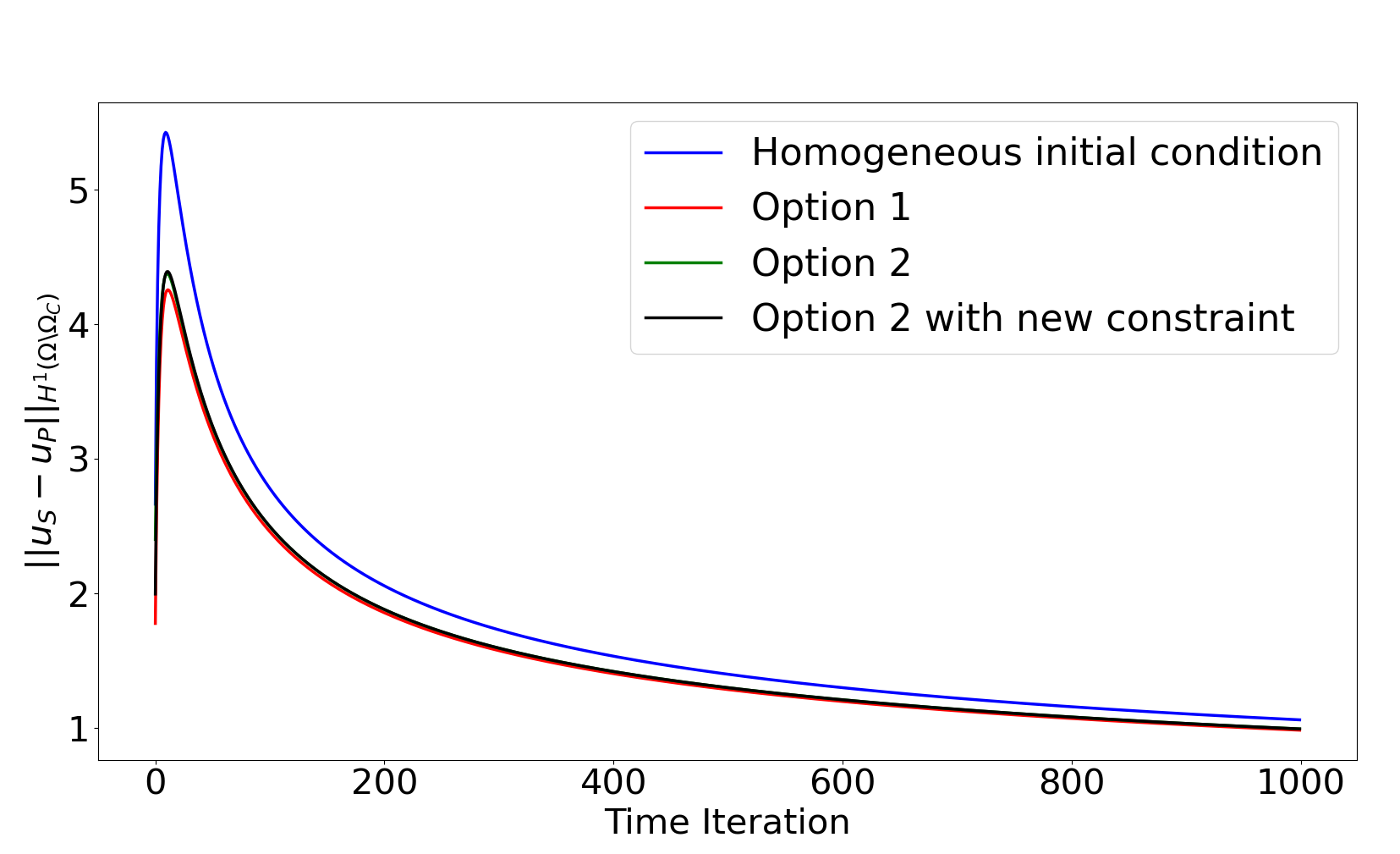}}
		\subfigure[$c^*(t)$]{\includegraphics[width = 0.48\textwidth]{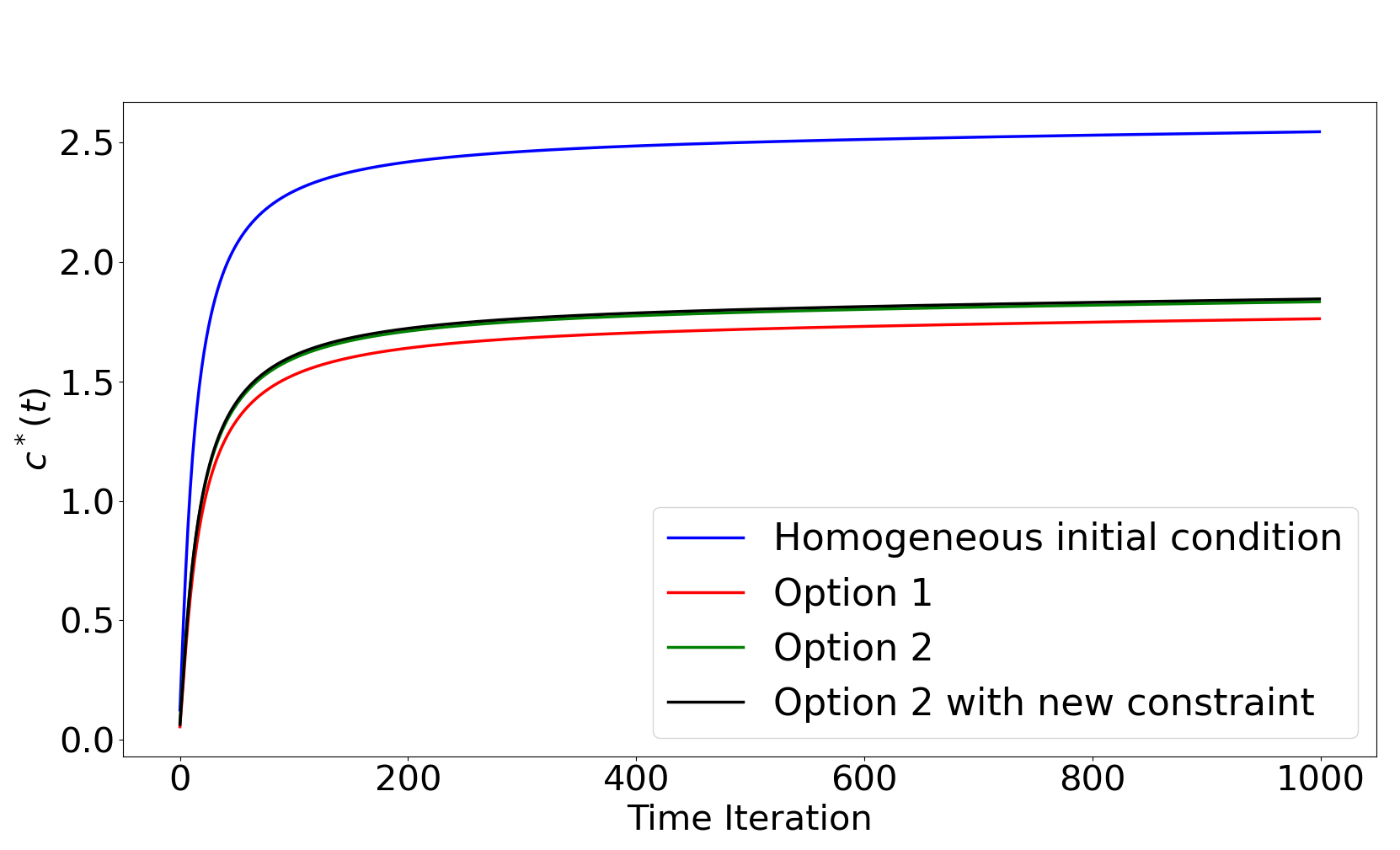}}
		\caption{Various norms of differences between the two approaches and $c^*(t)$ are shown. Here, we consider non-zero initial condition $u_0=C$ with $C = 10$. Shown are the error quantifiers for homogeneously extended initial condition $\bar{u}_0=C$ (blue)  and the Gaussian-shaped extension of the initial condition with $(p_0, t_0)$ computed from Option 1 (red) without value-pair constraint, Option 2 (green) with constraint (Equation \eqref{Eq_p0_t0}) in Table \ref{Tbl_opt_options} and Option 2 with the new continuity constraint (Equation \eqref{Eq_new_constraint_C}) in Table \ref{Tbl_opt_options_new}.}
		\label{Fig_positive_constant_10}
	\end{figure}
	\begin{figure}[h!]
		\centering
		\subfigure[$\|u_S(\boldsymbol{x},t) - u_P(\boldsymbol{x},t)\|_{L^2(\Omega\backslash\Omega_C)}$]{\includegraphics[width = 0.48\textwidth]{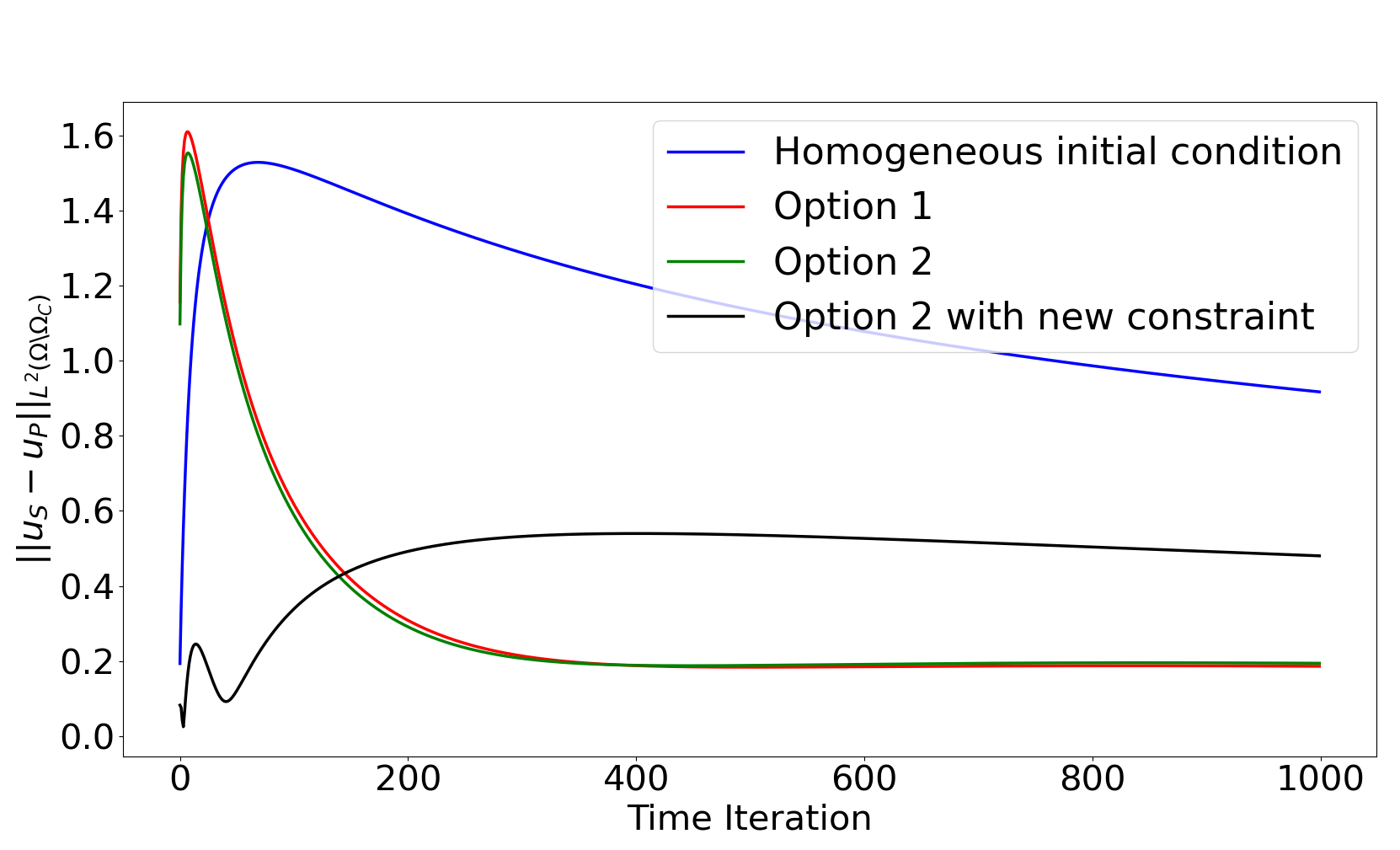}}
		\subfigure[$\|\nabla u_S(\boldsymbol{x},t) - \nabla u_P(\boldsymbol{x},t)\|_{L^2(\Omega\backslash\Omega_C)}$]{\includegraphics[width = 0.48\textwidth]{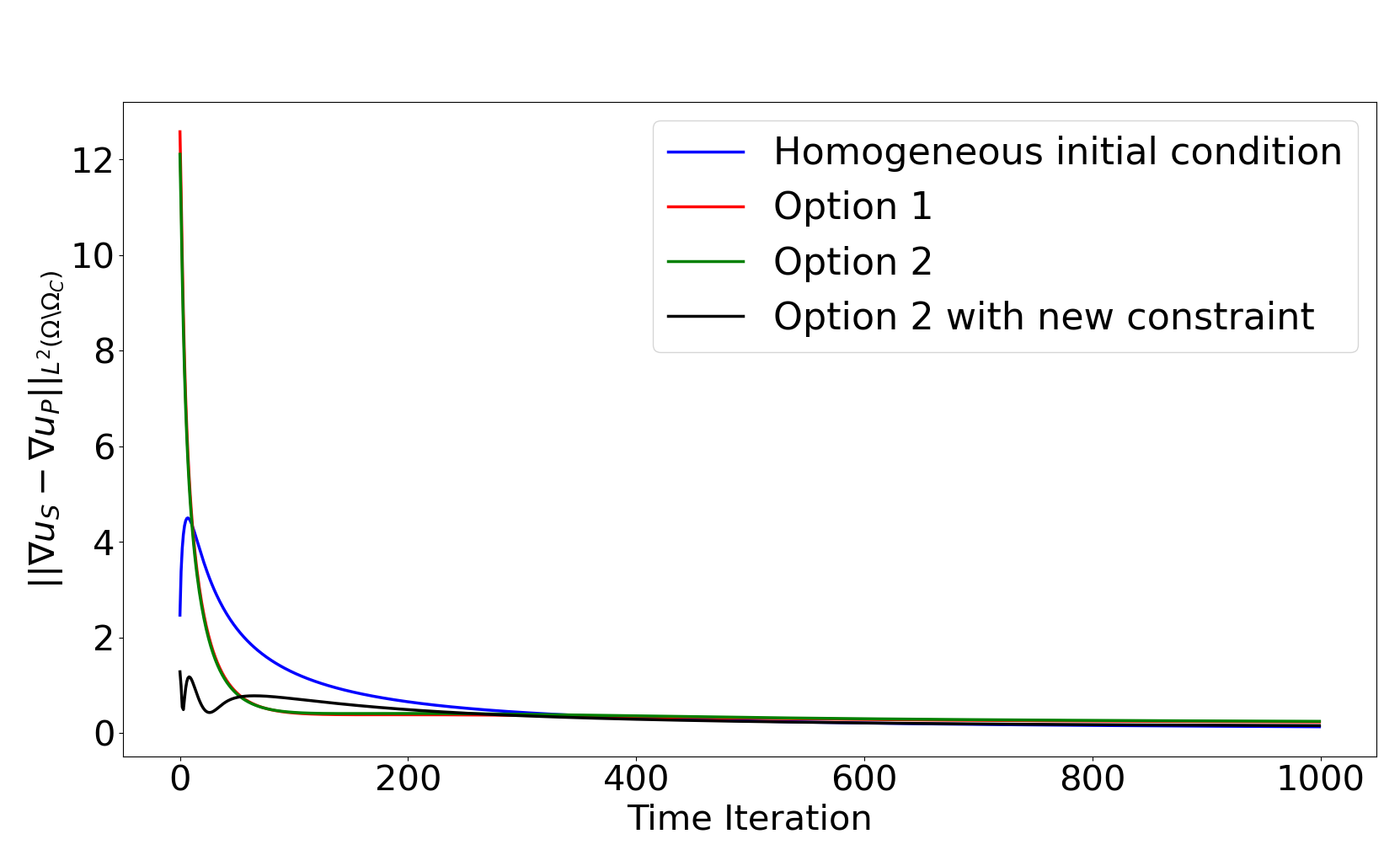}}
		\subfigure[$\|u_S(\boldsymbol{x},t) - u_P(\boldsymbol{x},t)\|_{H^1(\Omega\backslash\Omega_C)}$]{\includegraphics[width = 0.48\textwidth]{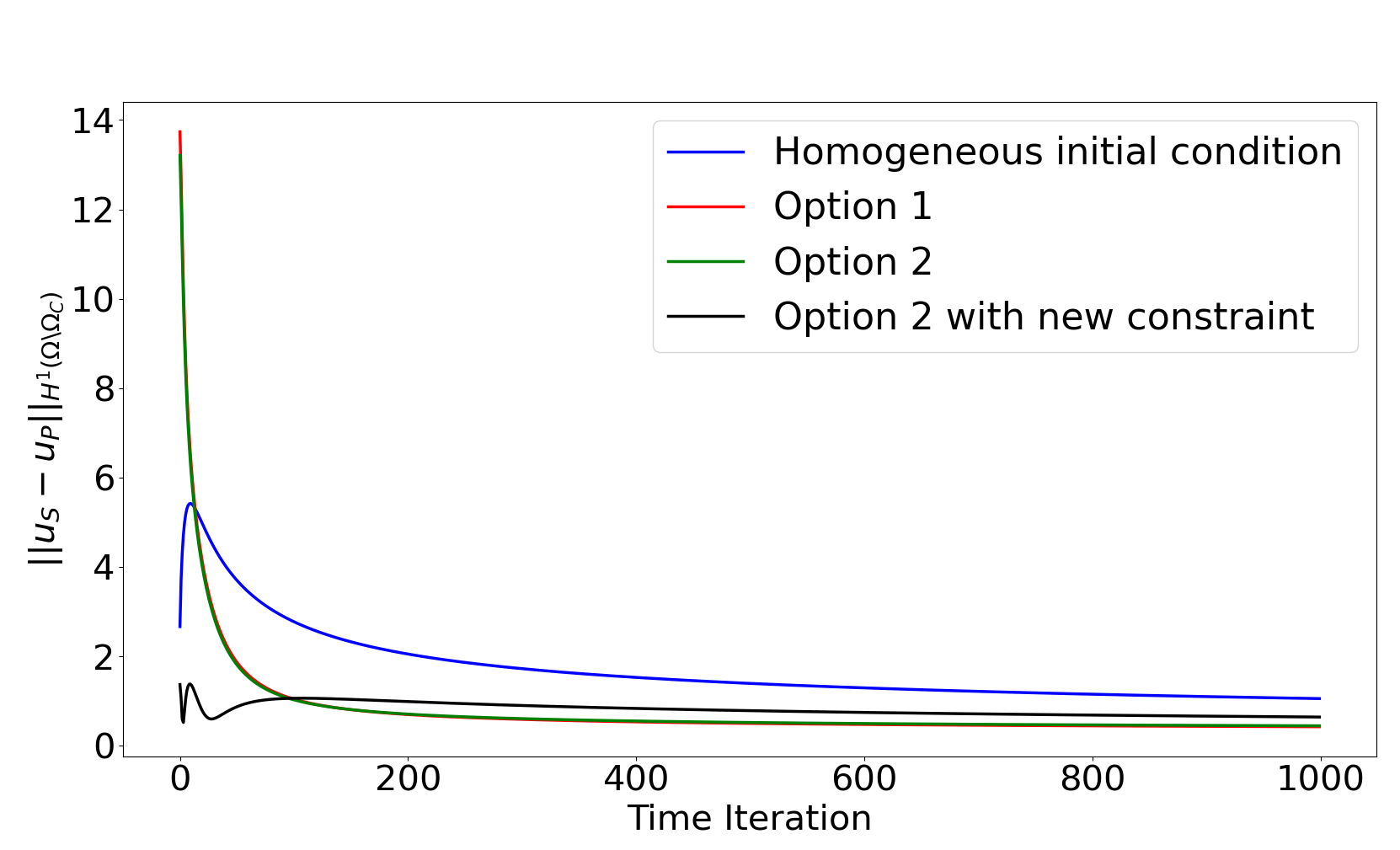}}
		\subfigure[$c^*(t)$]{\includegraphics[width = 0.48\textwidth]{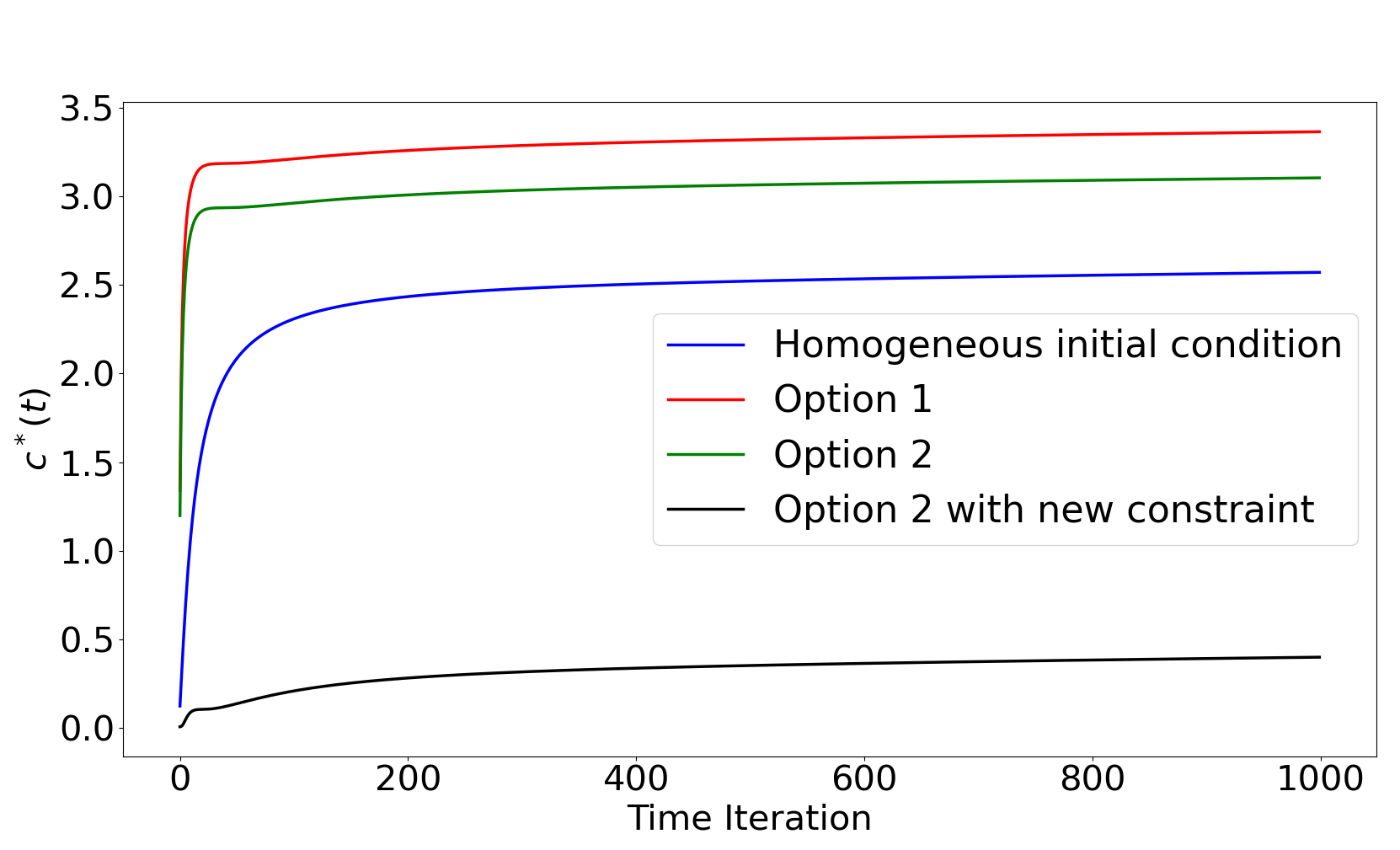}}
		\caption{Various norms of differences between the two approaches and $c^*(t)$ are shown. Here, we consider non-zero initial condition $u_0=C$ with $C = 0.1$. Shown are the error quantifiers for homogeneously extended initial condition $\bar{u}_0=C$ (blue)  and the Gaussian-shaped extension of the initial condition with $(p_0, t_0)$ computed from Option 1 (red) without value-pair constraint, Option 2 (green) with constraint (Equation \eqref{Eq_p0_t0}) in Table \ref{Tbl_opt_options} and Option 2 with the new continuity constraint (Equation \eqref{Eq_new_constraint_C}) in Table \ref{Tbl_opt_options_new}.}
		\label{Fig_positive_constant_01}
	\end{figure}
	\begin{figure}[h!]
		\centering
		\subfigure[$\|u_S(\boldsymbol{x},t) - u_P(\boldsymbol{x},t)\|_{L^2(\Omega\backslash\Omega_C)}$]{\includegraphics[width = 0.48\textwidth]{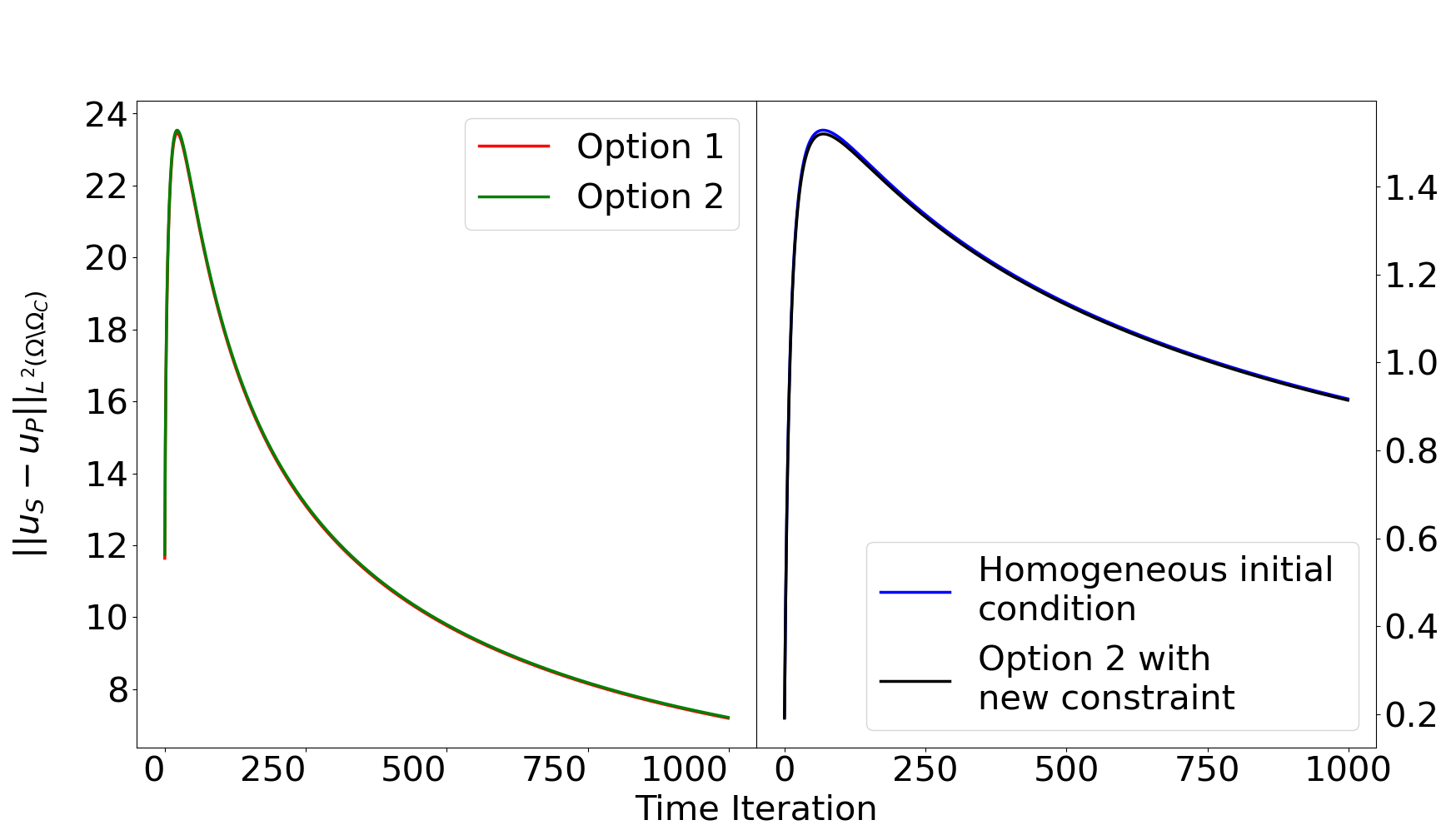}}
		\subfigure[$\|\nabla u_S(\boldsymbol{x},t) - \nabla u_P(\boldsymbol{x},t)\|_{L^2(\Omega\backslash\Omega_C)}$]{\includegraphics[width = 0.48\textwidth]{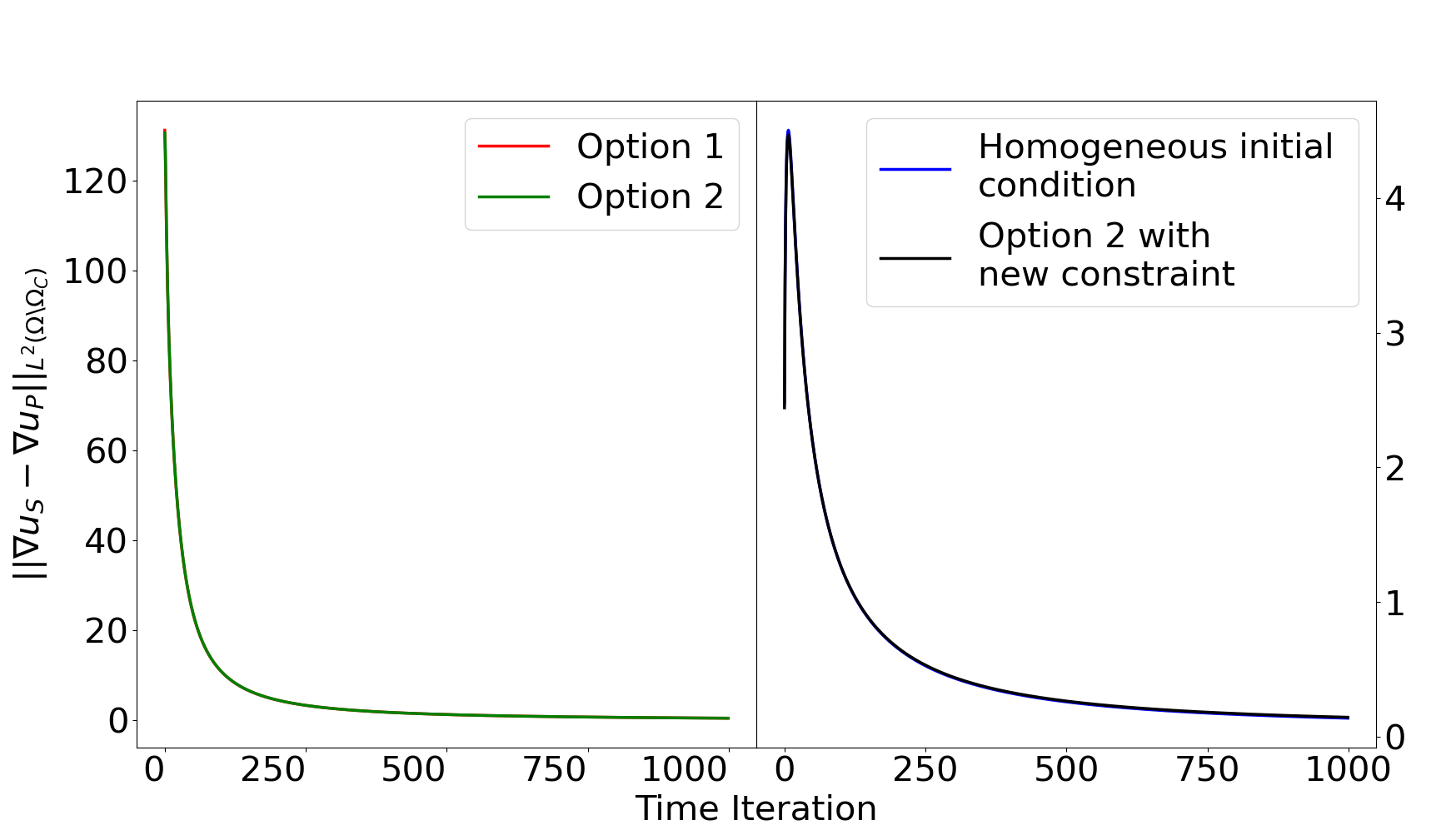}}
		\subfigure[$\|u_S(\boldsymbol{x},t) - u_P(\boldsymbol{x},t)\|_{H^1(\Omega\backslash\Omega_C)}$]{\includegraphics[width = 0.48\textwidth]{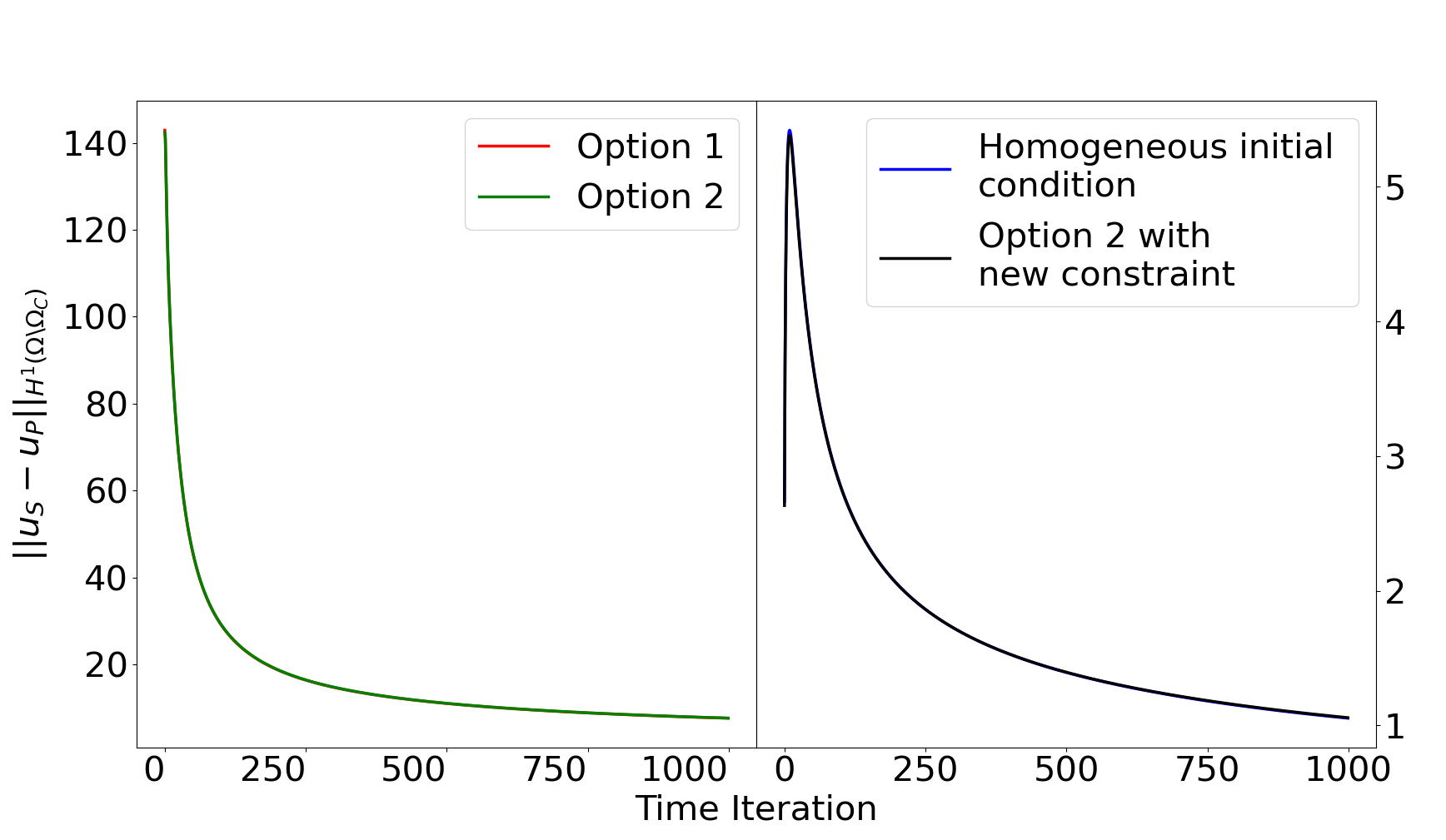}}
		\subfigure[$c^*(t)$]{\includegraphics[width = 0.48\textwidth]{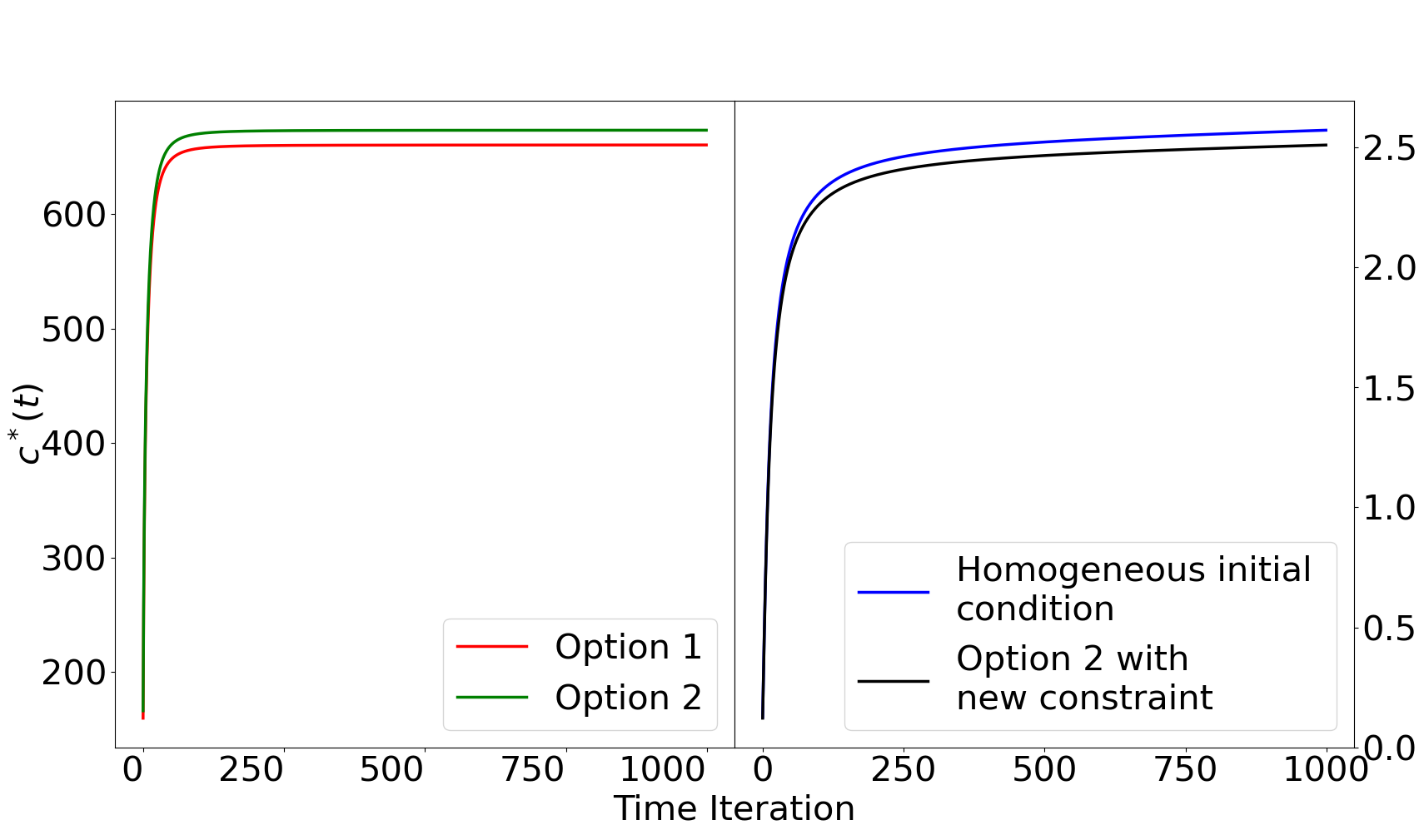}}
		\caption{Various norms of differences between the two approaches and $c^*(t)$ are shown. Here, we consider non-zero initial condition $u_0=C$ with $C = 100$. Shown are the error quantifiers for homogeneously extended initial condition $\bar{u}_0=C$ (blue) and the Gaussian-shaped extension of the initial condition with $(p_0, t_0)$ computed from Option 1 (red) without value-pair constraint, Option 2 (green) with constraint (Equation \eqref{Eq_p0_t0}) in Table \ref{Tbl_opt_options} and Option 2 with the new continuity constraint (Equation \eqref{Eq_new_constraint_C}) in Table \ref{Tbl_opt_options_new}.}
		\label{Fig_positive_constant_100}
	\end{figure}
	
	For different values of $C$ in $\{0.1,  10, 100\}$ there is quite a different total amount of mass in the environment. A relative measure of error is therefore appropriate in order to compare the different cases, i.e. initial conditions $u_0=C$ vary with $D=0.1$ fixed. We propose the ratio of the total flux difference over the cell boundary and the total mass in $\Omega\backslash\Omega_C$. By the Cauchy-Schwarz Inequality one can bound this relative quantifier by
	\begin{align}\displaystyle
		& \frac{\int_0^t\int_{\partial\Omega_C}|\phi(\boldsymbol{x}, s) - u_P\nabla D\cdot n|d\Gamma ds}{\|u_S(\boldsymbol{x}, t)\|_{L^1(\Omega\backslash\Omega_C)}} \nonumber \\
		& \qquad\quad \leq\ \frac{\sqrt{|\partial\Omega_C|}\int_0^t \|\phi(\boldsymbol{x}, s) - D\nabla u_P\cdot\boldsymbol{n}ds\|_{L^2(\partial\Omega_C)}ds}{\|u_S(\boldsymbol{x}, t)\|_{L^1(\Omega\backslash\Omega_C)}} 
		\ =:\ r.e(t). \label{Eq_relative_c_star}
	\end{align}
	The latter quantity $r.e(t)$ is more convenient because of its relation to $c^*(t)$.

	Note that the denominator in expression \eqref{Eq_relative_c_star} can be computed analytically as 
	\[
	\|u_S(\boldsymbol{x}, t)\|_{L^1(\Omega\backslash\Omega_C)} = \bigl|\Omega\backslash\Omega_C\bigr|C+2\pi R\phi t,
	\]
	where $|\Omega\backslash\Omega_C|$ is the area of the domain in the spatial exclusion approach, and $\phi$ is the constant flux density over the boundary. Figure \ref{Fig_relat_error_c_star} shows the relative errors for various value of $C$. All the relative errors are less than $8\%$.
	\begin{figure}[h!]
		\centering
		\includegraphics[width = 0.75\textwidth]{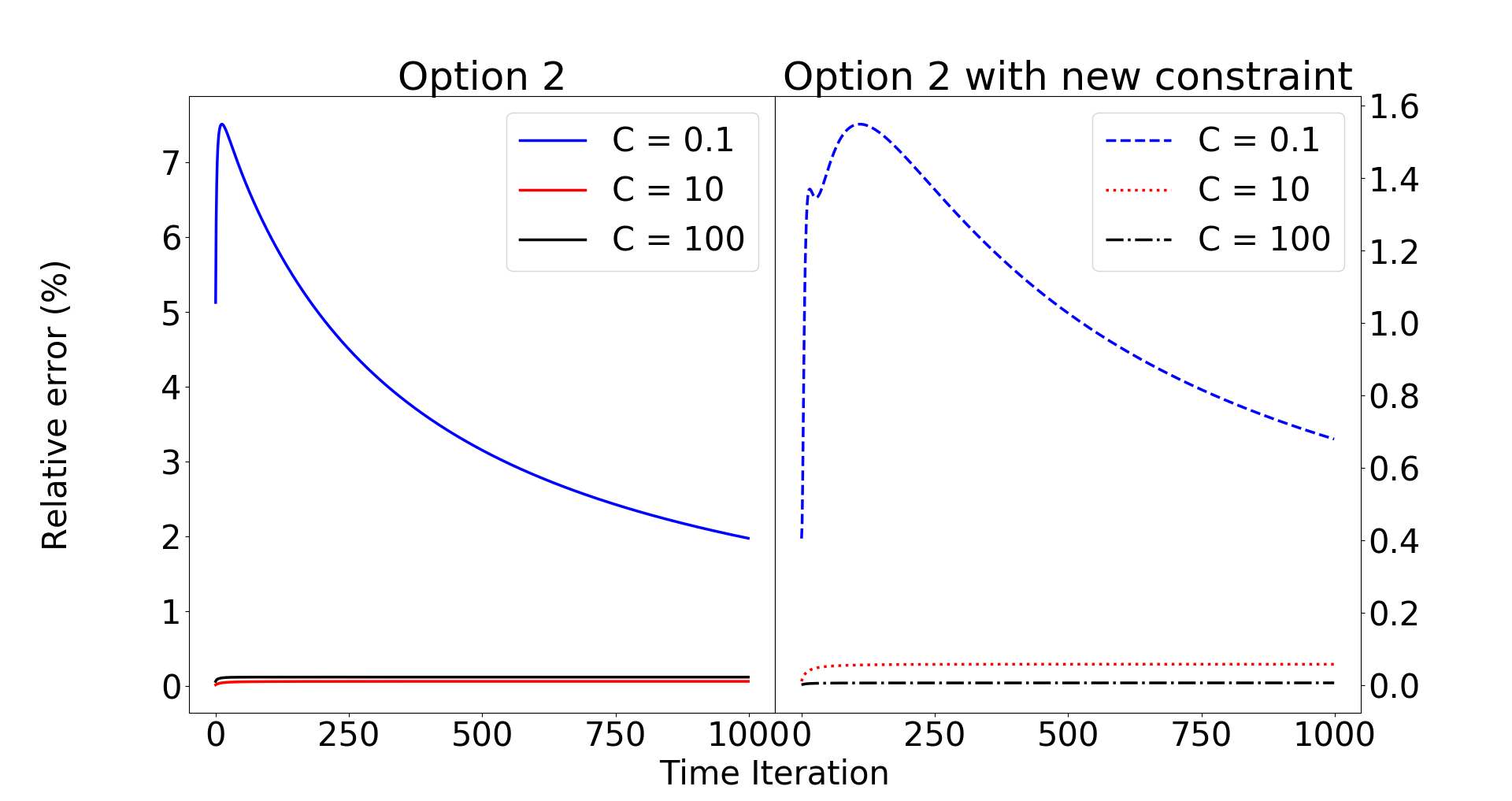}
		\caption{The relative errors defined in Equation (\ref{Eq_relative_c_star}) are shown in the plot, with $C = 0.1, 10, 100$, respectively. Here, we focus on Option 2 given in Table \ref{Tbl_opt_options} and Option 2 with new constraint in Table \ref{Tbl_opt_options_new}.}
		\label{Fig_relat_error_c_star}
	\end{figure}
	
	\section{Conclusions and Discussion}\label{Sec_Conclu}
	\noindent
	In this paper, we investigated conditions needed to replace the spatial exclusion model for mass emitting objects (`cells') with diffusion in the environment by a point source model for the sake of computational efficiency or theoretical convenience  in such a way that solutions to the latter still approximate well the corresponding solutions to the former in $L^2$- and $H^1$-norms. Such an approximation is convenient when many cells are considered in the computational domain and is particularly interesting when these cells will be moving. Here, we considered the simplest case of a circular cell, fixed at its initial location, that keeps releasing the compounds at a constant rate. This compound is not taken up by the cell, nor is there a reaction in the environment. It stays in the environment, diffusing, forever after release. This simplified situation has been studied first to identify sources of error in the approximation that stem from the replacement of a spatial (circular) objects by a point source only.
	\vskip 0.2cm
	
	\noindent We identified two main sources of error: 
	\begin{enumerate}
		\item [(1)] A {\bf systematic time delay} between the solution of the spatial exclusion model and of the point source model, where the latter is lagging behind the former. It originates in inappropriateness of the chosen extension of the initial condition of the spatial exclusion model on the environment to the whole domain, on which the point source model is defined. This time delay occurs already for a single cell in the domain. It decreases with increasing diffusivity. The delay can be turned into a transient effect e.g. by choosing an extension by means of a Gaussian-shaped function on the part of the domain that corresponds to the cell's interior.
		
		\item [(1)] {\bf Absence of reflection} in the point source model. In this model, diffusing compounds will move simply through the parts of the domain that constitute the cells' interior in the spatial exclusion model. In that model, these would have reflected on the cell boundary and hence stayed in the environment. On the long run, after many reflections, the two solutions should become different. 
	\end{enumerate} 
	
	First, we provided an analytical condition such that the solutions to the two models are consistent, i.e. that they are equal on their shared domain of definition, the cells' environment. Although this condition can never be satisfied exactly, it provides clues how to minimize the deviation between the two approaches: minimizing the deviation between the prescribed and generated flux density over the cell boundary. Hence, the choice for the quantity $c^*(t)$, introduced in \citet{HMEvers2015} as a measure for the quality of approximation, is motivated. It is compared to the use of $L^2$- and $H^1$-norms over the environmental domain, which are more expensive to compute. 
	
	There are infinitely many ways of extending the initial condition on the environment to the whole domain. We chose the Gaussian-shaped extension, since it already `fits' the diffusion process. Moreover, the preliminary results showed that all norms of differences between the two approaches converged towards a stable value. Thus, we were inspired to extend by means of the fundamental solution of the diffusion equation, with a chosen  amplitude and variance as parameters. Numerical results show the significant improvements in all the quantities, such as the local norm differences and $c^*(t)$ and the disappearance of the time delay after a transient. 
	
	This raised the question how to select the value of the two parameters. We proposed and compared multiple options, which all boil down to solving single-objective optimization problems. There seems to be no ``{\it best}" option. Moreover, a good choice may depend on which quantity is preferred for error quantification. When a non-zero initial condition is utilized, approximation quality improves when the extended initial condition has continuity of the flux over the boundary of the cells. This suggests to consider a multi-objective optimization problem in future work, where deviation from this continuity of flux in the initial condition is balanced with a measure of error of the flux over this boundary over time. This is a topic for future research.
	
	The examination of the case of multiple circular cells of identical radius in the domain indicated the effects of absence of reflection on the quality of approximation. This effect is also present in the situation of a single cell, but then it is less apparent. When there are multiple cells in the domain and their mutual distance is small compared to their radius, the solutions tend to deviate more. In this paper, we focused on removing the systematic time delay by proper extension of the initial condition. Further investigation is required what conditions can be imposed such that this source of error remains within acceptable tolerance, e.g. on this minimal distance between objects in comparison to their diameter and the diffusion constant, and the length of the simulation interval.
	\vskip 0.2cm
	
	For forthcoming work, there are various possibilities of interest. In addition to investigating the interactions between multiple cells and related conditions mentioned above, one can think of considering the prescription of inhomogeneous flux density over the cell boundary in space and possibly also over time. Moreover, cells may also take up compound from the environment. In the current paper, only circular cells have been considered, of equal radius, as an ideal study case. However, in applications, the cells will have various shapes. This can be resolved in a two-step approach. A spatial exclusion model with flux over such non-circular shaped boundary may first be approximated by such a model with inhomogeneous flux density over a circular-shaped boundary. The latter may then be approximated by a suitable point source model. This last type of approximation will be considered in a forthcoming paper. In the first type of approximation, we anticipate error caused by `curvature effects'. The curvature of the cell boundary influences the overall diffusion of the chemicals in case of reflection boundary conditions. In particular, sharp corners tend to confine diffusing matter and therefore may need particular attention, in a sense similar to `absence of reflection' as source of error mentioned above.
	
	In summary, this paper provided practical conditions to obtain an acceptable approximation of the solution of a spatial exclusion model by means of a point source model. We suggested multiple options to implement conditions numerically. We discussed how to select the two shape-parameters of a Gaussian-shaped extension of the initial condition of the spatial exclusion model in a practical setting. We provide insights into deviations that will be introduced by transferring a spatial exclusion model with diffusion to a point source model and how two maximize the consistency between the two solutions in their common spatial domain, over the time interval of simulation.

	\newpage
	\bibliographystyle{abbrvnat}
	\bibliography{Diffusion_hole_dirac_delta}

	\begin{appendix}
		\section*{Appendices}\label{Sec_appendix}
		\section{Solution Consistency between Two Approaches}\label{Sec_proposition_pf}
		\noindent
		It is out of interest to compare the two approaches, regarding the consistency of the solutions. We recall Proposition \ref{Prop_condition}, which shows the essential and necessary condition to obtain consistent solutions:
		\prop*
		\begin{proof}
			Equation \eqref{eq:L2-norm differnce} is obtained as part of the proof of sufficiency of the condition $D\nabla u_P(\boldsymbol{x},t)\cdot\boldsymbol{n} = \phi(\boldsymbol{x},t)$ (a.e.). To that end, we work on the weak forms (see Section \ref{Sec_sol_concept})  of the two boundary value problems. 
			Note that $u_S(\boldsymbol{x},t)$ and $u_P(\boldsymbol{x},t)$ are defined on different spatial domains. We subtract the two equations in their weak forms, which yields
			\begin{align*}
				&\int_{\Omega\backslash\bar{\Omega}_C}\frac{\partial u_S(\boldsymbol{x},t)}{\partial t}v_1(\boldsymbol{x},t) - \frac{\partial u_P(\boldsymbol{x},t)}{\partial t}v_2(\boldsymbol{x},t) d\Omega - \int_{\Omega_C}\frac{\partial u_P(\boldsymbol{x},t)}{\partial t}v_2(\boldsymbol{x},t)d\Omega + \int_{\partial\Omega_C}\phi(\boldsymbol{x},t)v_1(\boldsymbol{x},t)d\Gamma \\
				&+ \int_{\Omega\backslash\bar{\Omega}_C}D\nabla u_S(\boldsymbol{x},t)\nabla v_1(\boldsymbol{x},t) - D\nabla u_P(\boldsymbol{x},t)\nabla v_2(\boldsymbol{x},t)d\Omega - \int_{\Omega_C}D\nabla u_P(\boldsymbol{x},t)\nabla v_2(\boldsymbol{x},t)d\Omega \\
				&= -\int_{\Omega}\Phi(\boldsymbol{x},t)\delta(\boldsymbol{x} - \boldsymbol{x}_c)v_2(\boldsymbol{x},t)d\Omega.
			\end{align*}
			After some simplifications (applying Gauss Theorem and boundary conditions) and substituting $\displaystyle\frac{\partial u_P(\boldsymbol{x},t)}{\partial t} = D\Delta u_P + \Phi(\boldsymbol{x},t)\delta(\boldsymbol{x}-\boldsymbol{x}_c)$ in $\Omega_C$, we obtain
			\begin{align*}
				&\int_{\Omega\backslash\bar{\Omega}_C}\frac{\partial u_S(\boldsymbol{x},t)}{\partial t}v_1(\boldsymbol{x},t) - \frac{\partial u_P(\boldsymbol{x},t)}{\partial t}v_2(\boldsymbol{x},t) d\Omega + D\int_{\Omega\backslash\bar{\Omega}_C} \nabla u_S(\boldsymbol{x},t)\nabla v_1(\boldsymbol{x},t) - \nabla u_P(\boldsymbol{x},t)\nabla v_2(\boldsymbol{x},t) d\Omega \\
				&= \int_{\partial\Omega_C}D\nabla u_P(\boldsymbol{x},t)\cdot\boldsymbol{n} v_2(\boldsymbol{x},t)-\phi(\boldsymbol{x},t)v_1(\boldsymbol{x},t)d\Gamma.
			\end{align*}
			As $v_1(\boldsymbol{x},t)$ and $v_2(\boldsymbol{x},t)$ are test functions in Hilbert space, we select $v_1(\boldsymbol{x},t) = v_2(\boldsymbol{x},t) = w(\boldsymbol{x},t):= u_S(\boldsymbol{x},t) - u_P(\boldsymbol{x},t)$ in $\Omega\backslash\bar{\Omega}_C$. Then the above equation becomes 
			\begin{equation}
				\label{Eq_w}
				\begin{aligned}
					&\int_{\Omega\backslash\bar{\Omega}_C}\frac{\partial w(\boldsymbol{x},t)}{\partial t}w(\boldsymbol{x},t) d\Omega + \int_{\Omega\backslash\bar{\Omega}_C} D\|\nabla w(\boldsymbol{x},t)\|^2 d\Omega = \int_{\partial\Omega_C} (D\nabla u_P(\boldsymbol{x},t)\cdot\boldsymbol{n} - \phi(\boldsymbol{x},t))w(\boldsymbol{x},t) d\Gamma\\
					\Rightarrow & \frac{\partial}{\partial t}\int_{\Omega\backslash\bar{\Omega}_C}\frac{1}{2} \|w(\boldsymbol{x},t)\|^2d\Omega + \int_{\Omega\backslash\bar{\Omega}_C} D\|\nabla w(\boldsymbol{x},t)\|^2 d\Omega = \int_{\partial\Omega_C} (D\nabla u_P(\boldsymbol{x},t)\cdot\boldsymbol{n} - \phi(\boldsymbol{x},t))w(\boldsymbol{x},t) d\Gamma.
				\end{aligned}
			\end{equation}
			The latter equation is precisely the desired expression \eqref{eq:L2-norm differnce}.
			Integrating both side with respect to time $t$ from $t = 0$ to any definite time $t = T$ and using that initial conditions are equal on $\Omega\setminus\Omega_C$ yields
			\begin{equation}
				\label{Eq_w_time_int}
				\frac{1}{2}\|w(\boldsymbol{x},t)\|^2_{L^2(\Omega\backslash\bar{\Omega}_C)} + D\int_0^T\|\nabla w(\boldsymbol{x},t)\|^2_{L^2(\Omega\backslash\bar{\Omega}_C)}dt = \int_0^T \int_{\partial\Omega_C}(D\nabla u_P(\boldsymbol{x},t)\cdot\boldsymbol{n} - \phi(\boldsymbol{x},t))w(\boldsymbol{x},t)d\Gamma dt
			\end{equation}
			From Equation (\ref{Eq_w_time_int}), if $D\nabla u_P(\boldsymbol{x},t)\cdot\boldsymbol{n} = \phi(\boldsymbol{x},t)$ over the boundary of the spatial exclusion $\partial\Omega_C$, then the right-hand side becomes zero, subsequently, Equation (\ref{Eq_w_time_int}) only holds when $\|w(\boldsymbol{x}, t)\|^2_{L^2(\Omega\backslash\bar{\Omega}_C)} = 0 = \|\nabla w(\boldsymbol{x}, t)\|^2_{L^2(\Omega\backslash\bar{\Omega}_C)}$, which implies $$w(\boldsymbol{x}, t)= u_S(\boldsymbol{x}, t) - u_P(\boldsymbol{x}, t) = 0,$$ 
			that is, the solutions to the two approaches are consistent in $\Omega\backslash\bar{\Omega}_C$.

			For the other statement, we consider the weak form of both approaches in $(WF_S)$ and $(WF_P)$ and let us take a test function $v_1(\boldsymbol{x},t)$ on $\Omega\setminus\Omega_C \times \R_+$. It can be extended to a test function on $v_2$ on $\Omega\times \R_+$,  which we shall denote by the same symbol if no confusion can arise. Then we obtain 
			\begin{align*}
				&\int_{\Omega\backslash\bar{\Omega}_C}\frac{\partial u_S(\boldsymbol{x},t)}{\partial t}v(\boldsymbol{x},t)d\Omega + \int_{\Omega\backslash\bar{\Omega}_C}D\nabla u_S(\boldsymbol{x},t)\nabla v(\boldsymbol{x},t)d\Omega - \int_{\partial\Omega_C}\phi(\boldsymbol{x},t)v(\boldsymbol{x},t)d\Gamma\\
				= &\int_{\Omega\backslash\bar{\Omega_C}}\frac{\partial u_P(\boldsymbol{x},t)}{\partial t}v(\boldsymbol{x},t)d\Omega+\int_{\Omega_C}\frac{\partial u_P(\boldsymbol{x},t)}{\partial t}v(\boldsymbol{x},t)d\Omega + \int_{\Omega\backslash\bar{\Omega_C}}D\nabla u_P(\boldsymbol{x},t)\nabla v(\boldsymbol{x},t)d\Omega\\
				&+\int_{{\Omega_C}}D\nabla u_P(\boldsymbol{x},t)\nabla v(\boldsymbol{x},t)d\Omega-\int_\Omega \Phi(\boldsymbol{x},t)\delta(\boldsymbol{x}-\boldsymbol{x}_c)v(\boldsymbol{x},t)d\Omega.
			\end{align*}
			First, we start with proving that, if $u_S(\boldsymbol{x},t) = u_P(\boldsymbol{x},t)$, then $\phi(\boldsymbol{x},t) - D\nabla u_P(\boldsymbol{x},t)\cdot\boldsymbol{n}= 0, \mbox{ on $\partial\Omega_C$}$, where $
			\boldsymbol{n}$ is the unit norm vector pointing towards the centre of $\Omega_C$. Hence, given that the solutions are equal, the above equation can be simplified and yields 
			\begin{align*}
				&-\int_{\partial\Omega_C}\phi(\boldsymbol{x},t)v(\boldsymbol{x},t)d\Gamma + \int_\Omega\Phi(\boldsymbol{x},t)\delta(\boldsymbol{x}-\boldsymbol{x}_c)v(\boldsymbol{x},t)d\Omega-\int_{\Omega_C}\frac{\partial u_P(\boldsymbol{x},t)}{\partial t}v(\boldsymbol{x},t)d\Omega\\
				&-\int_{{\Omega_C}}D\nabla u_P(\boldsymbol{x},t)\nabla v(\boldsymbol{x},t)d\Omega=0
			\end{align*}
			Since $\Omega_C$ is strictly embedded in $\Omega$, the partial differential equation in $(BVP_P)$ also holds for $\Omega_C$. Furthermore, as the cell center $\boldsymbol{x}_c$ is inside $\Omega_C$, it can be concluded that 
			$$\int_\Omega\Phi(\boldsymbol{x},t)\delta(\boldsymbol{x}-\boldsymbol{x}_c)v(\boldsymbol{x},t)d\Omega = \int_{\Omega_C}\Phi(\boldsymbol{x},t)\delta(\boldsymbol{x}-\boldsymbol{x}_c)v(\boldsymbol{x},t)d\Omega = \Phi(\boldsymbol{x}_c,t)v(\boldsymbol{x}_c,t).$$
			Then, the equation is rephrased as
			\begin{align*}
				&-\int_{\partial\Omega_C}\phi(\boldsymbol{x},t)v(\boldsymbol{x},t)d\Gamma + \int_\Omega\Phi(\boldsymbol{x},t)\delta(\boldsymbol{x}-\boldsymbol{x}_c)v(\boldsymbol{x},t)d\Omega-\int_{\Omega_C}D\Delta u_P(\boldsymbol{x},t)v(\boldsymbol{x},t)d\Omega\\
				&+\int_{\Omega_C}\Phi(\boldsymbol{x},t)\delta(\boldsymbol{x},t)v(\boldsymbol{x},t)d\Omega-\int_{{\Omega_C}}D\nabla u_P(\boldsymbol{x},t)\nabla v(\boldsymbol{x},t)d\Omega=0 \\
				\Rightarrow &-\int_{\partial\Omega_C}\phi(\boldsymbol{x},t)v(\boldsymbol{x},t)d\Gamma -\int_{\Omega_C}D\nabla\cdot(\nabla u_P(\boldsymbol{x},t) v(\boldsymbol{x},t))- D\nabla u_P(\boldsymbol{x},t)\nabla v(\boldsymbol{x},t) d\Omega\\
				&-\int_{{\Omega_C}}D\nabla u_P(\boldsymbol{x},t)\nabla v(\boldsymbol{x},t)d\Omega=0 \\ 
				\Rightarrow &\int_{\partial\Omega_C}(\phi(\boldsymbol{x},t) - D\nabla u_P(\boldsymbol{x},t)\cdot\boldsymbol{n})v(\boldsymbol{x},t) = 0.
			\end{align*}
			The last step is done by the Gaussian Theorem \citep{evans2010partial} and $\boldsymbol{n}$ is pointing towards the centre of $\Omega_C$. Note that $v(\boldsymbol{x},t)$ is a test function in $H^1$ space, hence, by DuBois-Raymond lemma \citep{van2005numerical}, we conclude that $$\phi(\boldsymbol{x},t) - D\nabla u_P(\boldsymbol{x},t)\cdot\boldsymbol{n} = 0, \mbox{ on $\partial\Omega_C$}.$$

			Hence, we proved that the solutions to both approaches are consistent in the domain $\Omega\backslash\bar{\Omega}_C$ if and only if the flux over the boundary of the hole is the same, that is, \[
			D\nabla u_P(\boldsymbol{x},t) = \phi(\boldsymbol{x},t), \mbox{ over $\partial\Omega_C$ and $t\geqslant0$.}
			\]
		\end{proof}
		
		\section{Properties of $\phi_{sum}$}
		\noindent
		Recall that in the setting of a single Dirac source at $\boldsymbol{x}_c$ with mass efflux rate $\Phi(t)$ in the infinitely extended space $\R^2$ and with initial condition equal to $p_0 P^D_{t_0}$, the flux density over the boundary of the hole, $\partial\Omega_C$, which is a circle of radius $R$ around $\boldsymbol{x}_c$, is given by
		\begin{equation}\label{eq:phi sum}
			\phi_{sum}(t) = \frac{p_0 R}{8\pi D(t+t_0)^2} \exp\left\{-\frac{R^2}{4D(t+t_0)}\right\} + \frac{\Phi}{2\pi R} \exp\left\{-\frac{R^2}{4Dt}\right\},
		\end{equation}
		when $\Phi$ is assumed constant in time. In this section we shall summarize various properties of $\phi_{sum}(t)$. The graphs of $\phi_{sum}(R,t)$ and that of $\phi_1(R,t)$, $\phi_2(R,t)$, i.e. the first and second term in \eqref{eq:phi sum}, respectively, are shown in Figure \ref{Fig_phi_12_sum} for $\phi(\boldsymbol{x},t) = 1$, $\Phi= \pi$, $t_0 = 4.0$ and $p_0$ determined by Equation (\ref{Eq_p0_t0}). That is, $\phi_{sum}(0)=\phi=1$.
		\begin{figure}[h!]
			\centering
			\includegraphics[width = 0.75\textwidth]{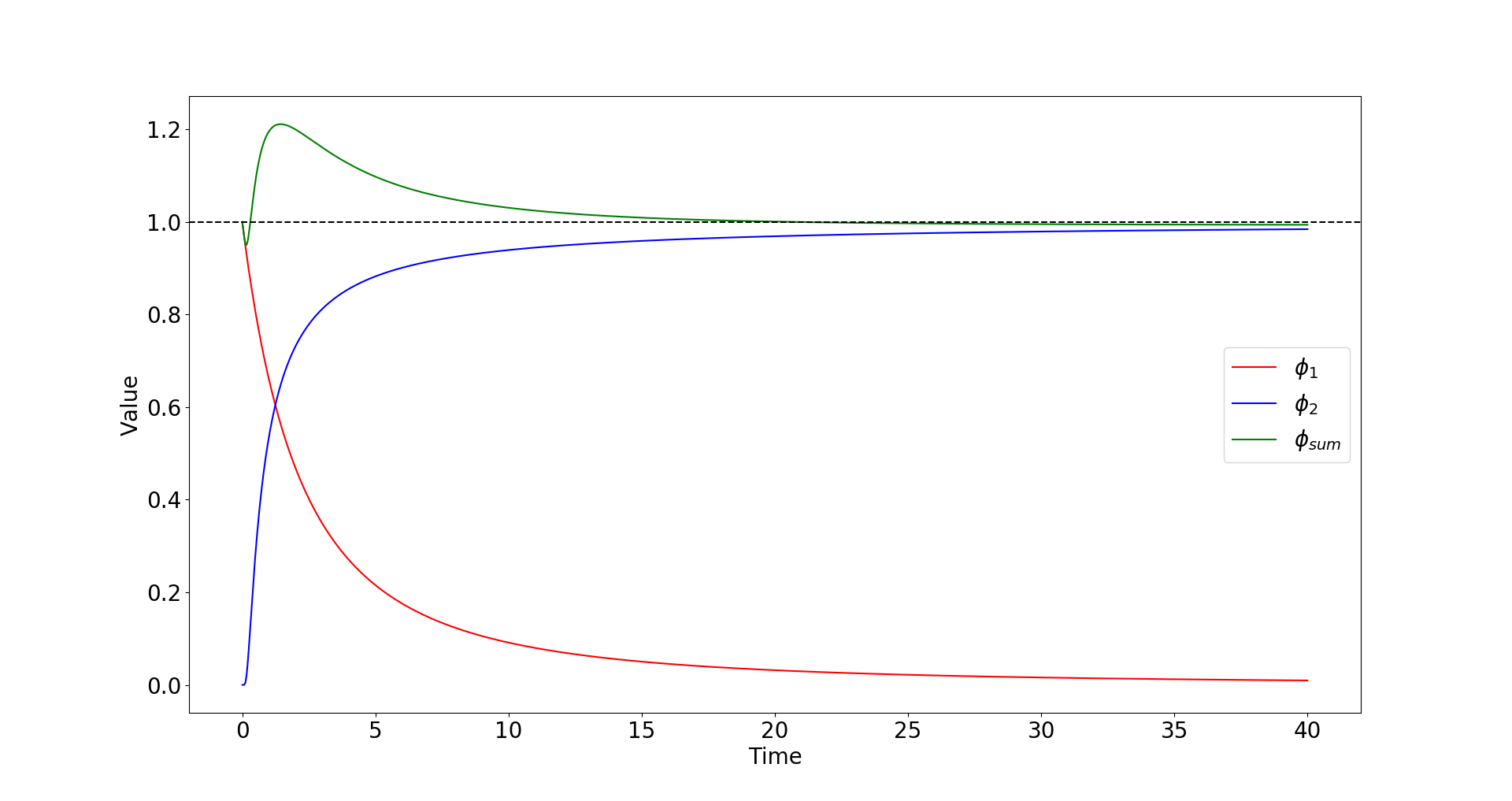}
			\caption{The graphs of $\phi_1$, $\phi_2$ and their summation, $\phi_{sum}$. Here, we use $t_0 = 4.0$, $\phi=1$ such that $\Phi=\pi$ and $p_0$ is computed by Equation (\ref{Eq_p0_t0}). Hence, $\phi_{sum}(0)\approx \phi$.}
			\label{Fig_phi_12_sum}
		\end{figure}
		It is immediately clear that
		\begin{equation}
			\phi_{sum}(0) = \frac{p_0 R}{8\pi D t_0},\qquad
			\lim_{t\to\infty} \phi_{sum}(t) = \frac{\Phi}{2\pi R}.
		\end{equation}
		Therefore, if one wants to abstain from any systematic deviation of flux -- hence between the solutions for spatial exclusion and point source approach in view of Proposition \ref{Prop_condition} -- then it is necessary to require
		\begin{equation}\label{eq:assymp condition on flux}
			\frac{\Phi}{2\pi R} = \phi,
		\end{equation}
		which condition we shall assume to be satisfied from now on. 
		\vskip 0.2cm
		
		Define
		\[
		\alpha := \frac{R^2}{4D},\qquad \beta:= \frac{p_0}{\pi R\phi},\qquad \gamma :=\beta\phi,
		\]
		such that
		\[
		\phi_{sum}(t) = \frac{p_0}{2\pi R} \alpha\,\frac{1}{(t+t_0)^2}\, e^{-\alpha/(t+t_0)}\ + \phi\,e^{-\alpha/t}\ = \ \mbox{$\frac{1}{2}$} \alpha\gamma\, \frac{1}{(t+t_0)^2}\, e^{-\alpha/(t+t_0)}\ + \phi\,e^{-\alpha/t}.
		\]
		A straightforward computation yields
		\begin{equation}\label{eq:deriv phi sum}
			\phi_{sum}'(t)\ =\ -\alpha\gamma\,\frac{t+t_0 -\frac{1}{2}\alpha}{(t+t_0)^4}\, e^{-\alpha/(t+t_0)}\ +\   \alpha\phi\,  \frac{1}{t^2}\, e^{-\alpha/t}.
		\end{equation}
		Define further
		\[ 
		g(t) := \frac{1}{t^2}e^{-\alpha/ t}\qquad\mbox{and}\qquad h(t) := \frac{t-\frac{1}{2}\alpha}{t^4} e^{-\alpha/ t}.
		\]
		Then, 
		\begin{equation}\label{eq:critical points phi-sum}
			\phi'_{sum}(t^*)=0\quad\mbox{{\it if and only} if}\quad g(t^*) = \beta \, h(t^*+t_0).
		\end{equation}
		
		One readily computes that
		\[
		g'(t) = -\frac{1}{t^4} e^{-\alpha/ t}(2t - \alpha)\quad\mbox{and}\quad h'(t) = -\frac{3}{t^6} e^{-\alpha/ t} \bigl( t^2 -\alpha t + \mbox{$\frac{1}{6}$} \alpha^2\bigr).
		\]
		Then, $g$ is a strictly positive function on $(0,\infty)$, with  $g(t)\to 0$ as $t\to\infty$ and $t\downarrow 0$. $g$ has a maximum value at $t=\frac{1}{2}\alpha$ and $g'(0)=0$. Similarly, $h(0)=0$ and $h(t)\to 0$ as $t\to\infty$. Moreover, $h(t)<0$ for $0<t<\frac{1}{2}\alpha$ and $h(t)>0$ for $t>\frac{1}{2}\alpha$. $h'(0)=0$ and there exists $0<t_-<t_+$ that solve $t^2 -\alpha t +\frac{1}{6}\alpha^2=0$. One has $t_\pm = \alpha(\frac{1}{2} \pm\frac{1}{3}\sqrt{3})$. Clearly, $h'(t)<0$ for $t>t_+$, so $h$ has a positive maximum value at $t_+>\frac{1}{2}\alpha$ and a negative minimum value at $t_-<\frac{1}{2}{\alpha}$.
		\vskip 0.2cm 
		
		\noindent We obtain:
		\begin{lemma}\label{lem:shape phi-sum}
			For $t_0>0$, $\phi_{sum}$ has either zero, one, two  or three critical points, where -- generically -- $\phi_{sum}$ changes sign. $\phi_{sum}(t)<\phi$ for $t$ sufficiently large. Case-by-case:
			\begin{enumerate}
				\item[({\it i})] If there are no critical points, then $\phi_{sum}$ has a minimum at $t=0$. Necessarily, $\phi_{sum}(0)<0$ and $\phi'_{sum}(0)>0$.
				\item[({\it ii})] If there is one critical point, then $\phi_{sum}$ has a minimum at this point $t=t_*>0$. Necessarily, $\phi_{sum}(t_*)<\phi$ and $\phi'_{sum}(0)<0$. $\phi_{sum}$ has a maximum at the boundary $t=0$. 
				\item[({\it iii})] If there are two critical point, then $\phi_{sum}$ has a minimum at some $t_*>0$ with $\phi(t_*)<\phi$ and a maximum at $0<t^*<t_*$. Necessarily, $\phi'_{sum}(0)>0$ and $\phi_{sum}$ has a minimum at the boundary point $t=0$.
				\item[({\it iv})]  If there are three critical point, then $\phi_{sum}$ has a minimum at some $t_{*,2}>0$ with $\phi(t_{*,2})<\phi$ and at $0<t_{*,1}<t_{*,2}$. Moreover, there is a maximum at $t^*$ with $t_{*,1}<t^*<t_{2,*}$.Necessarily, $\phi'_{sum}(0)<0$ and $\phi_{sum}$ has a maximum at the boundary point $t=0$.
			\end{enumerate}
		\end{lemma}
		
		\begin{proof}
			The claim on the number of critical points follows from the qualitative properties of $g(t)$ and $h(t)$ described above and the critical point characterisation in Equation \eqref{eq:critical points phi-sum}. The intersection of $g(t)$ and $\beta h(t+t_0)$ are transversal, generically. Therefore, $\phi'_{sum}(t)$ will change sign. One has
			\[
			\phi'_{sum}(t)\ =\ -\alpha\gamma h(t+t_0) + \alpha\phi g(t)\ = \ 
			\alpha\phi g(t)\left[ 1-\beta\, \frac{h(t+t_0)}{g(t)}\right].
			\]
			Moreover,
			\[
			\frac{h(t+t_0)}{g(t)}\ =\ \frac{t^2(t+t_0-\frac{1}{2}\alpha)}{(t+t_0)^4}\, \exp\left(\frac{\alpha t_0}{t(t+t_0} \right)\ \to 0\qquad\mbox{as}\ t\to\infty.
			\]
			Hence, $\phi'_{sum}(t)>0$ for $t$ sufficiently large. Since $\phi_{sum}(t)\to\phi$ as $t\to\infty$, one must have $\phi_{sum}(t)<\phi$ for $t$ large. In case $({\it i})$ the boundary point must then be a minimum with $\phi_{sum}(0)<\phi$. In the other cases there must exist a minimum at a largest $t$-value $t_*$ or $t_{*,2}>0$ with $\phi_{sum}(t_{*,2})<\phi$
		\end{proof}
		
		Note that Figure \ref{Fig_phi_12_sum} shows an example of case ({\it iv}). It can be seen by careful inspection of the graph that indeed, $\phi_{sum}(t)<\phi=1$ for $t$ sufficiently large.

		\section{Locations of multiple cells in Section \ref{Sec_number_of_holes}}
		\label{app_sec:location cells}
		\noindent
		In Section \ref{Sec_number_of_holes}, the impact of the influence from other cells was discussed. Figure \ref{Fig_mesh_number_of_cells} show the locations of the cells corresponding to the cases in Figure \ref{Fig_multi_holes}.
		\begin{figure}[h!]
			\centering
			\subfigure[One cell]{
				\includegraphics[width=0.475\textwidth]{mesh_hole.png}
				\label{fig_one_hole}}
			\subfigure[Two cell with smaller distance]{
				\includegraphics[width=0.48\textwidth]{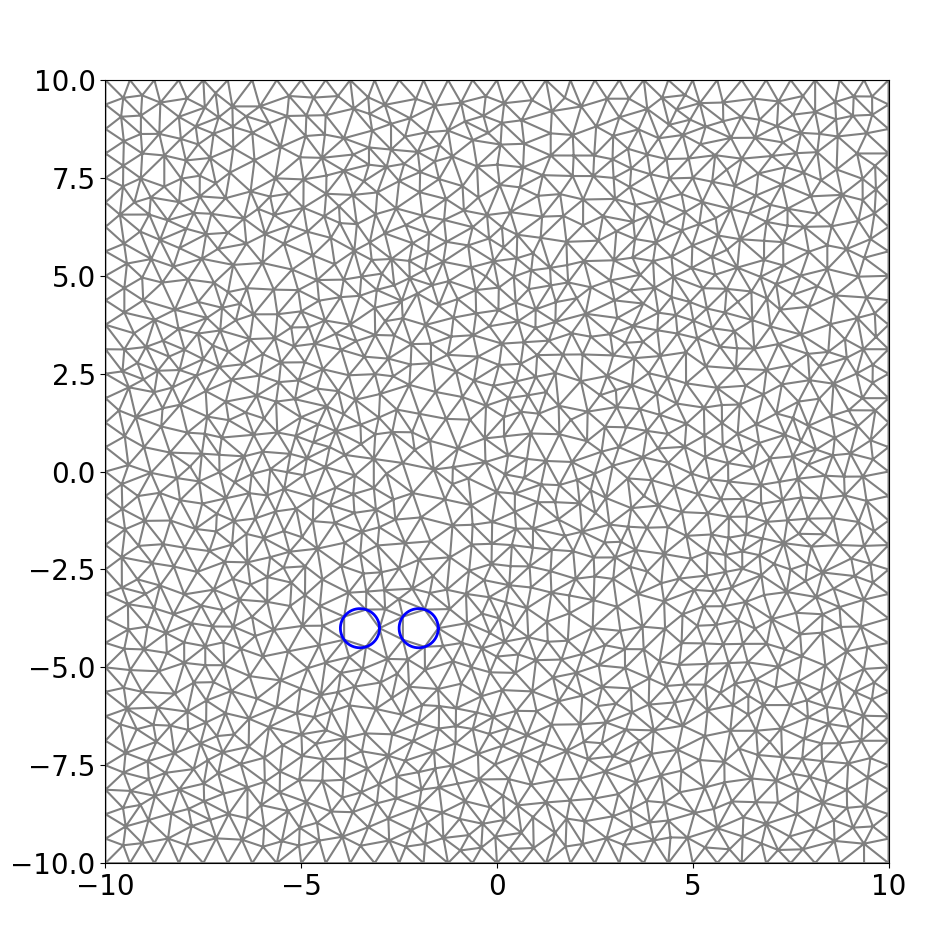}
				\label{fig_two_asys_holes}}
			\subfigure[Two cell with larger distance]{
				\includegraphics[width=0.48\textwidth]{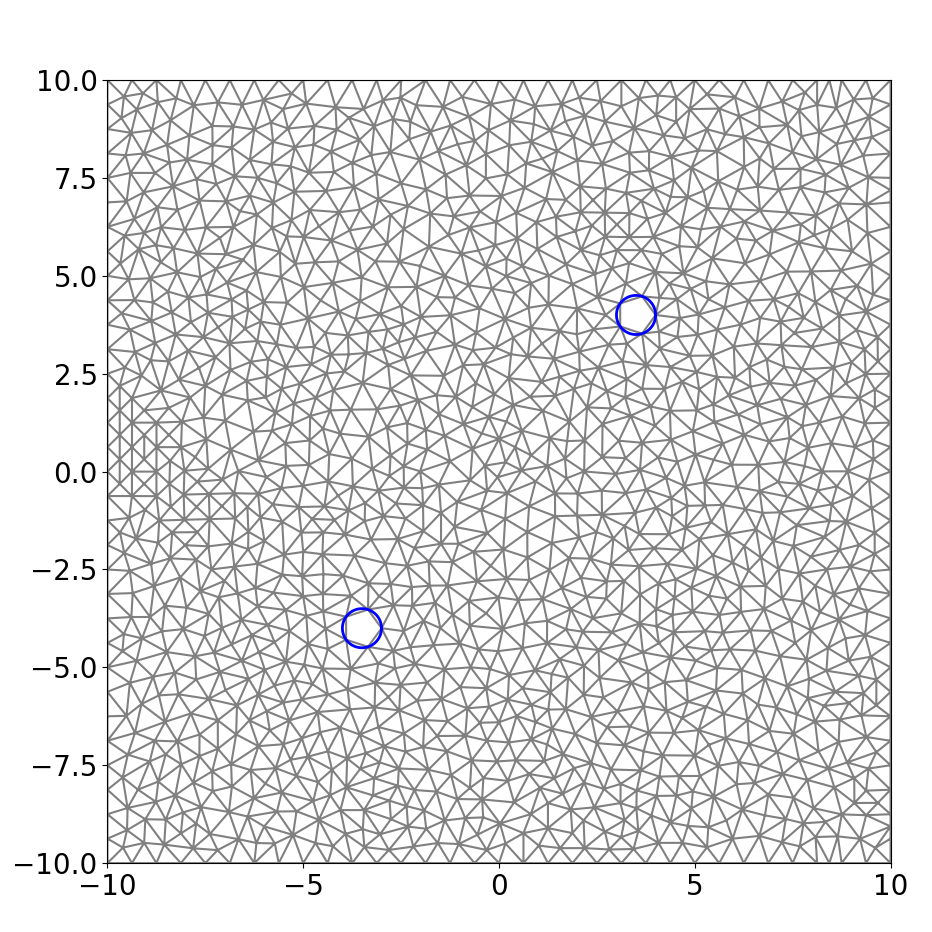}
				\label{fig_two_sys_holes}}
			\subfigure[Ten cells]{
				\includegraphics[width=0.48\textwidth]{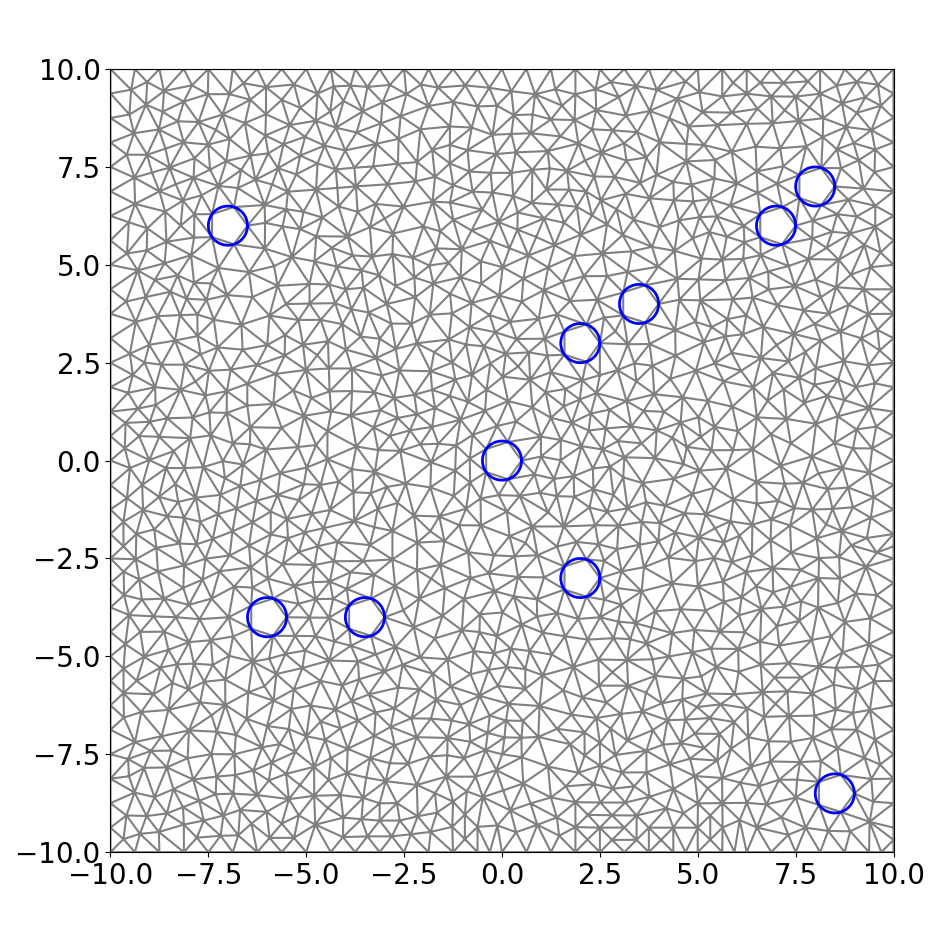}
				\label{fig_ten_holes}}
			\caption{The cell locations on the domain corresponding to Section \ref{Sec_number_of_holes}. For every case, we always have the cell, of which the center is $(-3.5, -4)$ and radius $\frac{1}{2}$. For this cell the measure $c^*(t)$ is computed in the comparison in Section \ref{Sec_number_of_holes}. To emphasize the locations, we show the schematics of the spatial exclusion approach with a coarse mesh.}
			\label{Fig_mesh_number_of_cells}
		\end{figure}
	\end{appendix}

\end{document}